\documentclass[12pt]{amsart}%
\usepackage[T1]{fontenc}
\usepackage[utf8]{inputenc}

\usepackage{color}
\usepackage{amsfonts}
\usepackage{amsmath}
\usepackage{amssymb}
\usepackage{graphicx}
\usepackage{color}
\usepackage{enumerate}

\numberwithin{equation}{section}

\newcommand{\CGO}{{\sc cgo}}
\newcommand{\F}{{\mathcal{F}}}
\newcommand{\bigO}{{\mathcal{O}}}
\newcommand{\Om}{\Omega}
\newcommand{\ra}{\rightarrow}
\newcommand{\dbar}{\overline{\partial}}
\newcommand{\R}{{\mathbb R}}
\newcommand{\C}{{\mathbb C}}
\newcommand{\Pc}{\mathcal{P}}
\def\oJ{\overline{J}}

\def\agt{}
\newcommand{\mltext}{}
\newcommand{\revised}{}
\newcommand{\revisedtwo}{}
\newcommand{\revisedthree}{}

\def\im{{\,\hbox{Im}}\,}
\def\re{{\text{Re}}}
\def\oal{\overline\alpha}
\def\onu{\overline\nu}
\def\oz{\overline{z}}
\def\ozp{\overline{z_1}}
\def\ozpp{\overline{z_2}}
\def\tw{\widehat\omega}
\def\vp{\varphi}
\def\tSz{\widetilde{\Sigma}}
\def\sone{\mathbb S^1}
\def\tP{\tilde\Phi}

\def\bs{\bigskip}
\def\ms{\medskip}
\def\lb{\linebreak}

\def\be{\begin{equation}}
\def\ee{\end{equation}}
\def\beqa{\begin{eqnarray}}
\def\eeqa{\end{eqnarray}}

\newtheorem{theorem}{Theorem}[section]

\newtheorem{proposition}{Proposition}[section]
\newtheorem{definition}[theorem]{Definition}

\begin{document}
\title[Propagation and recovery of singularities]{Propagation and recovery of singularities in the inverse conductivity problem}
\author[\quad Greenleaf, Lassas, Santacesaria, Siltanen and Uhlmann]{\qquad A. Greenleaf, M. Lassas, M. Santacesaria, S. Siltanen \qquad and G. Uhlmann}

\date{A.G.  partially supported by DMS-1362271 and a Simons Foundation Fellowship,
M.L. and S.S are partially supported by Academy of Finland, G.U. is partially supported by a FiDiPro professorship.}

\address{A. Greenleaf, Department of Mathematics, University of Rochester, Rochester, NY 14627}
 
\address{M. Lassas and S. Siltanen, Department of Mathematics, University of  Helsinki, FIN-00014}
  
\address{M. Santacesaria, Dipartimento di Matematica, Politecnico di Milano, 20133 Milano, Italy}
   
\address{G. Uhlmann, Department of Mathematics, University of Washington, Seattle, WA 98195}

\maketitle

\begin{abstract}

The ill-posedness of Calder\'on's inverse conductivity problem, 
responsible for the poor spatial resolution of Electrical Impedance Tomography (EIT),  has been an impetus for 
the development of hybrid imaging techniques, which compensate for this lack of  resolution by coupling with a 
second type of physical wave, typically modeled by a hyperbolic PDE. 
We show in 2D how, using EIT data alone, to use propagation of singularities for complex principal type PDE to 
efficiently detect interior jumps and other singularities of the conductivity.
Analysis of variants of the CGO solutions of Astala and P\"aiv\"arinta [\emph{Ann. Math.}, {\bf 163} (2006)]
allows us to  exploit a complex principal 
type geometry underlying the problem  
and show that the leading term in a Born series is an invertible 
nonlinear generalized Radon transform of the conductivity. 
The wave  front set of all higher-order terms
can be characterized, and, under a prior, some 
refined descriptions are possible.
We present numerics to show that this approach  is effective for detecting 
inclusions  within inclusions.

\end{abstract}

\newpage
\tableofcontents
\newpage

\section{Introduction}\label{Introduction}

\noindent
Electrical impedance tomography (EIT) aims to reconstruct the electric conductivity, $\sigma$, inside a body from active current and voltage 
measurements at the boundary. In many important applications of EIT, such as medical imaging and geophysical prospecting, the primary 
interest is in detecting the location of interfaces between regions of {\agt inhomogeneous but relatively} smooth conductivity. 
{\agt For example, the conductivity of bone is much lower than that of either skin or brain tissue, 
so there are jumps in conductivity of opposite signs as one transverses the skull.}

In this paper we present a new approach  in two dimensions to  determining  the singularities of a conductivity from EIT data. 
Analyzing the  {complex geometrical optics} (CGO) solutions, originally introduced 
by Sylvester and Uhlmann  \cite{SylUhl1987} and in the form required here by
Astala and P{\"a}iv{\"a}rinta \cite{Astala2006} and Huhtanen and Per{\"a}m{\"a}ki \cite{Huhtanen2012},
we transform the boundary values of the CGO solutions, which are determined 
by the Dirichlet-to-Neumann map \cite{Astala2006a},
in such a way as to extract the leading singularities of the conductivity, $\sigma$.

We show that  the leading term of a Born series 
derived from the boundary data is a nonlinear Radon transform of  $\sigma$  
and allows for good reconstruction of the singularities of $\sigma$,  with  the  higher order terms  representing multiple scattering.
Although one cannot escape the exponential ill-posedness inherent in  EIT, the well-posedness of Radon inversion 
results in  a robust method for detecting the leading singularities of $\sigma$.
In particular, one is able to detect inclusions within inclusions  (i.e., nested inclusions) within an unknown inhomogeneous background conductivity;
this has been a challenge for other EIT methods.  
This property is crucial for one of the main applications motivating this study, namely using EIT for classifying strokes 
into ischemic (an embolism preventing blood flow to part of the brain) and hemorrhagic (bleeding in the brain); see \cite{Holder1992B,Holder1992,Malone2014}.

Our algorithm consists of two steps, the first of which is the reconstruction of the boundary values of the CGO solutions, 
and this is known to be exponentially ill-posed, i.e., satisfy only logarithmic stability estimates \cite{Knudsen2009}. 
The second step begins with a separation of variables and partial Fourier transform in the radial component of the spectral variable.
Thus, the instability of our algorithm arises from the exponential instability of the reconstruction of the CGO solutions from the Dirichlet-to-Neumann map, 
the instability arising from low pass filtering in Fourier inversion
(similar to those of regularization methods used for CT and other linear inverse problems),
and (presumably) the multiple scattering terms in the Born series we work with, which we only control rigorously for low orders and under some priors.
Nevertheless,  based on both the microlocal analysis and numerical simulations we present, the method appears
to allow for robust detection  of singularities of $\sigma$, in particular the location and signs of jumps.
See Sec. \ref{subsec illposedness} for further discussion of  the ill-posedness issues raised by this method.

EIT can be  modeled mathematically using the inverse conductivity problem of Calder\'on \cite{Calder'on1980}. 
Consider a bounded, simply connected domain $\Omega\subset\R^n$ with smooth boundary 
and a scalar conductivity coefficient $\sigma\in L^\infty(\Omega)$ satisfying $\sigma(x)\geq c>0$ 
almost everywhere. Applying a voltage distribution $f$ at the boundary leads to the elliptic boundary-value problem
\begin{equation}\label{conductivityeq}
 \nabla\cdot\sigma\nabla u = 0 \quad \mbox{ in }\Omega, \qquad  u|_{\partial\Omega}=f.
\end{equation}
Infinite-precision boundary measurements are then modeled by the Dirichlet-to-Neumann map 
\begin{equation}\label{DNmap}
 \Lambda_\sigma: f\mapsto \sigma\frac{\partial u}{\partial \vec{n}}\Big|_{\partial\Omega},
\end{equation}
where $\vec{n}$ is the outward normal vector of $\partial\Omega$.

\begin{figure}[t!]
\begin{picture}(300,200)
\put(-80,0){\includegraphics[height=6cm]{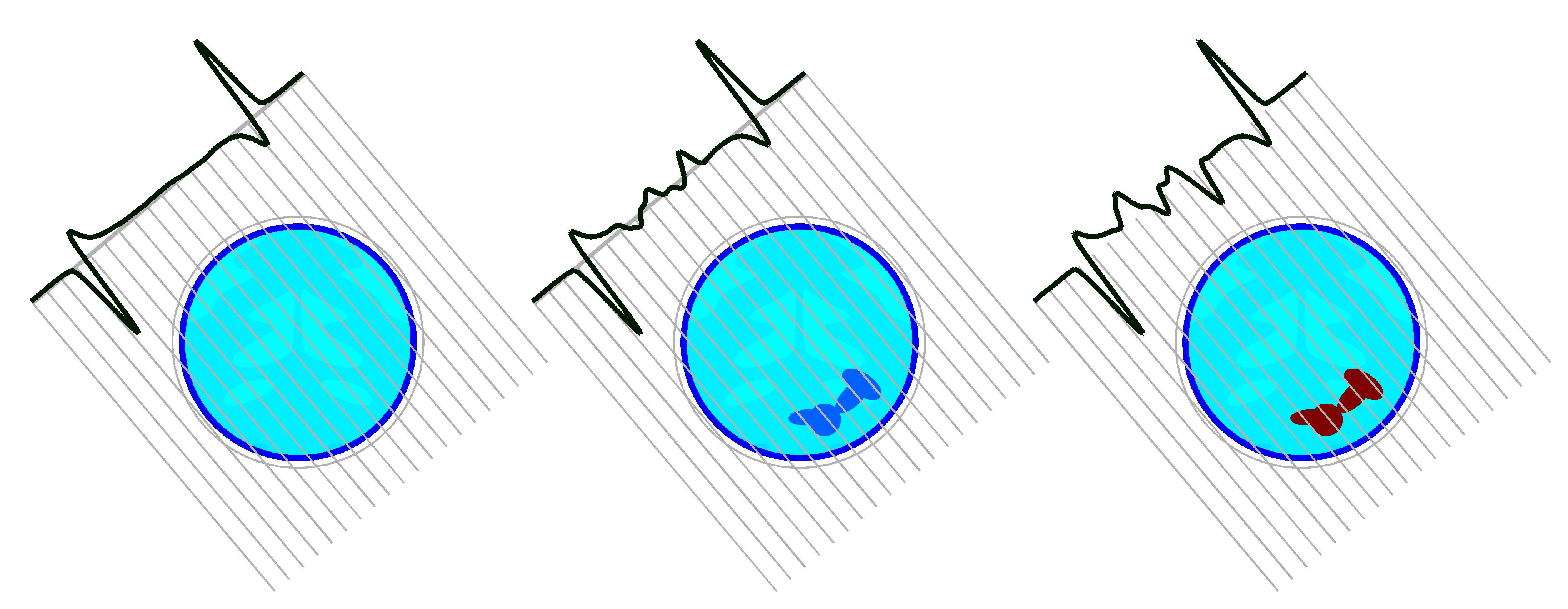}}
\color{red}
\put(-50,183){\line(0,-1){73}}
\put(-53,183){no jump}
\put(96,183){\line(0,-1){74}}
\put(88,183){jump down}
\put(233,183){\line(0,-1){65}}
\put(225,183){jump up}
\end{picture}
\caption{\label{Fig:strokephantoms2} The method provides information about inclusions within inclusions in an unknown inhomogeneous 
background. Jump singularities in the conductivity show up in the function values much like in parallel-beam X-ray tomography: recording 
integrals along parallel lines over the coefficient function. Figures illustrate this using stroke-like  computational phantoms. {Left:} Intact 
brain. Dark blue ring, with low conductivity, models the skull.  {Middle:} Ischemic stroke, or blood clot preventing blood flow to 
the dark blue area. The conductivity in the affected area is less  than that of the background. {Right:} Hemorrhagic stroke, or bleeding in the brain. 
The conductivity in the affected area is greater than the background. 
The function shown is $T^{a,+}\mu(t/2,e^{i\varphi})-T^{a,-}\mu(t/2,e^{i\varphi})$, and $\varphi$ indicates a direction perpendicular to the virtual ``X-rays.''}
\end{figure}
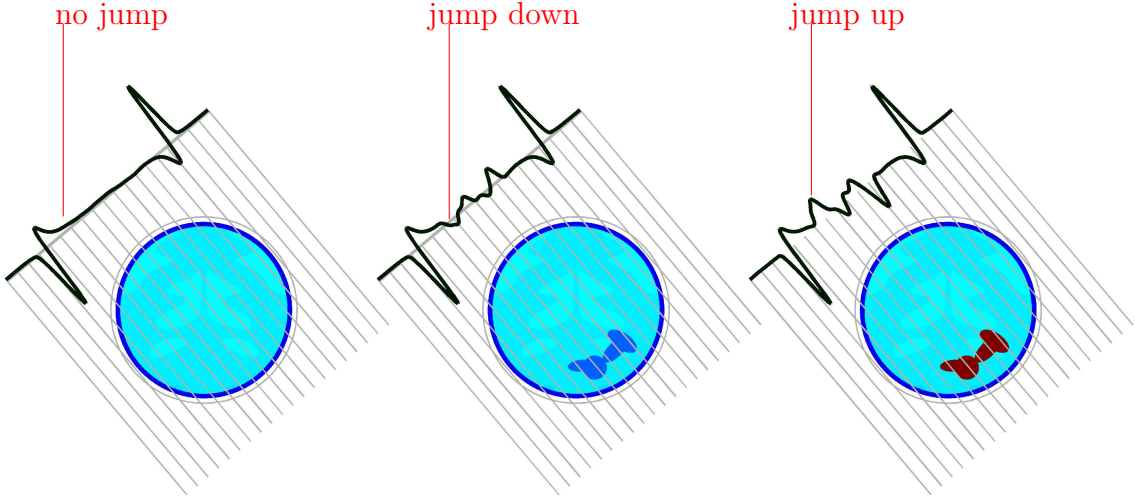

Astala and P{\"a}iv{\"a}rinta \cite{Astala2006a} transformed the construction of the CGO solutions  in dimension two 
was  by reducing the conductivity equation  to a Beltrami equation. 
{\agt Identify $\R^2$ with $\C$ by setting $z=x_1+ix_2$ and}
define a Beltrami coefficient, 
$$\mu(z) = (1-\sigma(z))/(1+\sigma(z)).$$
Since $c_1\le \sigma(z)\le c_2$, we have $|\mu(z)|\le 1-\epsilon$ for some $\epsilon>0$. 
Further, if we assume $\sigma\equiv 1$ outside some $\Omega_0\subset\subset\Omega$, then $\mbox{supp}(\mu)\subset\overline{\Omega_0}$.
Now consider the unique solution of
\begin{equation}\label{Beltrami_intro}
\bar \partial_z f_{\pm}(z,k) =\pm \mu(z) \overline{\partial_z f_{\pm}(z,k)}; \qquad  e^{-ikz} f_\pm(z,k) = 1+ \omega^{\pm}(z,k),
\end{equation}
where $ikz=ik(x_1+ix_2)$ and $\omega^{\pm}(z,k)=\mathcal{O}(1/|z|)$ as $|z|\ra\infty$. 
Here $z$ is considered as a spatial variable and $k\in\C$ as a spectral parameter. We note that $u=\mbox{Re}f_+$ satisfies (\ref{conductivityeq}), 
{\revised and denote $\omega^\pm$ by $\omega_\mu^\pm$ when  emphasizing dependence on the Beltrami coefficient $\mu$.}
{Recently, this technique has been generalised also for conductivities that are not in $L^\infty(\Omega)$ but only exponentially
integrable \cite{ALPapde}.

\medskip

{\revised The two crucial ideas {\revisedtwo of the current work are:}

(i) to analyze the scattering series we use the modified construction of Beltrami-CGO solutions by Huhtanen and Per{\"a}m{\"a}ki  \cite{Huhtanen2012}, 
which only  involves  exponentials 
of modulus 1 and where the solutions are constructed as a limit of an iteration of linear operations.
This differs from the original construction by Astala and P{\"a}iv{\"a}rinta \cite{Astala2006a}, 
where the construction of the exponentially growing solutions is based on the Fredholm theorem; and

(ii) to transform the CGO solutions by introducing polar coordinates  in the spectral parameter $k$, 
followed by a partial Fourier transform in the radial direction. 

\smallskip

These ideas are used as follows:
Formally one can view the Beltrami equation (\ref{Beltrami_intro}) as a scattering equation, where $\mu$ is  
considered as a compactly supported scatterer and the ``incident field'' is the constant function 1. 
Using (i), we write the CGO solutions $\omega^{\pm}$ as a ``scattering series'', 
\begin{equation}\label{scatteringseries}
 \omega^{\pm}(z,k)  \sim \sum_{n=1}^{\infty} \omega_n^{\pm}(z,k),
\end{equation}
considered as {\revisedtwo a formal power series} (cf.\ Theorem \ref{thm:Frechet}) 


Using (ii), we decompose $k = \tau e^{i \varphi}$ and then,
for each $n$,  form the
partial Fourier transform of the $n$-th order scattering term from (\ref{scatteringseries})   in $\tau$,
denoting these by
\begin{equation}\label{1DFourierTrick}
 \widehat \omega_n^\pm(z,t,e^{i \varphi}) := \F_{\tau \to t} \big(\omega_n^\pm(z,\tau e^{i \varphi})\big).
\end{equation}

{\revised As is shown in Sec. \ref{Weighted averages and the Radon transform},
singularities in $\sigma$ can be detected from averaged versions of $\widehat\omega^\pm_1$, 
denoted $ \widehat \omega^{a,\pm}_1$,
formed by taking a complex contour integral of $ \widehat \omega_1^\pm(z,t,e^{i \varphi})$ over $z\in\partial\Omega$}; 
see Fig. \ref{Fig:strokephantoms2}.

{\revised Recall that the traces of CGO solutions $\omega^{\pm}$
can be recovered perfectly from infinite-precision data $\Lambda_\sigma$ \cite{Astala2006a,Astala2006}.
When $\sigma$ is close to 1, the single-scattering term  $\omega^{\pm}_1$ is close to
 $\omega^{\pm}.$}
Fig. \ref{Fig:strokephantoms2} suggests that what we can recover resembles parallel-beam X-ray projection data of the singularities of $\sigma$. 
Indeed, we derive approximate reconstruction formulae for $\sigma$ analogous to the classical filtered back-projection method of X-ray tomography.

The wave front sets of  all of the terms $\widehat\omega_n^{\pm}$    
are analyzed in Thm. \ref{thm WF}.
More detailed descriptions of the  initial three terms, $\widehat \omega_1^{\pm},\, \widehat \omega_2^{\pm}$ and $\widehat \omega_3^{\pm}$,
identifying the latter two as sums of paired Lagrangian distributions under a prior on the conductivity, 
are given in Secs. \ref{subsec microlocal}, \ref{Analysis of the first-order term} and \ref{sec higher}, resp.

Let  $X=\{\mu\in L^\infty(\Omega);\ \hbox{ess supp}(\mu)\subset \Omega_0,\ \|\mu\|_{L^\infty(\Omega)}\le 1-\epsilon\},$
recalling that $\Omega_0\subset \subset \Omega$. The expansion   in (\ref{scatteringseries}) comes from the following:

\begin{theorem}\label{thm:Frechet}
{\revised  For $k\in\mathbb C$, define  nonlinear operators} $W^\pm(\cdot\,;k):X\to  L^2(\Om)$ by  
$$
W^\pm(\mu;k)(z):=\omega _\mu^\pm(z,k). 
$$
{\revised Then, at any $\mu_0\in X$,  $W^\pm(\cdot\,;k)$ has Fr\'echet derivatives in $\mu$  of all orders $n\in \mathbb N$, 
denoted by $D^{n}W_k|_{\mu_0}$, }
and the multiple scattering terms in (\ref{scatteringseries}) are given by
\begin{equation}\label{eqn Aone}
\omega_n^\pm=[D^{n}W_k^\pm(\mu,\mu,\dots,\mu)]\big|_{\mu=0}.
\end{equation}
The $n$-th order scattering operators, 
\begin{equation}\label{eqn Atwo}
T^\pm_n:\mu\mapsto \widehat \omega_n^\pm:=\F_{\tau\to t}(\omega _n^\pm(z,\tau e^{i\varphi})),\quad z\in \partial \Omega,\ t\in \R,
 \  e^{i\varphi}\in  \mathbb S^1,
\end{equation}
{\revised which are homogeneous forms of degree $n$ in $\mu$, 
have associated multilinear operators whose  Schwartz kernels $K_n$ have
wave front relations which can be explicitly computed. 
See formulas (\ref{can_rel_T0_a}) and (\ref{Cjs}) for the case $n=1$ and (\ref{lagrangians LnJ}) for $n\ge 2$. 
$K_1$ is a Fourier integral distribution; 
$K_2$ is a generalized Fourier integral (or paired Lagrangian) distribution; and for $n\ge 3$, 
$K_n$ has wave front set contained in a union of a family of $2^{n-1}$ pairwise cleanly intersecting Lagrangians.
}
\end{theorem}

Singularity propagation for the first-order scattering $\widehat \omega_1^{\pm}$ is 
described by a Radon-type transform and a filtered back-projection formula.

\begin{theorem}\label{thm:FBP}
Define averaged operators {\mltext $T_n^{a,\pm}$ for $n\in\mathbb N$ and $T^{a,\pm}$ by
the complex contour integral
\footnote{{\revised Throughout, $d\mathbf{z}$ will denote the element of complex contour integration along a curve, 
while $d^1\mathbf{x}$  is arc length measure. $d^2z$ denotes two-dimensional Lebesgue measure in $\mathbb C$.}}
,
\begin{eqnarray}\label{averaging}
 T_n^{a,\pm} \mu (t,e^{i \varphi}) &=& \frac{1}{2 \pi i} \int_{\partial\Omega}\widehat \omega_n^\pm(z,t,e^{i \varphi})d\mathbf{z},\\
 \label{averaging total}
  T^{a,\pm} \mu (t,e^{i \varphi}) &=& \frac{1}{2 \pi i} \int_{\partial\Omega}\widehat \omega^\pm(z,t,e^{i \varphi})d\mathbf{z},
\end{eqnarray}
with $\omega_n^\pm$ defined via  formulas (\ref{eqn Aone})--(\ref{eqn Atwo}) and 
$\omega^\pm$ defined via  (\ref{Beltrami_intro}).}
Then we have 
\begin{equation}\label{FBP}
(- \Delta)^{-1/2} (T_1^{a,\pm})^{\ast}T_1^{a,\pm} \mu = \mu.
\end{equation}
\end{theorem}

\noindent
Theorem \ref{thm:FBP} suggests an approximate reconstruction algorithm:
{\revised
\begin{itemize}
\item Given $\Lambda_\sigma$, follow \cite[Section 4.1]{Astala2011} to compute both $\omega^+(z,k)$ and $\omega^-(z,k)$ 
for $z\in\partial\Omega$ by solving the boundary integral equation derived in \cite{Astala2006}. 

\item {\agt Introduce polar coordinates in the spectral variable $k$ and compute the partial} Fourier transform, $\widehat \omega^\pm(z,t,e^{i \varphi})$.

\item Using the operator $T^{a,\pm}$ defined in (\ref{averaging total}), 
we compute $\widetilde{\mu}^+:= \Delta^{-1/2} (T_1^{a,+})^{\ast}T^{a,+} \mu$ and  $\widetilde{\mu}^-:= \Delta^{-1/2} (T_1^{a,-})^{\ast}T^{a,-} \mu$. Note the difference to (\ref{FBP}).

\item {\revised Approximately} reconstruct by $\sigma=(\mu-1)/(\mu+1)\approx(\widetilde{\mu}-1)/(\widetilde{\mu}+1)$, where  $\widetilde{\mu}=(\widetilde{\mu}^+-\widetilde{\mu}^-)/2$. The approximation comes from using $T^{a,\pm} \mu$ instead of $T^{a,\pm}_1 \mu$ in the previous step.
\end{itemize} 
}
See the middle column of Fig. \ref{Fig:rec_hemclot} for an example. 

One can also use the 
identity $(T_1^{a,\pm})^{\ast} T_1^{a,\pm} = (-\Delta)^{1/2}$ to enhance the singularities in the reconstruction. 
This is analogous to $\Lambda$-tomography in the context of linear X-ray tomography \cite{Faridani1992,Faridani1997}. 
See the right-most column in Fig. \ref{Fig:rec_hemclot} for reconstructions using the operator $(T_1^{a,\pm})^{\ast} T^{a,\pm}$.

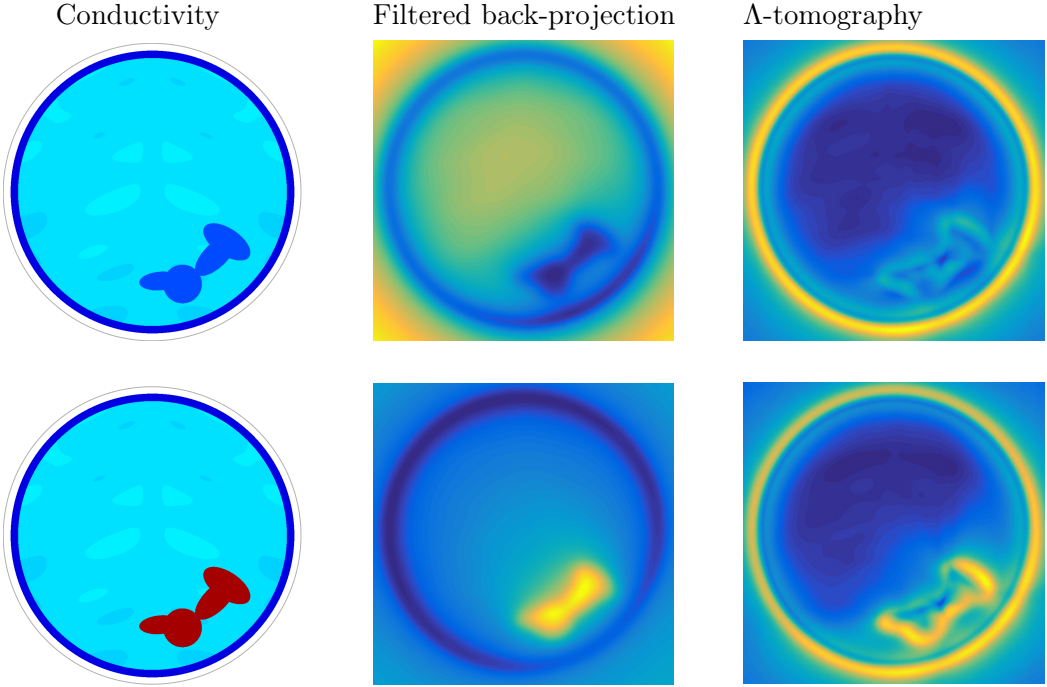
\begin{figure}
\begin{picture}(300,260)
\put(-80,0){\includegraphics[width=4cm]{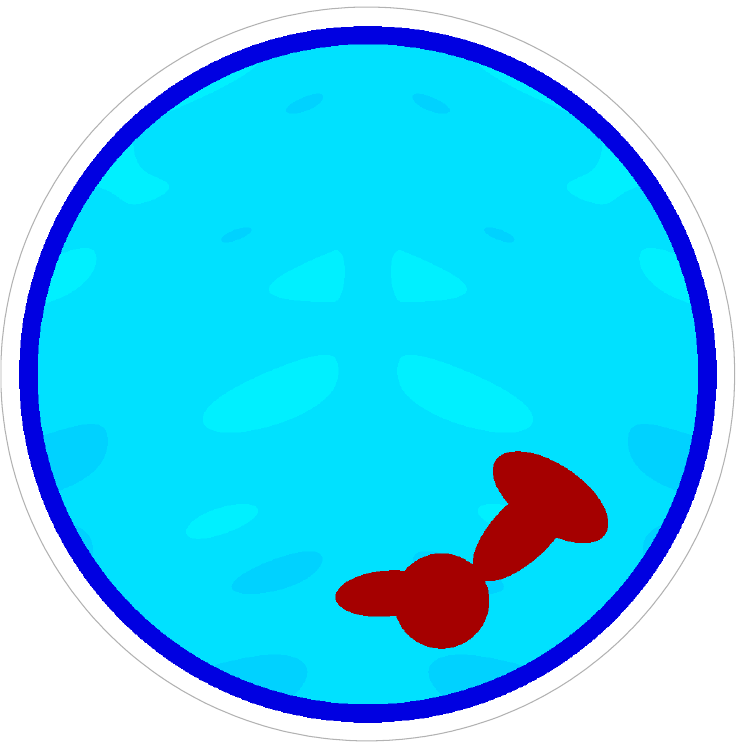}}
\put(60,0){\includegraphics[width=4cm]{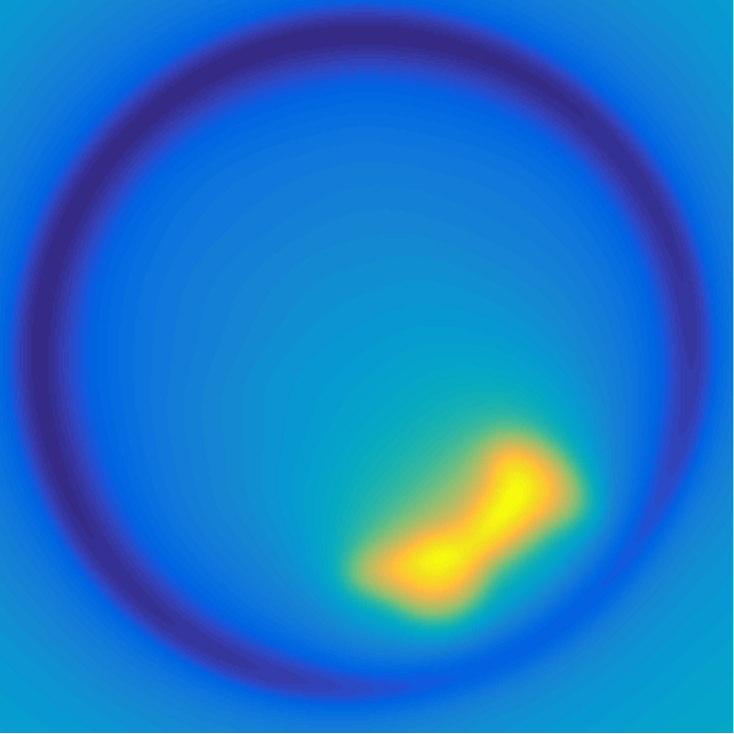}}
\put(200,0){\includegraphics[width=4cm]{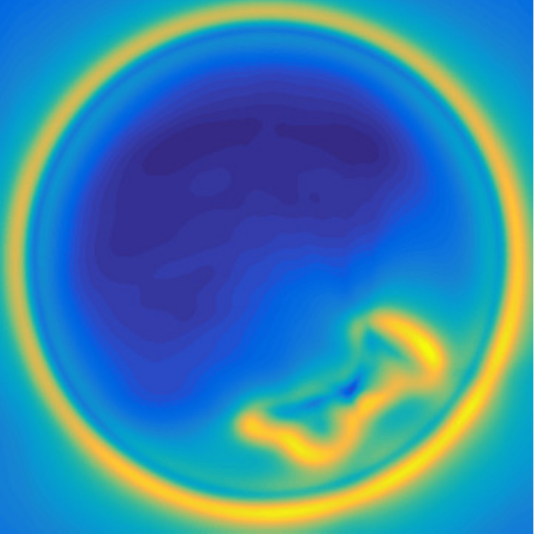}}
\put(-80,130){\includegraphics[width=4cm]{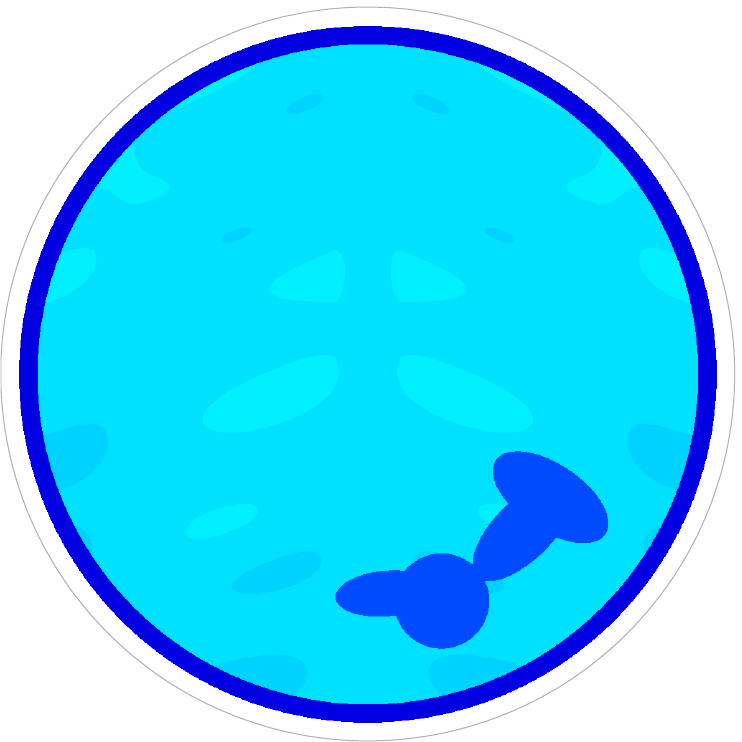}}
\put(60,130){\includegraphics[width=4cm]{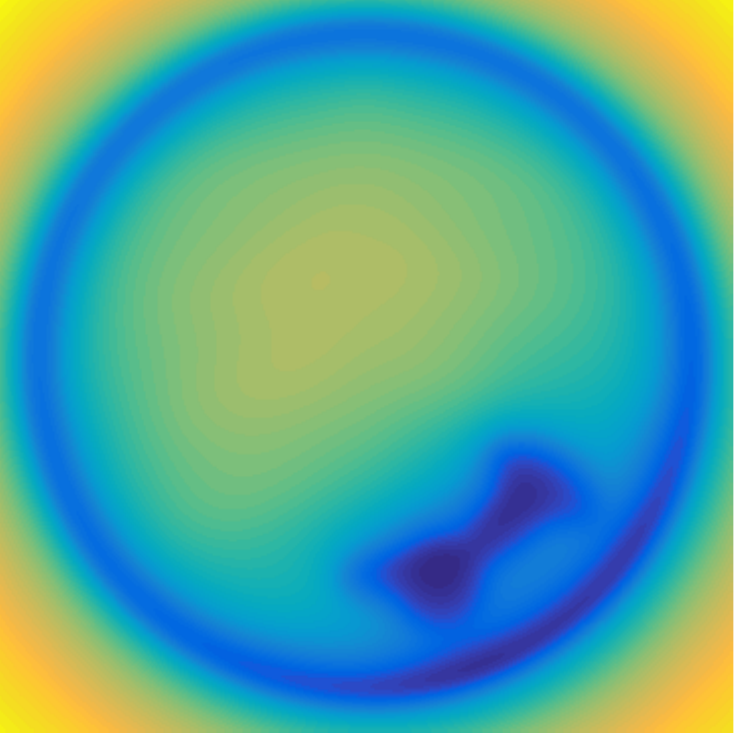}}
\put(200,130){\includegraphics[width=4cm]{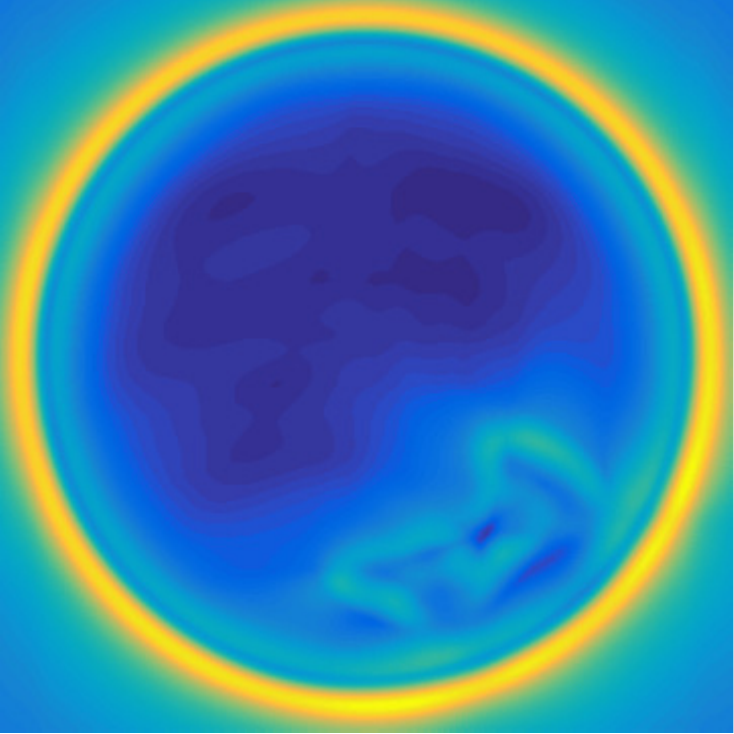}}
\put(-60,250){\small Conductivity}
\put(60,250){\small Filtered back-projection}
\put(200,250){\small $\Lambda$-tomography}
\end{picture}
\caption{\label{Fig:rec_hemclot}Reconstructions, of computational phantoms   
modeling ischemic strokes (top row) and hemorrhagic  strokes (bottom row), from very high precision simulated EIT data. 
The results are promising for {\agt portable, cost-effective  classification of strokes without use of ionizing radiation.}}
\end{figure}

\medskip

Our general theorem on singularity propagation is quite technical, 
and so we illustrate it here using a simple example, 
postponing the precise statement and proof to Section \ref{Analysis of the higher order terms} below.

{\revisedthree
Assume that the conductivity is of the form $\sigma(z)=\sigma(|z|)$ and  smooth except for a jump across the circle $|z|=\rho$. 
 One can  describe the singular supports of the $\widehat \omega^\pm_{n}(z,t,e^{i\varphi})$. 
For $m\in\mathbb N$, define hypersurfaces
\begin{eqnarray*}
& &\hspace{-5mm}\Pi_{m}=\{ (z,t,e^{i\varphi})\in \mathbb C\times \R\times \mathbb S^1:\ t=2\rho m\}.
\end{eqnarray*}
Using the analysis later in the paper, one can see that
\begin{eqnarray*}
\hbox{\big(sing supp}(\widehat \omega^\pm_{n})\!& \cap &\! \{(z,t,e^{i\varphi});\ |z|\ge 1\}\big) \subset  \\
& &\bigcup\{\Pi_m: -n\le m\le n,\,\, m\equiv n\!\!\mod 2\}. 
\end{eqnarray*}
However, 
it turns out that, by a parity symmetry property described in Sec. \ref{sec evenodd} , 
subtracting $\widehat\omega^-$ from $\widehat\omega^+$   eliminates the even terms, $\widehat \omega^\pm_{2n}$, 
so that their singularities, including a strong one for $\widehat\omega^\pm_2$ at $t=0$, do not create artifacts in the imaging.
See Fig. \ref{fig:propsing2} for a diagram of singularity propagation in the case $\rho=0.2$.
}

\begin{figure}
\begin{picture}(300,370)
\put(-35,10){\includegraphics[width=14cm]{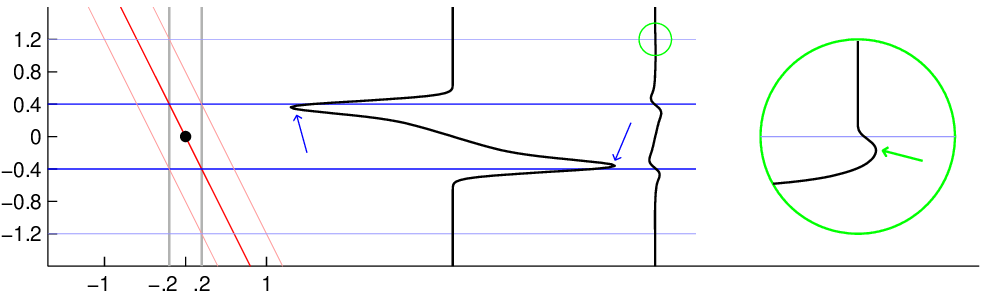}}
\put(125,135){\textcolor{black}{\small $T^{a,+}_1\!\!\mu(t,e^{i0})$}}
\put(210,135){\textcolor{black}{\small $T^{a,+}_3\!\!\mu(t,e^{i0})$}}

\put(100,5){$x_1$}
\put(-45,80){$t$}

\put(-30,210){\includegraphics[width=6cm]{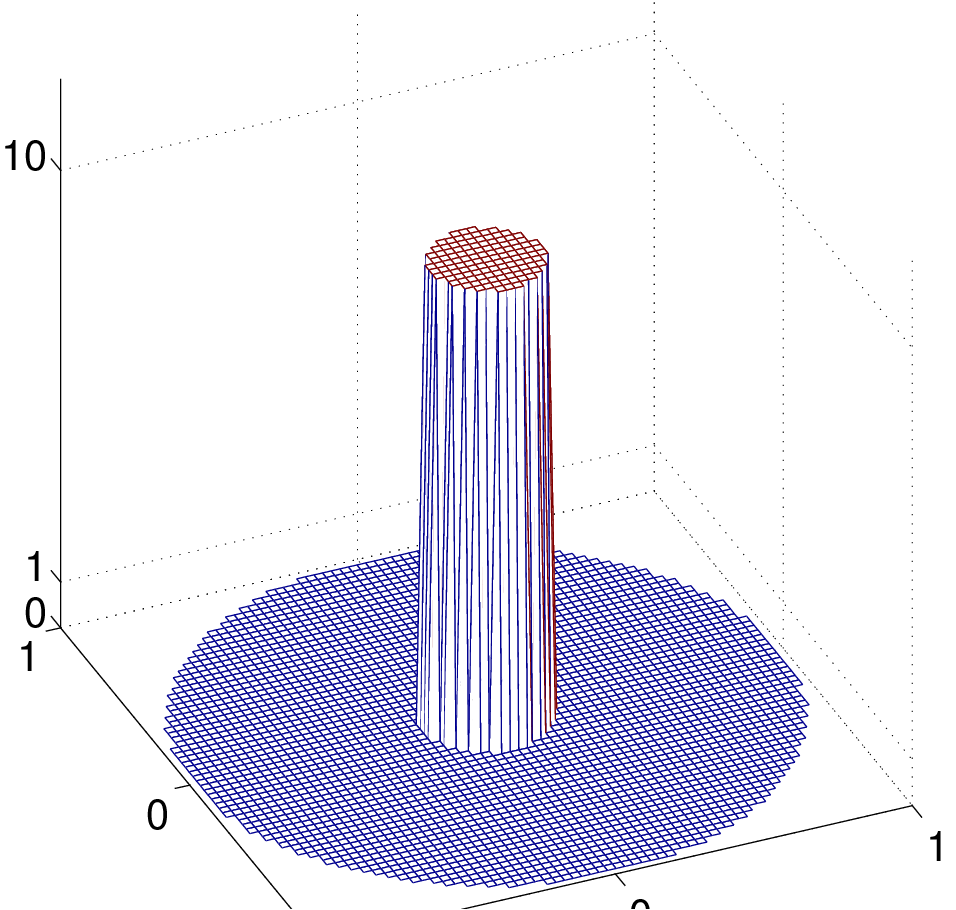}}

\put(170,200){\includegraphics[width=6.5cm]{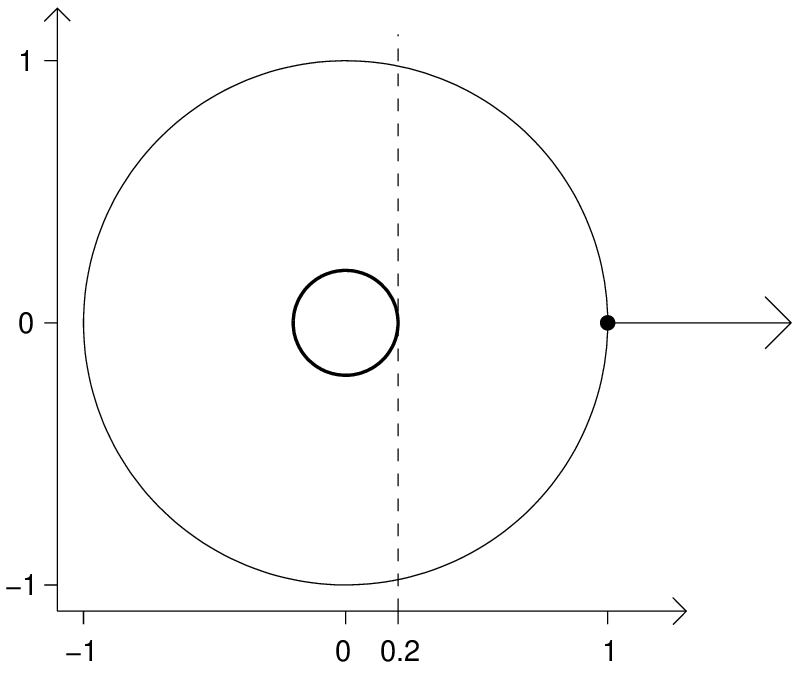}}
\put(168,354){$x_2$}
\put(340,200){$x_1$}
\put(315,285){\textcolor{black}{$z=1$}}
\put(325,265){\textcolor{black}{$k=\tau e^{i\varphi}$}}
\put(325,250){\textcolor{black}{$\varphi=0$}}

\put(-40,360){(a)}
\put(145,360){(b)}
\put(-40,140){(c)}
\put(270,110){(d)}
\end{picture}
\caption{\label{fig:propsing2}
(a) Three-dimensional plot of the conductivity having a jump along the circle with 
radius $\rho=0.2$ and center at the origin. (b) Unit disc and singular support of the conductivity in the $z=x_1+ix_2$ plane. 
(c) The term $T^{a,+}_1\!\!\mu(t,e^{i0})$ has peaks, indicated by blue arrows, at $t=\pm 2\rho$ corresponding to the locations of the main 
singularities in $\mu$, as expected by Theorem \ref{thm:FBP}.  
The higher-order term $T^{a,+}_3\!\!\mu(t,e^{i0})$,  
{\agt  smaller than $T^{a,+}_1\!\!\mu(t,e^{i0})$ in amplitude, exhibits singularities caused by reflections
at both  $t=\pm 2\rho$ and $t=\pm 6\rho$. (d)  The singularities of the term $T^{a,+}_3\!\!\mu(t,e^{i0})$  at $t=\pm 6\rho$ are very small. 
Shown is a zoom-in near $t=6\rho$, with amplitude increased by a factor of 70.}}
\end{figure}
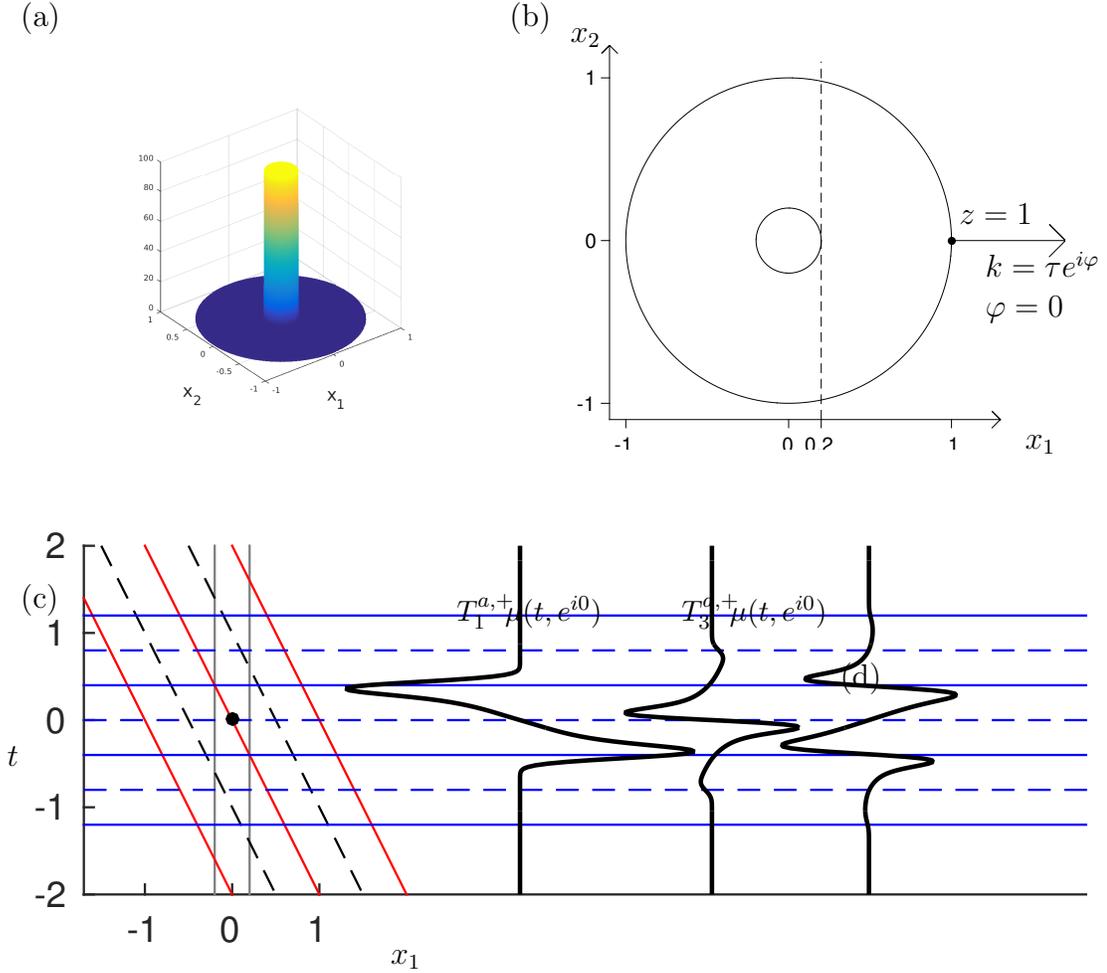

\subsection{Ill-posedness, noise and deconvolution}\label{subsec illposedness}

{\revised 
The exponential ill-posedness of the Calder\'on inverse problem (i.e., satisfying a stability 
estimate of only logarithmic type) has important consequences for EIT with realistic data.
Calder\'on inverse problems for elliptic equations were shown to be exponentially ill-posed  by Mandache \cite{Mandache}. 
Corresponding to this, in \cite[Lemma 2.4]{Knudsen2009} it was shown that when the Dirichlet-to-Neumann map is given with error $\epsilon$, then the 
 boundary values of the CGO solutions, or equivalently, $\omega(z,k)|_{z\in \partial \Omega}$, 
 can be found with  accuracy $\epsilon$ only  for the frequencies $|k|\le  R_\epsilon=c\log(\epsilon^{-1})$.

This exponential instability holds even under the prior that conductivities consist of inclusions \cite{Alessandrini2005} .
Furthermore, inclusions  need to have a minimum size to be detectable \cite{Alessandrini1988,Isaacson1986,Cheney1992}, 
and in order  to appear in reconstructions, the deeper inclusions  are  inside an object, the larger they must be  \cite{Nag2009,Alessandrini2017,Garde2017}. 
Finally, the resolution of reconstructions is limited by noisy data. 
It is natural to ask how  these limitations  are reflected in the  approach described in this paper.

Our  results show that the part of the conductivity's wave front set   in the direction specified by $\varphi$ 
is seen as specific singularities in the function $\widehat \omega^\pm(z,\,\cdot\,,e^{i \varphi})$, 
defined in (\ref{1DFourierTrick}). 
However, due to algebraic decay of the principal symbol of a Fourier integral operator, the amplitude of the measured singularity is bounded by $C\mbox{dist}(\partial\Omega,z)^{-1}$, 
making it harder to recover details deep inside the imaging domain.

Furthermore, with {\revisedtwo realistic and noisy data,} we can compute $\omega^\pm(z,k)$ only in a disc $|k|\le k_{max}$ 
with a {\revisedtwo measurement apparatus and noise-dependent} radius $k_{max}>0$; see \cite{Knudsen2009,Astala2011,Astala2014}. 
With smaller noise we can take a larger $k_{max}$, whereas large noise forces $k_{max}$ to be small. This makes it more difficult to locate singularities precisely. 

To better understand the difficulty, consider the 
truncated Fourier transform:
\begin{equation}\label{1DFourierTrickWindowed}
  \int_{-k_{max}}^{k_{max}} e^{-it\tau}\omega^\pm(z,\tau e^{i \varphi})\,d\tau = 
  \int_{-\infty}^{\infty} e^{-it\tau}\omega^\pm(z,\tau e^{i \varphi})\chi_{k_{max}}(\tau)\,d\tau,
\end{equation}
where $\chi_{k_{max}}(\tau)$ is the characteristic function of the interval $[-k_{max},k_{max}]$. Note that
\begin{eqnarray}\label{CharFour}
  \widehat{\chi}_{k_{max}}(t)
  = C\,\frac{\sin (k_{max}\,t)}{t}
\end{eqnarray} 
with a constant $C\in\R$. Noise forces us to replace the Fourier transform in (\ref{1DFourierTrick}) 
by a truncated integral such as (\ref{1DFourierTrickWindowed}). Therefore, we need to apply  one-dimensional 
deconvolution in $t$ to recover $\widehat \omega^\pm(z,\,\cdot\,,e^{i \varphi})$ approximately from  $\widehat \omega^\pm(z,\,\cdot\,,e^{i \varphi})\ast \widehat{\chi}_{k_{max}}$.  
Higher noise level means a smaller  $k_{max}$, which by (\ref{CharFour}) leads to a wider blurring kernel $\widehat{\chi}_{k_{max}}$; 
due to the Nyquist-Shannon sampling theorem, this results in a
more ill-posed  deconvolution problem and thus limits the imaging resolution.

In practice it is better to use a smooth windowing function instead of the characteristic function for  reducing unwanted oscillations (Gibbs phenomenon), and there are many suitable deconvolution algorithms in the literature \cite{Chen2001,Candes2013,Candes2014}.
}

\medskip


{\revised It is also natural to ask how does the  method introduced here compares to previous work in terms of detecting inclusions and jumps.}
\medskip

Many methods have been proposed for regularized edge detection from EIT data. 
Examples include the {\em enclosure method} \cite{Ikehata2000c,Ikehata2000a,Bruhl2000,Ikehata2004,Ide2007,Uhlmann2008b}, 
the {\em factorization method} \cite{Kirsch1998,Bruhl2000,Lechleiter2006,Lechleiter2008a}, the {\em monotonicity method} \cite{Harrach2013,Harrach2015}. 
These methods can only detect the outer boundary of an inclusion in conductivity, whereas the method described here, 
which exploits the propagation of singularities for complex principal type operators, can see nested jump curves. 
Also, the proposed method can deal with  {\revised  inclusions within inclusions, and with conductivities having both positive and 
negative jumps, even in unknown inhomogeneous smooth background}.

One can also attempt edge detection based on EIT algorithms originally designed for reconstructing the full conductivity distribution. 
There are two main approaches: sharpening blurred EIT 
images in data-driven post-processing \cite{Hamilton2014,Hamilton2015}, and applying sparsity-promoting inversion methods such as total variation regularization 
\cite{Dobson1994,Kaipio2000,Rondi2001,Chan2004,Chung2005,Tanushev2007,vandenDoel2006,Jin2012a,Garde2016,Zhou2015}. 
As of now, the former approach does not have rigorous 
analysis available. Some of the latter kind of approaches are theoretically capable of detecting nested inclusions; 
however, in variational regularization there is  typically an instability issue, 
where a large low-contrast inclusion may be represented by a smaller high-contrast feature in the reconstruction. 
{\revised Numerical evidence suggests that  method introduced here can accurately and robustly reconstruct  jumps, both in terms of location and sign.}

{\section{Complex principal type structure of CGO solutions}
\label{CPT structure}

We start by describing the microlocal geometry underlying the exponentially growing, 
or so-called \emph{complex geometrical optics (CGO)}, solutions to the conductivity equation  
 on $\R^d,\, d\ge 2$, 
\begin{equation}\label{cond eqn sec two}
\nabla \,\cdotp \sigma\nabla u(x)=0,\quad x\in \R^n,
\end{equation}
originating in \cite{SylUhl1987}.
For complex frequencies  $\zeta= \zeta_R+i\zeta_I\in\C^n$ with  $ \zeta\cdot\zeta=0$,
one can decompose $\zeta=\tau\eta$, 
with $\tau\in\R$ and $ \eta=\eta_R+i\eta_I,\, |\eta_R|=|\eta_I|=1,\, \eta_R\cdot\eta_I=0$.
Now consider  solutions to (\ref{cond eqn sec two}) of the form
$$
u(x):=e^{i \zeta\cdot  x}w(x,\tau)=e^{i\tau \eta \cdot x}w(x,\tau).
$$
Physically speaking, $\tau$ can be considered as a spatial
frequency, with the voltage on the  boundary $\partial \Omega$
oscillating at length scale $\tau^{-1}$.

The conductivity equation (\ref{cond eqn sec two}) becomes 
\begin{eqnarray*}
0&=&\frac 1\sigma \nabla \,\cdotp \sigma\nabla u(x)\\
&=&
\frac 1\sigma \nabla \,\cdotp \sigma\nabla ( e^{i\tau \eta\cdotp x}w(x,\tau))
\\
&=&
(\Delta+ (\frac 1\sigma \nabla  \sigma)\,\cdotp \nabla)( e^{i\tau \eta\cdotp x}w(x,\tau))
\\
&=&
\bigg(\Delta  w(x,\tau)+2i\tau \eta\,\cdotp \nabla w(x,\tau)
+
( \frac 1\sigma \nabla  \sigma)\,\cdotp (\nabla  +i\tau \eta) w(x,\tau) \bigg) e^{i\tau \eta\cdotp x}.
\end{eqnarray*}
Hence, we have 
\begin{eqnarray*}
\Delta  w(x,\tau)+2i\tau \eta\,\cdotp \nabla w(x,\tau)
+
( \frac 1\sigma \nabla  \sigma)\,\cdotp (\nabla  +i\tau \eta) w(x,\tau)=0.
\end{eqnarray*}
Taking the  partial Fourier transform $\widehat w$ in the $\tau$ variable and denoting the resulting dual variable by $t$, 
which can be thought of as a ``pseudo-time,'' one obtains
\begin{eqnarray*}
\Delta  \widehat w(x,t)-2\eta\frac \partial{\partial t} \,\cdotp \nabla \widehat w(x,t)
+
( \frac 1\sigma \nabla  \sigma)\,\cdotp (\nabla  -\eta\frac \partial{\partial t}  )\widehat  w(x,t)=0.
\end{eqnarray*}
The principal part of this equation is given by the operator
$$
\widetilde \square =\mathcal P_R+i\mathcal P_I=\Delta  -2\eta\frac \partial{\partial t} \,\cdotp \nabla
$$
where 
$$
\mathcal P_R=\Delta  -2\eta_R\frac \partial{\partial t} \,\cdotp \nabla \hbox{ and } 
\mathcal P_I=-2\eta_I\frac \partial{\partial t} \,\cdotp \nabla.
$$
With $\xi$   the  variable dual to $x$,
the full symbols of $\mathcal P_R$ and $\mathcal P_I$ are
$$
p_R(x,t,\xi,\tau)=-\xi^2+2\tau \eta_R \,\cdotp \xi,\quad p_I(x,t,\xi,\tau)=2\tau \eta_I \,\cdotp \xi,
$$
and these commute in the sense of Poisson brackets:
$\{p_R,p_I\}=0$. Furthermore,  on the
characteristic variety 
\begin{eqnarray*}
\Sigma&:=&\{(x,t,\xi,\tau)\in \R^{d+1}\times(\R^{d+1}\setminus\{0\});\ p_R(x,t,\xi,\tau)=0,\ p_I(x,t,\xi,\tau)=0\}\\
&=&\{(x,t,\xi,\tau)\in \R^{d+1}\times(\R^{d+1}\setminus\{0\}) ;\ 
|\xi|^2-2\tau \eta_R \,\cdotp \xi=0,\ 
2\tau \eta_I \,\cdotp \xi=0  \}\\
&=&\{(x,t,\xi,\tau)\in \R^2\times \R\times \R^2\times(\R\setminus\{0\}) ;\ 
  \xi= 2 \tau\eta_R\hbox{ or } \xi=0 \},
\end{eqnarray*}
the gradients $dp_R=(-2\xi+2\tau\eta_R,2\eta_R \,\cdotp  \xi) $ and $dp_I=(
2\tau\eta_I,2\eta_I \,\cdotp  \xi)$ are linearly independent. 
Finally, no bicharacteristic leaf (see below) is trapped over a compact set.
Thus,   $\widetilde \square =\mathcal P_R+i\mathcal P_I$
is a \emph{complex principal type operator} in the sense of Duistermaat and H\"ormander \cite{DH1972}. 
\medskip

Recall that for a \emph{real} principal type operator, such as $\partial/\partial x_1$ in $\R^m,\, m\ge 2$, or   
the d'Alembertian wave operator, the  singularities propagate along  curves (the  
characteristics); for instance, for the wave equation, singularities propagate along light rays. 
\emph{Complex} principal type operators, such as $\partial_{x_1}+i\partial_{x_2}$ in $\R^m,\, m\ge 3$, 
or the operator  $\widetilde \square$ above, also propagate singularities, but now
along \emph{two} dimensional surfaces, called leaves, which are the spatial projections of the bicharacteristic surfaces 
formed by the joint flowout of $H_{p_R},\, H_{p_I}$. 
For the operator $\widetilde \square$ above, this roughly means 
that if $\widetilde \square \widehat w(x,t)=\widehat f(x,t)$ and $(x_0,t_0,\xi_0,\tau_0)\in \Sigma$ is in the wave front set of $\widehat f(x,t)$, 
then the wave front set of $\widehat w(x,t)$ contains a plane through this point. See \cite[Sec. 7.2]{DH1972} for detailed statements.

In the situation relevant for this paper,  the $x$-projection of any bicharacteristic leaf  is all of $\R^2$ and thus reaches all points of $\overline{\Omega}$.
{\revisedtwo Thus, complete information about $\sigma$ in the interior is accessible to boundary measurements  made at \emph{any} point on $\partial\Omega$. }
We will see below that although this is the case, using suitable weighted integrals over the boundary produces far superior imaging; 
however, this is due to the amplitudes, not the underlying geometry.
  
 For the remainder of the paper, we limit ourselves to the Calder\'on problem in $\R^2$; 
 we begin by recalling the complex Beltrami equation formalism and CGO solutions of \cite{Astala2006a}, 
 as well as their modification in \cite{Huhtanen2012}.} 
 {\revised The complex analysis in these approaches  reflects the complex principal type structure discussed above,
disguised by the fact that we are working in two dimensions.}

\section{Conductivity equations and CGO solutions}\label{Construction of CGO solutions}

\noindent On a domain $\Omega\subset \R^2=\C$, let $\sigma\in L^\infty(\Omega)$ be a strictly positive conductivity, 
$\sigma\equiv 1$ near $\partial\Omega$, and extended to be $\equiv 1$ outside of $\Omega$. 
The complex frequencies $\zeta\in\mathbb C^2$ with $\zeta\cdot\zeta=0$ may be parametrized
by $\zeta=(k,ik),\, k\in\C$; thus,  with $z=x_1+ix_2$,  one has $\zeta\cdot x= kz$. 
Following Astala and P{\"a}iv{\"a}rinta \cite{Astala2006a}, consider simultaneously the conductivity equations for the two scalar conductivities $\sigma$ and $\sigma^{-1}$, 
\begin{eqnarray}
  \label{specialcondeq1}
  \nabla\cdot\sigma\nabla u_1 &=& 0, \qquad u_1\sim e^{ikz},\\
  \label{specialcondeq2}
  \nabla\cdot\sigma^{-1}\nabla u_2 &=& 0, \qquad u_2\sim e^{ikz}.
\end{eqnarray}
 The complex geometrical optics (\CGO) \ solutions of \cite{Astala2006a} are specified by their asymptotics
 {\mltext $u_j\sim e^{ikz}$, meaning that for all $k\in \C$,}
\begin{equation}\label{intro:eq713'}
  u_j(z,k) = e^{ikz}\big(1+\bigO(\frac{1}{z})\big), \mbox{ as }|z|\to \infty.
\end{equation}
The \CGO \  solutions  are constructed via the 
Beltrami equa\nolinebreak tion,
\begin{equation}\label{Beltrami}
  \dbar_z f_\mu = \mu\,\overline{\partial_z f_\mu}, 
\end{equation}
where the Beltrami coefficient $\mu$ is defined in terms of  $\sigma$  by  
\begin{equation}\label{def:mu}
  \mu := \frac{1-\sigma}{1+\sigma}.
\end{equation}
$\mu$ is a  compactly supported, $(-1,1)$-valued function and, 
due to the assumption that  $0< c_1\le\sigma\le c_2<\infty$, one has $ |\mu|\le 1-\epsilon$ for some $\epsilon>0$.
It was shown in \cite{Astala2006a} that (\ref{Beltrami}) has solutions, for coefficients $\mu$ and $-\mu$, resp.,  of the form
\begin{eqnarray}\label{intro:eq13'}
  f_\mu(z,k) &=& e^{ikz}(1+\omega^+(z,k))\\
  f_{-\mu}(z,k) &=& e^{ikz}(1+\omega^-(z,k))\nonumber
\end{eqnarray}
with 
$$
  \omega^{\pm}(z,k)={\mathcal O}(\frac{1}{|z|}) \mbox{ as }|z|\to \infty.
$$
The various \CGO\ solutions are then related by the equation
\begin{equation}\label{relation}
2 u_1(z,k) =   f_\mu(z,k) +  f_{-\mu}(z,k)  + {\overline{ f_{\mu}(z,k)}} - {\overline{ f_{-\mu}(z,k)}},
\end{equation}
which follows from the fact that the real part of  $f_\mu(z,k)$ solves the equation (\ref{specialcondeq1}),  while the imaginary part solves (\ref{specialcondeq2}).
\medskip

In this work we will mainly focus on $\omega^+$, henceforth denoted simply by $\omega$; 
however, we will use $\omega^-$ in the symmetry discussion in Sec. \ref{sec evenodd}. 
Both of these can be extracted from voltage/current measurements for $\sigma$  at the boundary, $\partial\Omega$, 
as encoded in the Dirichlet-to-Neumann (DN) map of (\ref{specialcondeq1}). For the most part we will suppress the superscripts $\pm$, 
with it being understood in the formulas that for $\omega^\pm$, one uses $\pm\mu$.
\medskip

Huhtanen and Per\"am\"aki \cite{Huhtanen2012} introduced   the following modified derivation
of $\omega$, which, by avoiding issues caused by the exponential growth in the $k^{\perp}$ directions,  is highly efficient from a computational point of view. 
Let  $e_k(z):=\exp(i(kz+\overline{k}\overline{z}))=\exp\left(i2\re\left(kz\right)\right)$;  note that $|e_z(z)|\equiv 1$ and $\overline{e_k}=e_{-k}$. Define (as in \cite{Astala2006,Astala2006a})
\begin{equation}\label{def:APalpha}
\nu(z,k) := e_{-k}(z)\mu(z), \hbox{ and } \alpha(z,k) := -i\overline{k}e_{-k}(z)\mu(z). 
\end{equation}
Both $\alpha$ and $\nu$ are compactly supported in $\Omega$; 
since   $\dbar\omega = \nu\overline{\partial\omega}+\alpha\overline{\omega}+\alpha$, we see that $\dbar\omega$ is compactly supported as well. For future use, also note that
\begin{equation}\label{def:APalphabar}
\overline{\nu}(z,k) = e_{k}(z)\mu(z)\quad \hbox{ and }\quad
\overline{\alpha}(z,k) = i{k}e_{k}(z)\mu(z).
\end{equation}
\medskip

It was shown in   \cite[eqn.(4.8)]{Astala2006a} that $\omega(z,k)$ satisfies the inhomogeneous Beltrami equation,
\begin{equation} \label{omegaPDE}
  \dbar\omega - \nu\overline{\partial\omega}-\alpha\overline{\omega}=\alpha,
\end{equation}
where the Cauchy-Riemann operator $\overline{\partial}$ and  derivative ${\partial}$  are taken with respect to $z$.
Recall  the (solid) Cauchy transform $P$ and  Beurling transform\index{Beurling transform} $S$, defined by
\begin{eqnarray}
  \label{CauchyPdef}
  Pf(z) &=&- \frac{1}{\pi}\int_\C \frac{f(z_1)}{z_1-z}\,d^2 z_1, \\
  \label{BeurlingSdef}
   Sg(z) &=& -\frac{1}{\pi}\int_\C \frac{g(z_1)}{(z_1-z)^2}\,d^2 z_1,
\end{eqnarray}
which satisfy $\dbar P=I$, $S=\partial P$ and $S\overline\partial=\partial$ on $C_0^\infty(\C)$.
\medskip

It is shown in \cite{Huhtanen2012}, using the results of \cite{Astala2006a}, that  (\ref{omegaPDE}) has a unique 
solution $\omega\in W^{1,p}(\C)$ for  $2<p<p_\epsilon:=1+\frac1{1-\epsilon}$, where $\epsilon>0$ is such that $|\mu|\le 1-\epsilon$.
Now define $u$ on $\Omega$  by $\overline{u} = -\dbar \omega$; note that $u\in L^p(\Om)$, $\omega = -P\overline{u}$
and $\partial\omega = -S\overline{u}$.
Re-expressing   (\ref{omegaPDE}) in terms of $u$ leads to 
$$
  -\overline{u} -\nu\overline{(-S\overline{u})} -\alpha\overline{(-P\overline{u})}=\alpha.
$$
Using (\ref{def:APalpha}), this  further simplifies to
\begin{equation}\label{belteq2_u}
  u + (-\overline{\nu}S -\overline{\alpha}P)\overline{u} = -\overline{\alpha},
\end{equation}
which then  can be expressed as the integral equation, 
\begin{equation}\label{belteq3_u}
  (I + A\rho)u = -\overline{\alpha},
\end{equation}
where  $\rho(f):=\overline{f}$
denotes complex conjugation and $A:=(-\overline{\alpha}P-\overline{\nu}S )$. 
As shown in \cite{Astala2006a} and  \cite[Sec. 2]{Huhtanen2012}, $I+A$ is invertible on $L^p(\Om)$. 
Denote by  $U(k,\mu)=u(\,\cdotp,k)|_{\Omega}$ the restriction to $\overline\Omega$ of the unique solution to (\ref{belteq3_u}), and hence (\ref{belteq2_u}).

\section{Fr\'echet differentiability and the Neumann series}\label{sec frechet}

We now come to the key construction of the paper. 
{\agt For $\epsilon>0$ and any $\Omega_0\subset\subset \Omega$, let }
$$X=\{\mu\in L^\infty(\Omega);\ \hbox{ess supp}(\mu)\subset \Omega_0,\ \|\mu\|_{L^\infty(\Omega)}\le 1-\epsilon\}.$$ 
Furthermore, define $Y$ to be the closure of $C^\infty(\overline\Omega)$ with respect to 
$$||u||_Y:=\|u\|_{L^2(\Omega)} + \|\, u|_{\partial\Omega}\, \|_{L^\infty(\partial\Omega)}.$$
For $k\in \C$, let $U_k$ be the $\R$-linear map $U_k:X\to  L^2(\Om)$, 
given by $U_k(\mu)=u_\mu(\,\cdotp,k)$,
where $u_\mu(z,k)$  is the unique solution 
$u=u_\mu(\,\cdotp,k)\in  L^2(\Om)$ of the equation (\ref{belteq2_u}).  
Define $W_k:X\to  Y$ by 
$$
W_k\mu=\omega _\mu(\,\cdotp,k)=-P(\overline {u_\mu(\,\cdotp,k)}).
$$

\subsection{Fr\'echet differentiability}\label{subsec Frechet diff}

We will show that, for each $k\in\C$,  $W_k$ is a $C^\infty$-map $X\to Y$ 
and analyze its Fr\'echet derivatives at  $\mu_0=0$. 
For each $k$, one can solve (\ref{belteq3_u}) by a Neumann series which converges 
for $\|\mu\|_{L^\infty}$ sufficiently small. 
We  analyze the individual terms of the series by introducing polar coordinates in the $k$ plane, 
$k=\tau e^{i\vp}$, and  then taking the partial Fourier transform in $\tau$. 
The leading term in the Neumann series will be the basis for the  edge detection imaging technique that is the main point of the paper,
while the higher order terms are transformed into multilinear operators acting on $\mu$. 
The remainder of the paper will then be devoted to understanding the Fourier transformed terms,
using the first derivative for effective edge detection in EIT and obtaining partial control over the higher derivatives.
\bs

\begin{theorem}\label{theo: Frechet derivatives}
The map $U_k:X\to  L^2(\Omega)$,  $U_k(\mu):=u_\mu(\,\cdotp,k)$, is
 infinitely
 Fr\'echet-differentiable with respect to $\mu$, and its   Fr\'echet derivatives
are real-analytic functions of $k\in \C$. Moreover, for $p\ge 1$, its $p$-th order Fr\'echet derivative at $\mu=0$
 in direction $(\mu_1,\mu_2,\dots,  \mu_p)\in (L^2(\Omega_0))^p$   satisfies
\begin{eqnarray}\label{frechet norm}
 & &\bigg\|\frac{D^p U_k}{D\mu^p}\bigg|_{\mu=0}( \mu_1,\mu_2,\dots,  \mu_p)\bigg \|_{L^2(\Omega)}\\
  &&\nonumber \quad
  \leq C_p(1+|k|)^p \|\mu_1\|_{L^2(\Omega)}\,\cdotp \| \mu_2\|_{L^2(\Omega)}\cdots
 \|\mu_p\|_{L^2(\Omega)}
\end{eqnarray}
 for some $C_p>0$.
 In particular, the first Fr\'echet derivative has the form
 \be\label{frechet eqn}
\frac{DU_k}{D\mu}\bigg|_{\mu=0}(\mu_1)=-P\rho(ike_{-k}\mu_1).
\ee
Moreover, for $k\in \C$ the map $W_k:X\to Y$, 
$$
W_k(\mu):= \omega_\mu(\,\cdotp,k)=-P\rho (u_\mu(\,\cdotp,k)),$$
 is infinitely 
 Fr\'echet-differentiable with respect to $\mu$
and its Fr\'echet derivatives  are real-analytic functions of $k\in \C$.
 \end{theorem}

\begin{proof}
We can rewrite (\ref{belteq2_u}) for $u=u_\mu(\,\cdotp,k)\in L^2(\Omega)$ as
\be\label{new eqn}
(I-e_{k}\mu S\rho) u+ike_{k}\mu P\rho u =  ike_{k}\mu.
\ee
On the left hand side, $e_k$ and $\mu$ denote pointwise  multiplication operators with the functions 
$e_k(z)$ and $\mu(z)$, resp.; on the right, $e_{k}(z)\mu(z)$ is an element of $L^2(\Omega)$.

Since $\|\rho\|_{L^2(\Omega)\to L^2(\Omega)}=1$, $\|S\|_{L^2(\Omega)\to L^2(\Omega)}=1$,  
and $\|\mu\|_{L^\infty(\Omega)}<1$,  the inverse operator
 $(I-e_{k}\mu S\rho)^{-1}:L^2(\Omega)\to L^2(\Omega)$ exists and is a $C^\omega$ function (i.e., a real analytic function) of $k$.
Thus, (\ref{new eqn}) can be rewritten as 
\be\label{third eqn}
(I-B_{\mu,k})u= K_{\mu,k}(ike_{k}\mu),
\ee
where 
\be\label{B operator}
K_{\mu,k}u= (I-e_{k}\mu S\rho )^{-1}u,\quad
B_{\mu,k}u=K_{\mu,k}( ike_{k}\mu P \rho u).
\ee
Since $P:L^2(\Om)\to L^2(\Om)$ is a compact operator, (\ref{B operator}) defines a compact operator
$B_{\mu,k}:L^2(\Om)\to L^2(\Om)$. 
To find the kernel of $I-B_{\mu,k}$,  consider $u^0\in L^2(\Om)$ satisfying $(I-B_{\mu,k})u^0=0$.
Then,
\begin{equation} \label{u0 eq}
(I-e_{k}\mu S\rho) u^0+ike_{k}\mu P\rho u^0 = 0.
\end{equation}
When we consider  $P$, given in (\ref{CauchyPdef}), as an operator $P:L^2(\Om)\to L^2_{loc}(\C)$,
equation (\ref{u0 eq}) yields  that  $f^0(z) = - e^{ikz} (P\overline{u^0})(z)\in L^2_{loc}(\C) $ satisfies
\begin{eqnarray} \label{omegaPDE 2}
& &   \dbar_z f^0(z) = \mu(z)\,\overline{\partial_z f^0(z)},\quad z\in \C,
\\ \nonumber
&&   e^{-ikz}f^0(z)= {\mathcal O}(\frac{1}{|z|}) \mbox{ as }|z|\to \infty.
\end{eqnarray}
 By   \cite{Astala2006a}, the solution $f^0$ of (\ref{omegaPDE 2}) has to be zero.
 Hence, $$u^0(z)=-\overline{\partial ( e^{-ikz}f^0(z))}=0$$ and 
  the operator $I-B_{\mu,k}:L^2(\Om)\to L^2(\Om)$ is one-to-one.
Thus  the Fredholm equation (\ref{third eqn}) is uniquely
solvable and we can write its solutions as $u=u_\mu(\,\cdotp,k)$,
\be\label{third eqn solved}
u_\mu(\,\cdotp,k)=(I-B_{\mu,k})^{-1}K_{\mu,k}
(ike_{k}\mu).
\ee
 By the Analytic Fredholm Theorem, the maps $k\mapsto K_{\mu,k}$ and $k\mapsto (I-B_{\mu,k})^{-1}$
are real-analytic,  $\C\to \mathcal L(L^2( \Omega),L^2(\Omega))$, where $\mathcal L(L^2(\Omega),L^2( \Omega))$ is the space of the bounded linear operators $L^2( \Omega)\to L^2( \Omega)$.
\ms

Define $K^{(p)}=\frac{D^p}{D\mu^p}K_{\mu,k}|_{\mu=0}$ and $B^{(p)}=\frac{D^p}{D\mu^p}B_{\mu,k}|_{\mu=0}$. 
Since $K_{\mu,k}|_{\mu=0}=I$, we see that
\begin{eqnarray*} 
K^{(p)}(\mu_1,\mu_2,\dots,  \mu_p)
&=&\sum_{\sigma} ( e_k  \mu_{\sigma(1)}S\rho)\circ ( e_k  \mu_{\sigma(2)}S\rho)\circ \dots
\circ (e_k  \mu_{\sigma(p)}S\rho),
\end{eqnarray*} 
where the sum is taken over permutations $\sigma:\{1,2,\dots,p\}\to \{1,2,\dots,p\}$.
Furthermore, one has
\begin{eqnarray*} 
B^{(p)}(\mu_1,\mu_2,\dots,  \mu_p)
&=&K^{(p-1)}(\mu_2,\mu_3,\mu_4,\dots,  \mu_p)\circ ( ike_{k}\mu_1 P \rho)\\
& &+K^{(p-1)}(\mu_1,\mu_3,\mu_4,\dots,  \mu_p)\circ ( ike_{k}\mu_2 P \rho )\\
& &+K^{(p-1)}(\mu_1,\mu_2,\mu_4,\dots,  \mu_p)\circ ( ike_{k}\mu_3 P \rho )\\
& &+\dots+K^{(p-1)}(\mu_1,\mu_2,,\dots,  \mu_{p-1})\circ ( ike_{k}\mu_p P \rho ).
\end{eqnarray*} 
We can compute the higher order derivatives
$\frac{D^p}{D\mu^p}(I-B_{\mu,k})^{-1}|_{\mu=0}$, in the direction $(\mu_1,\mu_2,\dots,  \mu_p)$, 
using the polarization identity for symmetric multilinear functions, if these derivatives are known in
the case when $\mu_1=\mu_2=\dots=\mu_p$.
In the latter case the derivatives 
can be computed using  Faa di Bruno's formula, 
which generalizes the chain rule to higher derivatives, 
 $$
     {d^p \over dt^p} f(g(t)) =\sum \frac{p!}{m_1!\,m_2!\,\cdots\,m_p!}\cdot f^{(m_1+\cdots+m_n)}(g(t))\cdot \prod_{j=1}^n\left(\frac{g^{(j)}(t)}{j!}\right)^{m_j},
$$
where the sum runs over indexes $(m_1,m_2,\dots,m_p)\in \mathbb N^p$ satisfying  $m_1+2m_2+\dots+pm_p = p$. Indeed, this formula can be applied
with $f(B)=(I-B)^{-1}$ and $g(t)=B_{t\mu_1,k}$. As $g(0)=0$ and the norm of the  $p$-th derivative of $B_{t\mu_1,k}$ 
with respect to $t$ is bounded by $c_p(1+|k|)^p\|\mu_1\|^p$, we obtain estimate
(\ref{frechet norm}). Moreover, since $k\mapsto ike_k\mu$ is a real analytic map, $\C\to L^2( \Omega)$,
we see that  the Fr\'echet derivatives 
$k\mapsto \frac {D^pu_\mu}{D\mu^p}|_{\mu=0}(\,\cdotp,k)\in L^2( \Omega)$ are real analytic maps of $k\in \C$.

Finally, recall that $\Omega_0\subset \Omega$  is a relatively compact set.
For $\mu\in X$, we have $\hbox{supp}(\mu)\subset \Omega_0$, and thus  the function $u_\mu(\,\cdotp,k)=U_k(\mu)$ is also supported in
$\Omega_0$. As $P$ is given in (\ref{CauchyPdef}) we see easily that  for  $(\mu_1,\mu_2,\dots,  \mu_p)\in (L^2(\Omega_1))^p$
  the Fr\'echet derivatives 
$$  \frac {D^pW_k }{D\mu^p}\bigg |_{\mu=0}(\mu_1,\mu_2,\dots,  \mu_p)=
-P\rho \frac {D^pU_k }{D\mu^p}\bigg|_{\mu=0}(\mu_1,\mu_2,\dots,  \mu_p)$$
are in $Y$, and  these derivatives 
 are real analytic functions of $k\in \C$.
\end{proof}

\subsection{Neumann series}\label{Neumann series}

Now consider a Neumann series expansion approach to solving (\ref{belteq3_u}), 
looking for  $u\sim\sum_{n=1}^\infty u_n$, with $u_1:=-\overline\alpha$ and $u_{n+1}:=-A\overline{u_n},\, n\ge 1$; the resulting $\omega_n$ are defined by
$$\omega=-P\overline{u}\sim\sum_{n=1}^\infty -P\overline{u_n}=:\sum_{n=1}^\infty \omega_n.$$ 
The first three terms of each expansion are given by
\begin{equation}\label{u zero}
u_1=-\overline{\alpha},\quad \omega_1=P\alpha,
\end{equation}
\begin{equation}\label{u one}
u_2=A\alpha=-(\overline\alpha P+\overline{\nu}S)(\alpha),\quad \omega_2=
P({\alpha}\overline{P\alpha}+{\nu}\overline{S\alpha}),
\end{equation}
and
\beqa\label{u two}
u_3 &=& -(\oal P +\onu S)(\alpha \overline{P}\oal +\nu\overline{S}\oal),\nonumber \\
\nonumber \\
\omega_3&=&P(\alpha\overline{P}+\nu\overline{S})(\overline{\alpha}P\alpha+\overline{\nu}S\alpha).
\eeqa
By Thm. \ref{theo: Frechet derivatives},  $U_k:X\to L^2(\Omega)$ is $C^\infty$, and hence we have
\beqa\label{u derivatives}
u_n(\,\cdotp,k)&=&\frac{D^n U_k}{D\mu^n}\bigg|_{\mu_0=0}(\mu,\mu,\dots,\mu),\\
\omega_n(\,\cdotp,k)&=&-P\rho \bigg( u_n(\,\cdotp,k)\bigg).\nonumber
\eeqa

Due to the polynomial {\agt growth in the} estimates (\ref{frechet norm}), the functions $u_n(z,k)$ and $ \omega_n(z,k)$ 
{\revised are} tempered distributions in the $k$ variable. 
Hence we can introduce polar coordinates, $k=\tau e^{i\vp}$,
and then take the partial Fourier transform with respect to $\tau$ of the tempered distributions
 $\tau \mapsto u_n(z,k)|_{k=\tau e^{i\vp}}$ and $\tau\mapsto \omega_n(z,k)|_{k=\tau e^{i\vp}}$.
Later we prove the following theorem concerning the partial Fourier transforms of the Fr\'echet derivatives:

\begin{theorem}\label{theo: F-transform of Frechet derivatives}
{\mltext Let $\mu\in X$ and consider the partial Fourier transforms of the Fr\'echet derivatives, 
 \begin{eqnarray}\label{F-transform of frechet der}
& &\tw_n^{z_0}(t,e^{i\vp})=\F_{\tau \to t}\bigg(\omega_n(z_0,k)\bigg|_{k=\tau e^{i\vp}}\bigg),\quad n=1,2,\dots,\\
\nonumber 
& &\omega_n(\,\cdotp,k)=-P\rho \big(\frac{D^{n+1} U_k}{D\mu^{n+1}}\bigg|_{\mu_0=0}( \mu,\mu,\dots,  \mu)\big),
\end{eqnarray}
that we  denote at $z_0\in \partial\Omega$ by
$$
\tw_n(z_0,t,e^{i\vp})=\tw_n^{z_0}(t,e^{i\vp}).
$$
Then we have 
$$
\tw_n^{z_0}(t,e^{i\vp})=T_n^{z_0}( \mu\otimes\cdots\otimes\mu),$$
 where
$T_n^{z_0}$  are  $n$-linear operators given by 
\begin{eqnarray*}
\nonumber
& &\!\! T_n^{z_0}(\mu_1\otimes\cdots\otimes\mu_n):=\\
& &\!\!\int_{\C^{n}} K_n^{z_0}(t,e^{i\vp};z_1,\dots,z_{n})\, \mu_1(z_1)\cdots \mu_n(z_{n})\, d^2z_1\cdots d^2z_{n}.
\end{eqnarray*}
}The wave front set of the Schwartz kernel $K_n^{z_0}$ 
is contained in the union of a collection  $\{\Lambda_J:\, J\in\mathcal J\}$ of $2^{n-1}$ pairwise
 cleanly intersecting 
Lagrangian manifolds,   indexed by
$\mathcal J$, the power set of $\{1,\dots,n-1\}$.
For each $J\in\mathcal J$, 
$\Lambda_J$  is the conormal bundle  of a
smooth submanifold,
$L_n^J\subset \R\times\sone\times \C^{n}$, i.e., $\Lambda_J=N^*L_n^J$, with
\begin{eqnarray}\label{lagrangians LnJ}\\
\nonumber
L_n^J:=\Big\{t+\left(-1\right)^{n+1}2\re\left(e^{i\vp}\sum_{j=1}^{n}(-1)^jz_j\right)=0\Big\}\cap \bigcap_{ j\in J} \{z_j-z_{j+1}=0\}.\hspace{-5mm}
\end{eqnarray}
 \end{theorem}
 \medskip

Roughly speaking, Theorem \ref{theo: F-transform of Frechet derivatives} implies that the operator $T_n^{z_0}$ transforms singularities of $\mu$ to singularities
of $\tw_n^{z_0}$  so that the singularities of $\mu$  propagate  along  
the $L_n^J$. {\agt Further discussion, as well as the proof of the theorem, will be found  later in the paper.}

\medskip

The first-order term  $\omega_1$ will serve as the  basis for stable edge and singularity detection,
while the higher order terms need to be characterized in terms their regularity and the location of their wave front sets. 
After the partial Fourier transform $\omega\to\tw$  described in the next section, 
the map $T_1:\mu\to\tw_1$ turns out to be essentially a derivative of the Radon transform. 
Thus,  \emph{the leading term of\, $\tw$ is a nonlinear Radon transform} of the conductivity $\sigma$, 
allowing for good reconstruction of the singularities of $\sigma$ from the singularities of $\tw_1$. 
The   higher order terms $\tw_n$  
{\agt record} 
scattering effects and explain artifacts observed in  
simulations; these  should be  filtered out or otherwise taken into account for efficient numerics and 
accurate reconstruction. 
We characterize this scattering in detail for  $\tw_2$ in terms of oscillatory integrals; almost as precisely for $\tw_3$; 
and in terms of the wave front set for  $\tw_n,\, n\ge 4$.

\section{Fourier transform and the virtual variable}\label{Analysis of the zeroth-order term}

We continue the analysis  with two elementary transformations of the problem: 
\smallskip

(i) First, one introduces polar coordinates in the complex frequency, $k$, writing $k=\tau e^{i\vp}$, with $\tau\in\R$ and $e^{i\vp}\in\sone$. 
\smallskip

(ii) Secondly, one takes a partial Fourier transform in $\tau$,  introducing a  nonphysical artificial (i.e.,  {\it virtual})  variable, $t$. We show that the introduction of this variable 
reveals the complex principal type structure of the problem, as discussed in \S\ref{CPT structure}. This allows for good propagation of singularities from the interior of $\Omega$ to the boundary, allowing singularities of the conductivity in the interior to be robustly detected by voltage-current measurements at the boundary.

By (\ref{def:APalpha}), $\omega_1=ikP(e_k\mu)$ (see also (\ref{frechet eqn})),  so that
\begin{equation}\label{omega zero}
\omega_1(z,k)=\frac{ik}\pi \int_{\C} \frac{e_k(z_1)\mu(z_1)}{z-z_1} \, d^2z_1.
\end{equation}
Write
the complex frequency as $k=\tau e^{i\vp}$ with $\tau\in\R,\, \vp\in [0,2\pi)$ (which we usually identify with $e^{i\vp}\in {\mathbb S}^1$). Taking the partial Fourier transform in $\tau$ then yields 
\begin{eqnarray}\label{omega zero tilde}
\quad\tw_1(z,t,e^{i\vp})&:=&\int_\R e^{-it\tau}\omega_1(z,\tau e^{i\vp})\, d\tau \\
&=& \frac{e^{i\vp}}\pi\int_{\R} \int_\C\frac{e^{-i\tau t}}{z-z_1}\, (i\tau)\, e_{\tau e^{i\vp}}(z_1) \mu(z_1)\, d^2z_1\, d\tau\nonumber \\
&=&  \frac{e^{i\vp}}\pi \int_{\R}\int_\C (i\tau)\, \frac{e^{-i\tau(t-2 \re(e^{i\vp}z_1))}}{z-z_1}\, \mu(z_1)\, d^2z_1\, d\tau\nonumber\\
&=&-2e^{i\vp}\int_{\C}\frac{\delta'(t-2\re(e^{i\vp}z_1))}{z-z_1}\, \mu(z_1)\, d^2z_1,\nonumber
\end{eqnarray}
with the integrals  interpreted in the sense of distributions. Note that since $t$ is dual to $\tau$, which is the (signed) length of a frequency variable,  for heuristic purposes $t$ may be thought of as  temporal.
\medskip

\subsection{Microlocal analysis of $\tw_1$}\label{subsec microlocal}

Fix $\Omega_0\subset\subset\Omega_2\subset\subset\Omega$ and assume once and for all that $supp(\mu)\subset\Omega_0$, 
i.e., $\sigma\equiv 1$ on $\Omega_0^c$.  Let $\Omega_1:=(\overline{\Omega_2})^c\supset\Omega^c\supset\partial\Omega$.
Then the map $T_1:\mathcal E'(\Omega_0)\to \mathcal D'(\Omega_1\times\R\times {\mathbb S}^1)$, defined by 
$$\mu(z_1)\longrightarrow (T_1\mu)(z,t,e^{i\vp}):=\tw_1(z,t,e^{i\vp}),$$
has Schwartz kernel
\begin{equation}\label{Schwartz}
K_1(z,t,e^{i\vp},z_1)=-2e^{i\vp}\frac{\delta'(t-2\re(e^{i\vp}z_1))}{z-z_1}.
\end{equation}

Note that $|z-z_1|\ge c>0$ for $z\in{\Omega_1}$ and $z_1\in\Omega_0$. 
For $z\in\partial\Omega$ and $z_1\in\Omega_0$, the factor $(z-z_1)^{-1}$ in (\ref{Schwartz})  is smooth, 
and $T_1$ acts on  $\mu\in\mathcal E'(\Omega_0)$ as a standard Fourier integral operator (FIO)  
(See \cite{Hor1971} for the standard facts concerning FIOs which we use.)  However, as we will see below, the amplitude $1/(z-z_1)$, although $C^\infty$,   both 
\medskip

\indent (i) accounts for the fall-off  rate in detectability of jumps, namely as  the inverse of the distance from the boundary; and
\smallskip

\indent (ii) causes artifacts, especially when  some singularities of $\mu$ are close to the boundary,
due to its large magnitude and the large gradient of its phase. 
\medskip

To see this, start by noting that the kernel $K_1$ is singular at the hypersurface,
\begin{equation}\label{Sigmas}
L:=\{(z,t,e^{i\vp},z_1):\, t-2\re(e^{i\vp}z_1)=0\}\subset \C\times\R\times {\mathbb S}^1\times\C,\nonumber
\end{equation}
Write $z=x+iy,\, z_1=x'+iy'$, and   use $\zeta,\,\zeta'$ to denote their dual variables, $(\xi,\eta),\, (\xi',\eta')$. Using the defining function 
$t-2\re(e^{i\vp}z_1)=t-2(\cos(\vp)x'-\sin(\vp)y')$, identifying $\C$ with $\R^2$ as above and ${\mathbb S}^1$ with $[0,2\pi)$, 
we see that the conormal bundle of $L$ is

\begin{eqnarray}\label{Lambdas}
\Lambda :=  N^*L 
& =&  \Big\{\Big(z,2\re(e^{i\vp}z_1),e^{i\vp},x',y';\, 0,0,\tau,\nonumber \\
  & &\qquad\qquad  2\tau\im(e^{i\vp}z_1),-2\tau e^{-i\vp}\Big): \\
 \qquad \qquad\quad & & \qquad\qquad z\in\Omega_1,\,z_1\in\Omega_0,\, e^{i\vp}\in\sone,\, \tau\in\R\setminus 0 \Big\},\nonumber 
\end{eqnarray}
which  is a Lagrangian submanifold of $T^*(\Omega_1\times\R\times {\mathbb S}^1\times\Omega_0)\setminus \nolinebreak0$. 
The kernel $K_1$   has the oscillatory representation,
\begin{equation}\label{K_0 osc int}
K_1(z,t,e^{i\vp},z_1)=\int_{\R} e^{i\tau\left(t-2\,\re\left(e^{i\vp}z_1\right)\right)} \frac{e^{i\vp}(i\tau)}{\pi(z-z_1)}\, d\tau,\, 
\end{equation}
interpreted in the sense of distributions.
The amplitude in (\ref{K_0 osc int}) belongs to the standard space of  symbols  $S^{1}_{1,0}$ on 
$(\Omega_1\times\R\times {\mathbb S}^1\times \Omega_0)\times (\R\setminus \nolinebreak0)$ \cite{Hor1971}. Thus, using H\"ormander's notation and orders for   Fourier integral (Lagrangian) distribution classes \cite{Hor1971}, $K_1$ is of order $1+\frac12-\frac04$, i.e., 
$K_1\in I^{0}(\Lambda)$.
We conclude that $T_1$ is an FIO of order 0 associated with  the canonical relation 
\begin{equation}\label{can_rel_T0_a}
C\subset \Big(T^*(\Omega_1\times\R\times {\mathbb S}^1)\setminus 0\Big)\times \Big(T^*\Omega_0\setminus 0\Big),
\end{equation}
written $T_1\in I^0(C)$, where
\begin{eqnarray}\label{Cjs}
C=\Lambda'&:=&\Big\{(z,t,e^{i\vp},\zeta,\tau,\Phi; z_1,\zeta_1): \\
& & \quad (z,t,e^{i\vp},z_1;\zeta,\tau,\Phi,-\zeta_1)\in\Lambda\Big\}.\nonumber
\end{eqnarray}
The wave front set of $K_1$ satisfies
$WF(K_1)\subset \Lambda$
(and actually, by the particular form of $K_1$, equality holds). 
Hence, by the H\"ormander-Sato lemma \cite[Thm. 2.5.14]{Hor1971}, 
$WF\big(T_1\mu\big)\subset C_0\circ WF(\mu),$
with $C$ considered as a set-theoretic relation from $T^*\Omega_0\setminus 0$ to $T^*(\Omega_1\times\R\times {\mathbb S}^1)\setminus 0$.

\medskip

We next consider the geometry of  $C$, parametrized as
\begin{eqnarray}\label{C1new}
C&=&\Big\{\big(z,2\re(e^{i\vp}z_1),e^{i\vp},0,\tau,2\tau\im(e^{i\vp}z_1); z_1,2\tau e^{-i\vp}\big):\nonumber \\
& & \qquad\quad z\in\Omega_1,\,z_1\in\Omega_0,\, e^{i\vp}\in\sone,\, \tau\in\R\setminus 0\Big\}.
\end{eqnarray}
$C$ is of dimension 6, while the natural projections to the left and right, $\pi_L:C\to T^*(\Omega_1\times\R\times {\mathbb S}^1)\setminus 0$ 
and $\pi_R:C\to T^*\Omega_0\setminus 0$, are into spaces of dimensions 8 and 4, resp.  $C$ satisfies the Bolker condition \cite{Gu,GuSt}:
$\pi_L$ is an  immersion (which is equivalent with $\pi_R$ being a submersion) and is globally  injective. 

{\revised However, $C$ in fact satisfies a much stronger condition  than the Bolker condition: the geometry of $C$ is  \emph{independent} of $z\in\Omega_1$,
and it is a canonical graph in the remaining variables.} 
If for any $z_0\in\Omega_1$ we set $K_1^{z_0}=K_1|_{z=z_0}$, then  one can factor $C=0_{T^*\Omega_1}\times C_0$ 
(with the obvious reordering of the variables), where $0_{T^*\Omega_1}$ is the zero-section of $T^*\Omega_1$ and 
\begin{eqnarray}\label{C zero sharp}
C_0&:=& WF(K_1^{z_0})'\nonumber \\
&=& \Big\{\big(2\re(e^{i\vp}z_1),e^{i\vp},\tau,2\tau\im(e^{i\vp}z_1);z_1,2\tau e^{-i\vp}\big)\nonumber \\
& & \qquad\qquad  : 
z_1\in\Omega_0,\, e^{i\vp}\in\sone,\, \tau\in\R\setminus 0\Big\} \\
&\subset& \big(T^*(\R\times {\mathbb S}^1)\setminus 0\big)\times \big(T^*\Omega_0\setminus 0\big).\nonumber
\end{eqnarray}
(Note that $C_0=N^*L_0'$, where 
$$L_0=\{(t,e^{i\vp},z_1)\in\R\times\sone\times\C: t-2\re(e^{i\vp}z_1)=0\}.)$$
{\revised From (\ref{C1new}),(\ref{C zero sharp}) one can see that $C$ satisfies the Bolker condition, but its product structure is in fact much more stringent.}

Hence, it is reasonable to form determined (i.e., 2D)  data sets from two-dimensional slices of the full $T_1$ by fixing  $z=z_0$; 
for these to correspond to boundary measurements,  assume that $z_0\in\partial\Omega\subset\Omega_1$. 
Thus, define 
$T_1^{z_0}:\mathcal E'(\Omega_0)\to \mathcal D'(\R\times {\mathbb S}^1)$ by
$\mu(z_1)\longrightarrow (T_1^{z_0}\mu)(t,\vp):= \tw_0(z_0,t,\vp)$.
$T_1^{z_0}$ has Schwartz kernel $K_1^{z_0}$ given by (\ref{K_0 osc int}), but with $z$ fixed at $z=z_0$, 
and  thus   $T_1^{z_0}$ is an FIO of order $1+\frac12-\frac44=\frac12$ with canonical relation $C_0$, i.e.,  $T_1^{z_0}\in I^{\frac12}(C_0)$.
Further, one easily checks from (\ref{C zero sharp}) that \linebreak$\pi_R:C_0\to T^*\Omega_0\setminus 0$ 
and $\pi_L:C_0\to    T^*(\R\times {\mathbb S}^1)\setminus 0$ are  local diffeomorphisms, injective if we either restrict $\tau>0$ or $\phi\in [0,\pi)$,  
in which case $C_0$  becomes a global canonical graph. 
\smallskip

Composing $T_1^{z_0}$ with the backprojection operator $(T_1^{z_0})^*$ then yields,
 by the transverse intersection calculus for FIOs \cite{Hor1971}, a normal operator $(T_1^{z_0})^*T_1^{z_0}$ 
 which is a $\Psi$DO of order 1 on $\Omega_0$, i.e., $(T_1^{z_0})^*T_1^{z_0}\in\Psi^{1}(\Omega_0)$.  
 We will show that the normal operator is elliptic and thus admits a left parametrix, $Q(z,D)\in \Psi^{-1}(\C)$,
so that 
\begin{equation}\label{parametrix}
Q(T_1^{z_0})^*T_0^{z_0}-I\hbox{  is a smoothing operator on }\mathcal E'(\Omega_0).
\end{equation}
Therefore, $T_1^{z_0}\mu$ determines $\mu$ mod $C^\infty$, 
making it possible to determine the singularities of the Beltrami multiplier $\mu$, 
and hence those of  the conductivity $\sigma$, from the singularities of $T_1^{z_0}\mu$. 
All of this follows from standard arguments once one shows that $T_1^{z_0}$ is an elliptic FIO. 

To  establish this ellipticity,  we may, because $z_0-z_1\ne 0$ for $z_1\in \Omega_0$, calculate the principal symbol $\sigma_{prin}(T_1^{z_0})$ using (\ref{Schwartz}). 
At a point of $C_0$, as given by the parametrization (\ref{C zero sharp}), we may calculate the induced symplectic form $\varkappa_{C_0}$ on $C_0$,

\begin{eqnarray}\label{Omega}
\varkappa_{C_0}&:=&\pi_R^*(\varkappa_{T^*\Omega_0})\nonumber \\
&\,=& -2\tau\, d\vp\wedge (s(\vp)dx' + c(\vp) dy')\\
& & +2d\tau\wedge (c(\vp)dx'- s(\vp) dy'),\nonumber
\end{eqnarray}
so that $\varkappa_{C_0}\wedge\varkappa_{C_0}=4\tau d\vp\wedge d\tau\wedge dx'\wedge dy'$, and the half density 
$$|\varkappa_{C_0}\wedge\varkappa_{C}|^\frac12=2|\tau|^\frac12 |d\vp\wedge d\tau\wedge dx'\wedge dy|^\frac12.$$
From this it follows that
$$\sigma_{prin}(T_1^{z_0})=\frac{-2e^{i\vp}(i\tau)}{2|\tau|^\frac12(z_0-z_1)}=\frac{(-i e^{i\vp})\, sgn(\tau)|\tau|^\frac12}{z_0-z_1},$$
which is elliptic of order $1/2$ on $C_0$.

\medskip

\noindent{\bf Example.} Although  (\ref{parametrix}) allows imaging  of general $\mu\in\mathcal E'(\Omega_0)$ 
from $\omega_1(z_0,\cdot,\cdot)$, consider the particular  case where $\mu$ is 
a piecewise smooth function with jumps across an embedded smooth curve $\gamma=\{z: g(z)=0\}\subset\Omega_0$ 
(not necessarily closed or connected), with unit normal $n$. 
In fact, consider the somewhat more general case of a $\mu$ which  is {\em conormal of order} $m\in\R,\, m\le -1,$ with respect to $\gamma$, i.e., is of the form
\be\label{mu osc int}
\mu(z)=\int_\R e^{ig(z)\theta}\, a_m(x,\theta)\, d\theta,
\ee
where $a_m$ belongs to the standard symbol class 
$S^m_{1,0}\left(\Omega_0\times \left(\R\setminus 0\right)\right)$. (In general, we will denote the orders or bi-orders 
of symbols by subscripts.) A $\mu$ which is {\mltext a piecewise smooth function} with jumps across $\gamma$ is of this form for $m=-1$; for 
$-2<m<-1$, a $\mu$ given by (\ref{mu osc int}) is {\mltext piecewise smooth}, as well as  H\"older continuous of order $-m-1$ across $\gamma$. 
(Recall that uniqueness in the Calder\'on problem for $C^\omega$ {\mltext piecewise smooth} conductivities was treated in \cite{KV} 
and some cases of conormal conductivities in \cite{GLU,Kim}.)
As a Fourier integral distribution, $\mu\in I^{m}(\Gamma)$ for the Lagrangian  manifold,
\be\label{Gamma}
\Gamma:=N^*\gamma=\big\{\left(z_1,\theta\, n\left(z_1\right)\right): z_1\in\gamma,\, \theta\in\R\setminus 0\big\}\subset T^*\Omega_0\setminus 0.
\ee
By the transverse intersection calculus, $T_1^{z_0}\mu\in I^{m+\frac12}(\widetilde\Gamma)$, where
\begin{eqnarray}\label{TildeGamma}
\widetilde\Gamma &:=& C\circ \Gamma= \Big\{\big(2\re(e^{i\vp}z_1),e^{i\vp},\tau,2\tau \im(e^{i\vp}z_1)\big):\nonumber \\
& & \qquad\quad z_1\in\gamma,\, e^{i\vp}=\overline{n(z_1)},\, \tau\in\R\setminus 0\Big\}\subset T^*(\R\times {\mathbb S}^1)\setminus 0.\nonumber
\end{eqnarray}
Thus, for $\vp$ fixed, $T_1^{z_0}\mu$ has singularities at those values of $t$ of the form $t=2\re(e^{i\vp}z_1)$ with $z_1$ ranging over the 
points of $\gamma$ with $n(z_1)=e^{-i\vp}$. (Under a finite order of tangency condition on $\gamma$, for each $\vp$ there are only a finite 
number of such points.) These values of $t$ depend on $\vp$ but \emph{are independent of} $z_0\in\partial\Omega$;
this reflects the complex principal type geometry underlying the problem, which has propagated the singularities of $\mu$ out to \emph{all} of 
the boundary points of $\Omega$.
Denoting these values of $t$ by $t_j(e^{i\vp})$, the distribution $T_1^{z_0}\mu$  has Lagrangian 
singularities conormal of order $m+\frac12$ on $\R$ at $\{t_j\}$, and thus  is of 
magnitude  $\sim |t-t_j|^{-m-\frac32}$ for $-\frac32<m\le -1$. In particular, if $\mu$ is {\mltext piecewise smooth} with jumps, 
for which $m=-1$,
the singularities have magnitude $\sim |t-t_j|^{-\frac12}$.
\medskip

\noindent{\bf Remark.} More generally, since $T_1^{z_0}$ is an elliptic FIO of order 1/2 associated to a canonical graph, 
if we denote the $L^2$-based Sobolev space of order $s\in\R$ by $H^s$, it follows that  if $\mu\in H^s\setminus H^{s-1}$, 
then $T_1^{z_0}\mu\in H^{s-\frac12}\setminus H^{s-\frac32}$, allowing us to image general singularities of $\mu$ and hence $\sigma$.

\subsection{`Averages' of $\tw_1$ and artifact removal}\label{Weighted averages and the Radon transform}

As described above, each $T_1^{z_0}\in I^{\frac12}(C_0)$; the  symbol  depends on $z_0$, the canonical 
relation (\ref{C zero sharp}) does not, and we now take advantage of this. 
For any $\C$-valued weight $a(\,\cdot\,)$ on $\partial\Omega$, define
\begin{equation}\label{omega a}
\tw_1^a(t,e^{i\vp}) := \int_{\partial\Omega} \tw_1(z_0,t,e^{i\vp})\, a(z_0)\, d\mathbf{z_0},
\end{equation}
and denote by $T_1^a$ the operator taking $\mu(z_1)\to\tw_1^a(t,e^{i\vp})$. (It will be clear from context 
when the superscript is a point $z_0\in\partial\Omega$ and when it is a function $a(\cdot)$ on the boundary.) 
(We emphasize that (\ref{omega a}) is a complex line integral.)
Then $T_1^a$ has kernel
\begin{eqnarray}\label{K a}
K_1^a(t,e^{i\vp},z_1)&:=& -2e^{i\vp}\Big[ \int_{\partial\Omega} \frac{a(z_0)\, d\mathbf{z_0}}{z_0-z_1}\Big] \delta'(t-2\re(e^{i\vp} z_1)\big)\nonumber \\
& = & -4\pi i e^{i\vp} \alpha(z_1)\delta'(t-2\re(e^{i\vp}z_1)),
\end{eqnarray}
where 
$$\alpha(z_1)=\frac1{2\pi i}\int_{\partial\Omega} \frac{a(z_0)\, d\mathbf{z_0}}{z_0-z_1},\quad z_1\in\Omega,$$
is the Cauchy (line) integral of $a$. 
We thus have 
$$\sigma_{prin}(T_1^a)=2\pi e^{i\vp}\alpha(z_1)sgn(\tau) |\tau|^\frac12\hbox{ on } C_0,$$
and therefore $(T_1^a)^* T_1^a\in\Psi^1(\Omega_0)$, with 
$$\sigma_{prin}\big((T_1^a)^* T_1^a\big)(z,\zeta)=2\pi^2 |\alpha(z)|^2|\zeta|,$$
since, by (\ref{C zero sharp}), $|\tau|=\frac12|\zeta'|$ on $C_0$. Thus,
$$(T_1^a)^* T_1^a=2\pi^2|\alpha|^2\cdot |D_{z}|\hbox{ mod }\Psi^0(\Omega_0).$$ 
By choosing 
$a\equiv (\pi\sqrt2)^{-1}$ in (\ref{omega a}), one has $\alpha\equiv (\pi\sqrt2)^{-1}$ on $\Omega$ and \linebreak
$\sigma_{prin}((T_1^a)^* T_1^a)(z,\zeta)=|\zeta|$, yielding 
\begin{equation}\label{t0a inversion2}
(T_1^a)^* T_1^a=|D_{z}|\hbox{ mod }\Psi^0,
\end{equation}
which faithfully reproduces the locations of the singularities of $\mu$ and accentuates their strength by one derivative.  
This is,  in the context of {\mltext our reconstruction method}, an  analogue of local (or $\Lambda$)-tomography \cite{Faridani1992} .
\medskip

Alternatively (now with the choice of $a=1/\pi$\,),  one may obtain an exact \emph{weighted, filtered backprojection} inversion formula,
\begin{equation}\label{t0a inversion}
(T_1^a)^*(|D_t|^{-1})T_1^a=I\hbox{ on } L^2(\Omega_0).
\end{equation}
On the level of the principal symbol, this follows from the microlocal analysis above, again since $|\tau|=\frac12|\zeta'|$ on $C_0$; for the exact result, note that
\begin{equation}\label{t0a expl}
T_1^a=-\big(\frac{i \pi }{\sqrt2}\big) e^{i\vp} \big(\frac{\partial}{\partial s}R\mu\big)(\frac12 t,e^{i\vp}),
\end{equation}
where $R$ is the standard Radon transform on $\R^2$, 
$$(Rf)(s,\omega)=\int_{\mathbf x\cdot\omega=s}f(\mathbf x)\, d^1\mathbf x,\, (s,\omega)\in\R\times \mathbb S^1.$$
\bigskip

\noindent{\bf Remark.}
Note that if we take $\Omega=\mathbb D$, so that $\partial\Omega$ can be parametrized by $z_0=e^{i\theta}$, then (\ref{omega a}) becomes
\begin{equation}\label{omega circle}
\tw_1^a(t,e^{i\vp})= \int_0^{2\pi} \tw_1(e^{i\theta},t,e^{i\vp})\, ie^{i\theta}\,d\theta.\nonumber
\end{equation}
Thus, the weight is (slowly) oscillatory 
when expressed in terms of $d\theta$, but through destructive interference suppresses the artifacts present in each 
individual $\tw_1^{z_0}$. {Fig. \ref{zeroth_aver}} illustrates, with a skull/hemorrhage phantom how using this simple
weight removes the artifacts caused by the  rapid change in the amplitude and phase of the Cauchy 
factor $(z_0-z_1)^{-1}$, {shown in Fig. \ref{zeroth_phantom}}.

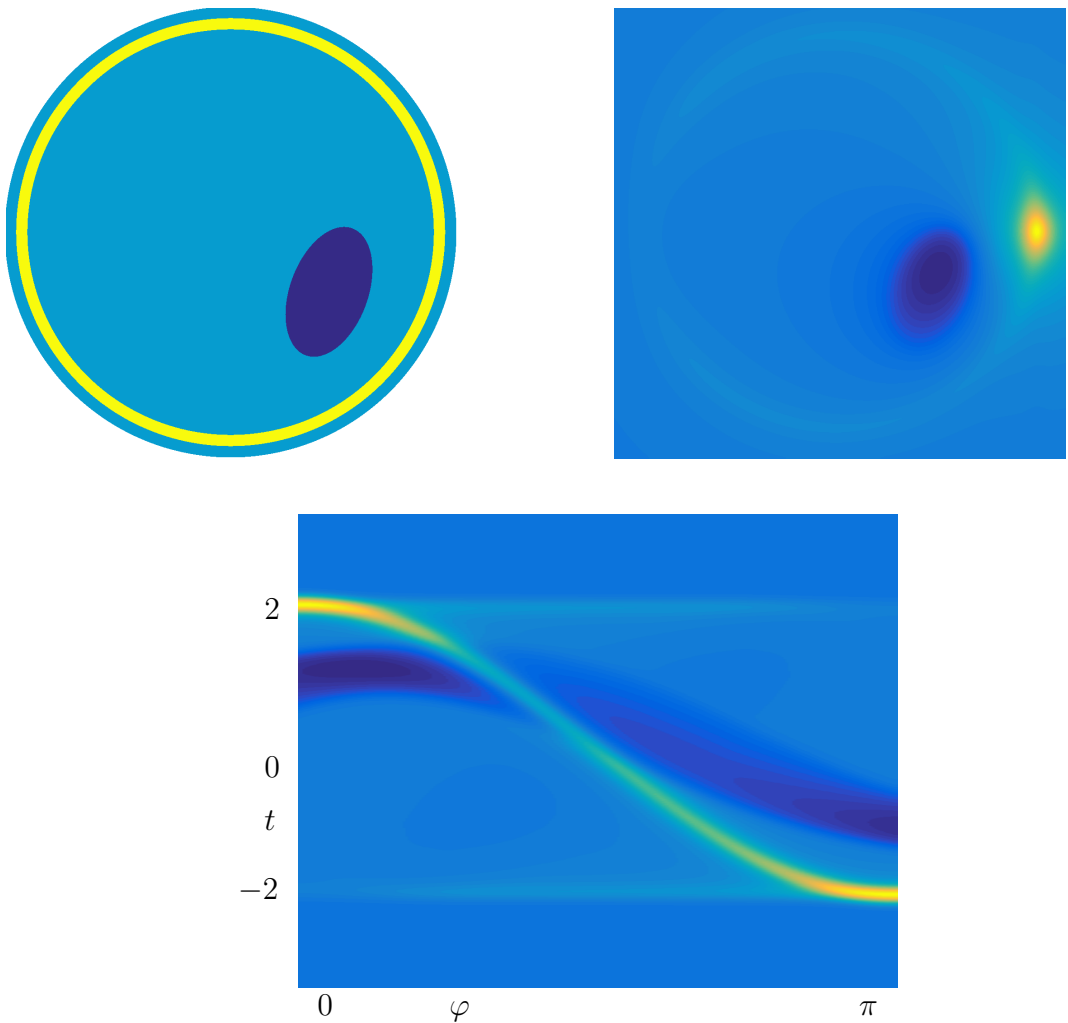
\begin{figure}
\begin{picture}(200,400)
\put(-120,220){\includegraphics[width=6cm]{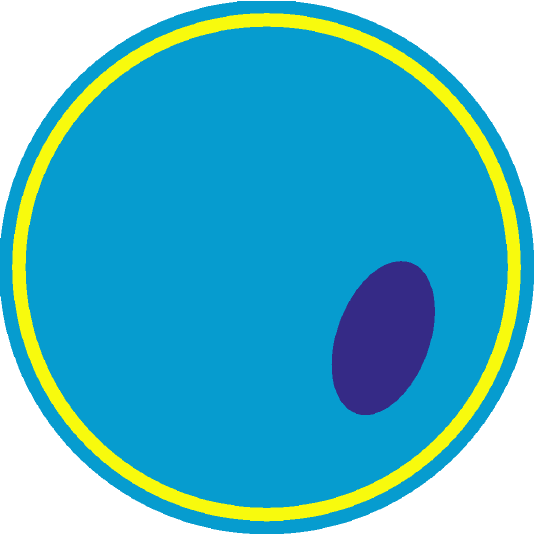}}
\put(110,220){\includegraphics[width=6cm]{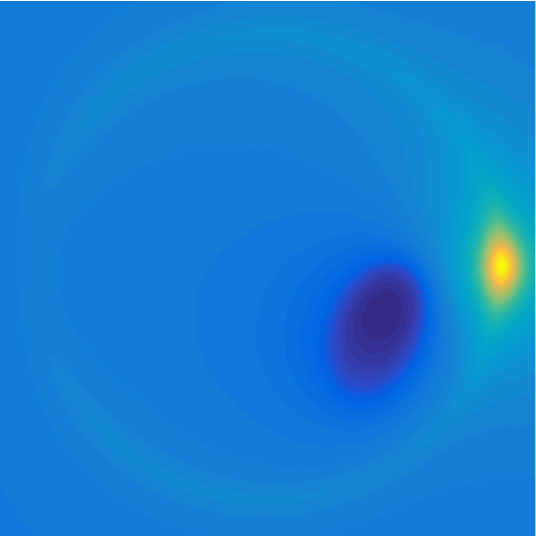}}
\put(-10,20){\includegraphics[width=8cm]{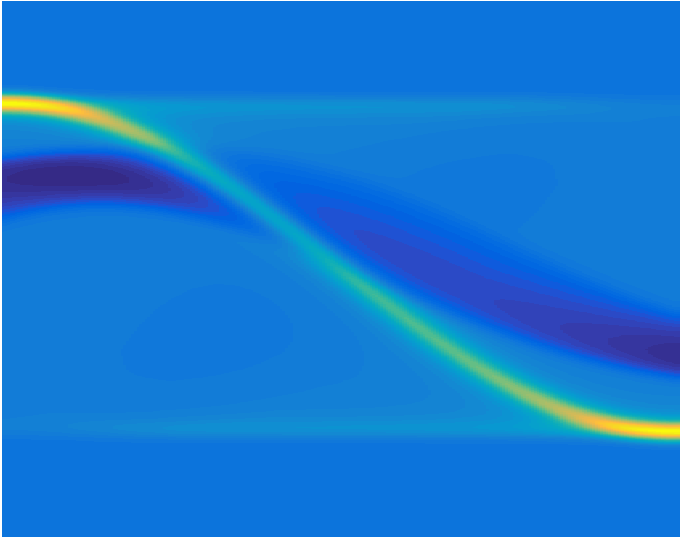}
\put(-170,-10){$\varphi$}
\put(-220,-10){$0$}
\put(-15,-10){$\pi$}
\put(-240,80){$0$}
\put(-240,60){$t$}
\put(-250,34){$-2$}
\put(-240,140){$2$}
}
\end{picture}
\caption{{\bf Artifacts from a single $T_1^{z_0}$.} Top left: Phantom  modeling hemorrhage (high conductivity inclusion)  
within skull (low conductivity shell). Bottom: $T_1^{z_0}\mu$ for $z_0 = 1$. 
Top right: backprojection  applied to $T_1^{z_0}\mu$.}\label{zeroth_phantom}
\end{figure}


\begin{figure}
\begin{picture}(200,400)
\put(-20,240){\includegraphics[width=8cm]{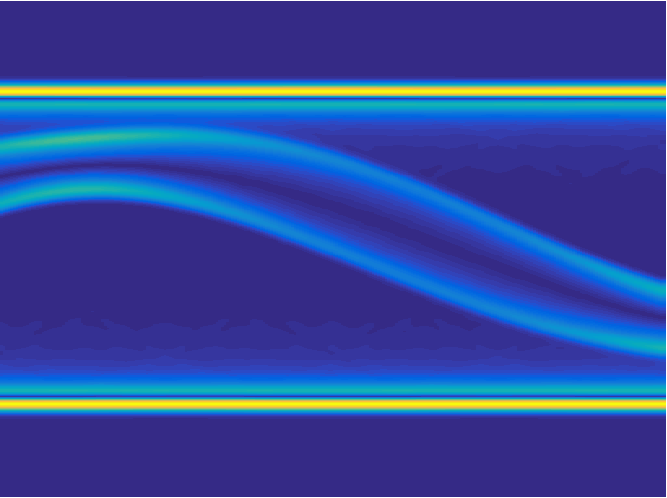}}
\put(0,0){\includegraphics[width=7cm]{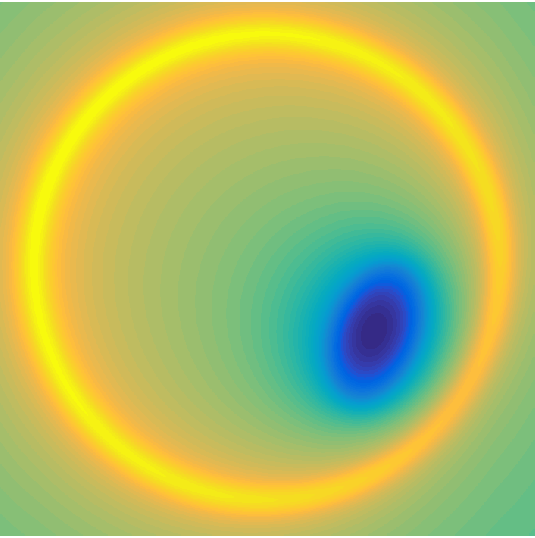}}
\put(40,230){$\varphi$}
\put(-10,230){$0$}
\put(190,230){$\pi$}
\put(-30,320){$0$}
\put(-30,295){$t$}
\put(-40,265){$-2$}
\put(-30,375){$2$}
\end{picture}
\caption{{\bf Artifact removal using weighted $T_1^a$.} Top: $T_1^{a}\mu$ for  phantom  in Fig. \ref{zeroth_phantom}. 
Bottom: reconstruction from $T_1^{a}\mu$ using formula \eqref{t0a inversion}.}\label{zeroth_aver}
\end{figure}

\section{Analysis of $\tw_2$}\label{Analysis of the first-order term}

Just as the introduction of polar coordinates and partial Fourier transform, applied to zeroth order term in the Neumann expansion 
{\agt (i.e., the Fr\'echet derivative of the scattering map at $\mu=0$),} 
give rise to a term  linear  in $\mu$, their application to the first order term (\ref{u one}) gives rise to a term which is  bilinear in $\mu$. wave front set analysis shows that this nonlinearity gives rise to two distinct types of singularities; we will see in \S\ref{Computational studies} that both of these are  visible in the numerics, and need to be taken into account to give good reconstruction based on $\tw_1^a$.
\medskip

We can rewrite (\ref{u one}) as
$$\omega_2(z,k)=P\big(\alpha(\overline{P}\oal)\big) + P\big(\nu(\overline{S}\oal)\big).$$
where the linear operators $\overline{P},\, \overline{S}$ are defined by $\overline{P}(f)=\overline{P(\overline{f})}$ and $\overline{S}(f)=\overline{S(\overline{f})}$. The kernels of 
$\overline{P},\, \overline{S}$ are just the complex conjugates of the kernels of $P,\, S$ in (\ref{CauchyPdef}), (\ref{BeurlingSdef}), resp.
We now denote the two interior variables in $\Omega_0$ by $z_1$ and $z_2$; using (\ref{def:APalphabar}), one  sees that

\begin{eqnarray}\label{omega one}
\omega_2(z,k)&=& \frac{-k^2}{\pi^2}\int_\C\int_\C \frac{e^{-2i\re(kz_1)}\mu(z_1)}{z_1-z}\frac{e^{2i\re(kz_2)}\mu(z_2)}{\ozpp-\ozp}\, d^2z_1\, d^2 z_2\nonumber \\
& & + \frac{ik}{\pi^2}\int_\C\int_\C \frac{e^{-2i\re(kz_1)}\mu(z_1)}{z_1-z}\frac{e^{2i\re(kz_2)}\mu(z_2)}{(\ozpp-\ozp)^2}\, d^2z_1\, d^2 z_2.
\end{eqnarray}
Thus, for $z_0\in\partial\Omega$,
\begin{eqnarray}\label{omega one tilde}
\tw_2(z_0,t,e^{i\vp})&=& \int_\R e^{-it\tau}\omega_1(z_0,\tau e^{i\vp})\, d\tau\nonumber \\
&=& \int_\C\int_\C K_1(z_0,t,e^{i\vp};z_1,z_2)\, \mu(z_1)\, \mu(z_2)\, d^2z_1\, d^2z_2
\end{eqnarray}
is given by a bilinear operator acting on $\mu\otimes\mu$, with kernel
\begin{eqnarray}\label{K one}
K_2^{z_0}(t,e^{i\vp};z_1,z_2) &=&\frac1{\pi^2}\Big(\frac{e^{2i\vp}\delta''(t+2\re(e^{i\vp}(z_1-z_2)))}{(z_1-z_0)(\ozpp-\ozp)}\nonumber \\
& &\qquad +\frac{e^{i\vp}\delta'(t+2\re(e^{i\vp}(z_1-z_2)))}{(z_1-z_0)(\ozpp-\ozp)^2}\Big).
\end{eqnarray}

$K_2^{z_0}$ has multiple  singularities, but, as in the case of $K_1$, the fact  that $|z_1-z_0|\ge c>0$ 
for $z_0\in\partial\Omega$ and $z_1\in supp(\mu)\subset\Omega_0$ eliminates the singularities at  $\{z_1-z_0=0\}$. 
The remaining singularities put $K_2^{z_0}$ in the general class of paired Lagrangian distributions introduced 
in \cite{MelroseU1979,GuilleminU1981}. In fact, $K_2^{z_0}$  lies in a more restrictive class of {\em nested conormal distributions} (see \cite{GreenleafU1990}), 
associated with the pair (independent of $z_0$), 

\begin{eqnarray}\label{L_1 and 3}
L_1&:=&\{t+2\re(e^{i\vp}(z_1-z_2))=0\} \nonumber \\
&\, \supset  & L_3 := \{t+2\re(e^{i\vp}(z_1-z_2))=0,\, z_1-z_2=0\} \\
& & \quad \, \, \, \, =  \{ t=0,\, z_2=z_1 \}. \nonumber
\end{eqnarray}
(The subscripts are chosen to indicate the respective codimensions in 
$\R_t\times {\mathbb S}^1_{\vp}\times\Omega_{0,\,z_1}\times\Omega_{0,\, z_2}$.)
These submanifolds have conormal bundles,
$$
\Lambda_1:= N^*L_1,\, \Lambda_3:= N^*L_3\subset 
T^*(\R_t\times {\mathbb S}^1_{\vp}\times\Omega_{0,\,z_1}\times\Omega_{0,\, z_2})\setminus 0,
$$
and $WF(K_2^{z_0})\subseteq\Lambda_1\cup\Lambda_3$. 
(As with $K_1^{z_0}$,   one can show from (\ref{K one}) that equality holds.)

\subsection{Bilinear wave front set analysis}\label{Bilinear wave front set analysis}

Define $\tw_2^{z_0}=\tw_2|_{z=z_0}$.
Since $\tw_2^{z_0}(t,e^{i\vp})=\langle K_2^{z_0}(t,e^{i\vp},\cdot,\cdot),\mu\otimes\mu\rangle$, we have
\begin{eqnarray}\label{WF tw1}
WF(\tw_2^{z_0})\subset WF(K_2^{z_0})'\circ WF(\mu\otimes\mu)\subset (\Lambda_1'\cup\Lambda_3')\circ WF(\mu\otimes\mu).\nonumber
\end{eqnarray}
Parametrizing $\Lambda_1,\, \Lambda_3$ in the usual way as conormal bundles, 
multiplying the variables dual to $z_1,z_2$ by $-1$ and then separating the variables on the left and right, we obtain  canonical relations in $T^*(\R\times {\mathbb S}^1)\times T^*(\Omega_0\times\Omega_0)$, 

\begin{eqnarray}\label{C1}
\quad C_1&:=& \Lambda_1'=\Big\{\Big(-2\re(e^{i\vp}(z_1-z_2)),e^{i\vp},\tau,-2\tau\im(e^{i\vp}(z_1-z_2));\nonumber\\
& &  \qquad\qquad\qquad z_1,z_2,-2\tau e^{i\vp},2\tau e^{i\vp}\Big):  \\
 & &  \qquad\qquad\qquad e^{i\vp}\in {\mathbb S}^1,\, z_1,z_2\in\Omega_0,\, \tau\in\R\setminus 0\Big\},\hbox{ and} \nonumber \\
 & &  \nonumber \\
 \qquad C_3&:=& \Lambda_3'=\Big\{\Big(0,e^{i\vp},\tau,0; z_1,z_1,\zeta,-\zeta\Big):\nonumber\\
 & &  \qquad\qquad\quad e^{i\vp}\in {\mathbb S}^1,\, z_1\in\Omega_0,\, (\tau,\zeta)\in\R^3\setminus 0\Big\}\label{C3}.
\end{eqnarray}
\bs

Representing $\mu\otimes\mu=\mu(z_1)\mu(z_2)$ as $(\mu\otimes 1)\cdot (1\otimes\mu)$; 
from a basic result concerning wave front sets of products \cite[Thm. 2.5.10]{Hor1971}, one sees that
\begin{eqnarray}\label{WF mu times mu}
WF(\mu\otimes\mu)&\subseteq& WF(\mu\otimes 1)\cup WF(1\otimes \mu)\cup \big(WF(\mu\otimes 1)+WF(1\otimes\mu)\big)\nonumber  \\
&\subseteq& \big(WF(\mu)\times O_{T^*\Omega_0}\big)\cup \big( O_{T^*\Omega_0}\times WF(\mu)\big)\nonumber  \\
& & \quad \cup\, \big(WF(\mu)\times WF(\mu)\big),
\end{eqnarray}
where the sets are  interpreted as subsets of $T^*\C^2\setminus 0$, writing elements as either $(z_1,z_2;\zeta_1,\zeta_2)$ or $(z_1,\zeta_1;z_2,\zeta_2)$.

Since $\zeta_1\ne0$, $\zeta_2\ne0$ at all points of $C_1$,  and $\zeta_1=0\iff \zeta_2=0$ on $C_3$,  the relation $C_1\,\cup\, C_3$, when  applied to the first two terms on the RHS of (\ref{WF mu times mu}),  gives the empty set.     

On the other hand, $C_1\,\cup\, C_3$, when  applied to  $WF(\mu)\times WF(\mu)$, contributes nontrivially to $WF(\tw_2^{z_0})$.
First, the application of $C_3$ gives 
\begin{eqnarray}\label{C_3 on WF mu}
 \qquad  \Big\{\big(0,e^{i\vp},\tau,0)\big):
 \exists \, z_1 \hbox{ s.t. } (z_1,\tau e^{-i\vp})\in WF(\mu)\Big\}\subset N^*\big\{t=0\big\}. 
\end{eqnarray}
Secondly, $C_1$ yields a contribution to $WF(\tw_2^{z_0})$ contained in what we call
the CGO \emph{two-scattering} of $\mu$, defined by 
\begin{eqnarray}\label{C_1 on WF mu}
Sc^{(2)}(\mu)&:= &  \Big\{\big(-2\re(e^{i\vp}(z_1-z_2)),e^{i\vp},\tau,-2\tau \im(e^{i\vp}(z_1-z_2))\big):\nonumber  \\
& & \qquad   \exists\, z_1,z_2\in\Omega_0\hbox{ s.t. } (z_1,\tau e^{-i\vp}),(z_2,-\tau e^{-i\vp})\in WF(\mu)\Big\}. 
\end{eqnarray}
Thus, pairs of points  in $WF(\mu)$ with spatial coordinates $z_1,\, z_2$ and  antipodal covectors $\pm\tau  e^{-i\vp}$ give rise 
to elements of
$WF(\tw_2^{z_0})$  at  \lb$t=-2\re(e^{i\vp}(z_1-z_2))$. 
Note that the expression in (\ref{C_3 on WF mu}) is not necessarily contained in $Sc^{(2)}(\mu)$, 
even if we allow $z_1=z_2$ in (\ref{C_1 on WF mu}), since $WF(\mu)$ is not necessarily symmetric 
under $(z,\zeta)\to (z,-\zeta)$ (although this does hold for $\mu$ which are smooth with jumps).
\ms

For later use, it is also be convenient  to define 
\be\label{Res01}
Sc^{(0)}(\mu):=N^*\{(t,e^{i\vp}): t=0\}\hbox{ and } Sc^{(1)}(\mu):= C_0\circ WF(\mu),
\ee
where $C_0$ is as in (\ref{C zero sharp}) above, so that the wave front set analysis so far can be summarized as,
\be\label{WFtw01}
WF(\tw_1)\subset Sc^{(1)}(\mu)\hbox{ and }WF(\tw_2)\subset Sc^{(0)}(\mu)\cup Sc^{(2)}(\mu).
\ee
This is extended to general $WF(\tw_n)$ in (\ref{WFtwn}) below.
\ms

\noindent{\bf Remarks.}

\begin{enumerate}

\item Note that if the $\tw_2^{z_0}$ are averaged out using a function $a(z_0)$ on $\partial\Omega$ as was done for $\tw_1$, 
the wave front analysis above is still valid for the resulting $\tw_2^a$, and we will refer to either as simply $\tw_2$ in the following discussion.
\ms

\item It follows from (\ref{C_3 on WF mu}) that for any $\mu$ with $\mu\notin C^\infty$, and any $z_0\in\partial\Omega$, 
we always will see singularities of $\tw_2$ at $t=0$. The only dependence  on $\mu$  of these artifacts in $WF(\tw_2)$ 
is the  question of for which incident directions $\vp$ of the complex plane wave do they occur, as dictated by (\ref{C_3 on WF mu}).
\ms

\item In addition, by (\ref{C_1 on WF mu}), any  spatially separated singularities of $\mu$ with antipodal covectors $\pm\zeta=\pm(\xi,\eta)$ 
give rise to singularities  of $\tw_2$ at  $t=-2\re(e^{i\vp}(z_1-z_2)),  \vp=-\arg(\zeta)$. Under translations, neither 
the covectors nor the differences $z_1-z_2$ associated to such scatterings change, although the  factor $(z_1-z)^{-1}$ in the kernel (\ref{K one}), 
which is evaluated at $z=z_0$,  does. Hence,  the locations and orders of
these artifacts (but not their magnitude or phase)  
are essentially independent of translations  within $\Omega_0$ of inclusions present in  $\mu$.

\end{enumerate}
\bigskip 

Given the invertibility of $T_1^{a}$ mod $C^\infty$ (at least for constant weight $a(\cdot)$), from the point of view of
{\mltext our reconstruction method,} the singularities  of $\tw_2^{a}$ at $t=0$ and at $Sc^{(2)}(\mu)$, although part of $\tw$, 
produce artifacts which  interfere with reconstruction of the singularities of $\mu$ and should either  be better characterized or filtered out. 
In  the next subsection, we do the former for a class  of $\mu$ which includes those which are {\mltext piecewise smooth} with jumps.

\subsection{Bilinear operator theory}\label{Bilinear operator theory}

Not only is  $WF(K_2^{z_0})\subset\Lambda_1\,\cup\,\Lambda_3$, but in fact $K_1^{z_0}$ belongs to the class of 
nested conormal distributions  associated with the pair $L_1\supset L_3$ (see \cite{GreenleafU1990}), and thus to the  
Lagrangian distributions associated with the cleanly intersecting pair $\Lambda_1,\, \Lambda_3$: 
$$K_2^{z_0}\in 
I^{1,0}(\Lambda_1,\Lambda_3)+I^{1,-1}(\Lambda_1,\Lambda_3).$$
Any $K_2^a$ is a linear superposition of these and thus belongs to the same class.
The \emph{linear} operators $T_2^{z_0},\, T_2^a:\mathcal E'(\Omega_0\times\Omega_0)\to\mathcal D'(\R\times
\mathbb S^1)$ with Schwartz kernels $K_2^{z_0},\, K_2^a$, resp., which we will refer to simply as $T_2$, thus 
belong to  a sum of spaces of singular Fourier integral operators,
$I^{1,\, 0}(C_1,C_3)+I^{1,\, -1}(C_1,C_3)$, and  have some similarity to singular Radon transforms (\cite{PhSt}; 
see also \cite{GreenleafU1990}), but (i)  this underlying geometry has to our knowledge not been studied before; 
and (ii) we are interested in {\em bilinear} operators with these kernels. 
Rather than pursuing optimal bounds for $T_2$ on function spaces,
we shall focus on the goal of characterizing the singularities of $\tw_2$ when $\mu$ is piecewise smooth with 
jumps. We will show that, away from $t=0$, $\tw_2$ is 1/2 derivative smoother than $\tw_1$. 
On the other hand,  at 
$t=0$ it is possible for $\tw_2$  to be as singular as the strongest  singularities of $\tw_1$; 
this is present in the full $\tw$ (computed from the DN data) and produces strong 
artifacts, which can be seen in numerics  when attempting to  reconstruct $\mu$. For this reason, data should be 
either preprocessed by filtering out a neighborhood of $t=0$ before applying backprojection, or alternatively resort to the subtraction techniques discussed in Section \ref{sec evenodd}.
\medskip

It will be helpful to  work (as with the example (\ref{mu osc int}) above) in the slightly greater generality of distributions (still denoted $\mu$) that are  
  conormal for a curve $\gamma\subset\Omega_0$, having
an oscillatory integral representation such as (\ref{mu osc int}) with an amplitude of some order $m\in\R$.
For such a $\mu$ (even for one not coming from a conductivity), we may still define both $\tw_1^{z_0}$ and $
\tw_1^a$ (denoted generically by $\tw_1$), and they  belong to $I^{m+\frac12}(\widetilde\Gamma)$, where  $
\widetilde\Gamma=C_0\circ N^*\gamma\subset T^*(\R\times\sone)\setminus 0$ is as in (\ref{TildeGamma}). 
We  also define $\tw_2:= T_2^{z_0}(\mu\otimes\mu)$ or $T_2^{a}(\mu\otimes\mu)$. 
\ms

To make the microlocal analysis of $\tw_2$  tractable, we now impose a  curvature condition on $\gamma$: 
Since $\nabla g(z)\perp T_z\gamma$ at a point $z\in\gamma$, we have $i\nabla g(z)\in T_zg$; thus, $\gamma$ has nonzero Gaussian curvature at $z$ iff
\begin{equation}\label{curv}
(i\nabla g(z))^t\nabla^2g(z)(i\nabla g(z))\ne 0,
\end{equation}
which {\em we henceforth assume  holds at all points of $\gamma$ }(or at least at all $z\in \hbox{ sing supp }\mu\subset\gamma$, which is all that matters). 
\bs

Note that (\ref{curv}) implies the finite order tangency condition referred to in the Example of \S\ref{subsec microlocal}, 
so that for each $e^{i\vp}\in\mathbb S^1$, $\tw_0(\cdot,e^{i\vp})$ is singular at a finite number of values $t=t_j(e^{i\vp})$.
\ms

\begin{theorem}\label{thm omegaone}
Under the curvature assumption (\ref{curv}),

(i)  $Sc^{(2)}(\mu)$,  defined as in (\ref{C_1 on WF mu}), is a smooth Lagrangian manifold  in $T^*(\R\times\sone)\setminus 0$; and 
\ms

(ii) if $\mu$ is  as in (\ref{mu osc int}) for some $m\in\R$, then 
\be\label{twone claim}
\tw_2=T_2^{z_0}(\mu\otimes\mu)
\in I^{2m+\frac32,-\frac12}\big(Sc^{(2)}(\mu),Sc^{(0)}(\mu)\big).
\ee
\end{theorem}

Microlocally away from $\Lambda_0\cap\Lambda_1$, a distribution $u\in I^{p,l}
(\Lambda_0,\Lambda_1)$ \lb belongs to $I^{p}(\Lambda_1\setminus\Lambda_0)$ and to $I^{p+l}
(\Lambda_0\setminus\Lambda_1)$ \cite{MelroseU1979,GuilleminU1981}.
Thus, $\tw_2\in I^{2m+1}\left(Sc^{(2)}\left(\mu\right)\right)$ on $Res^{(2)}(\mu)\setminus N^*\{t=0\}$  
and hence  is smoother than  $\tw_1\in I^{m+\frac12}(\widetilde\Gamma)$  if $m<-\frac12$. 
In contrast,  on $N^*\{t=0\}\setminus Sc^{(2)}(\mu)$, one has
$\tw_2\in I^{2m+\frac32}(N^*\{t=0\})$, which is guaranteed to be smoother than $\tw_1$ only if $m<-1$.

In particular, for $m=-1$, corresponding to $\sigma$ (and hence $\mu$) being {\mltext piecewise smooth} with jumps, 
one has  $\tw_2\in I^{-1}(Sc^{(2)}(\mu))$, while $\tw_1\in I^{-\frac12}(\widetilde\Gamma)$,  so that these 
artifacts are $1/2$ derivative smoother than the faithful image of $\mu$  encoded by $\tw_1$. 
On the other hand, the singularity of $\tw_2$ at $N^*\{t=0\})$ can be {just as strong} as the singularity of $\tw_1$ at $\tilde\Gamma$.

To summarize: for conductivities with jumps,  applying standard Radon transform 
backprojection methods to the 
full data $\tw$, or even its approximation $\tw_1+\tw_2$, rather than just $\tw_1$ (which is not  measurable directly) can  result in artifacts which are 
smoother than the leading singularities only  if one  filters out  a neighborhood of $t=0$. 
\medskip

To see (i) and (ii), start by noting from (\ref{K one}) that  $T_2(\mu\otimes\mu)(t,e^{i\vp})$ is a sum of two terms of the form
\begin{equation}\label{Tonemumu}
\int e^{i\Phi} \, a_{p,l}(*;\tau;\sigma)\, b_m(z_1;\theta_1)\, b_m(z_2;\theta_2)\, d\theta_1\, d\theta_2\, dz_1\, dz_2\, d\sigma\, d\tau,
\end{equation}
where (recalling that $g$ is a defining function for $\gamma$),
\begin{eqnarray}\label{Phi}
\qquad\Phi\,&  &\, =\Phi(t,e^{i\vp},z_1,z_2,\tau,\sigma,\theta_1,\theta_2)\nonumber \\
& &:=\tau(t+2 Re(e^{i\vp}(z_1-z_2)))+\sigma\cdot(z_1-z_2)+\theta_1 g(z_1)+\theta_2 g(z_2),
\end{eqnarray}
$b_m\in S^m_{1,0}(\Omega_0\times(\R\setminus 0))$, and the $a_{p,l}$ are product-type symbols satisfying
$$|\partial_{t,\vp,z_1,z_2}^\gamma \partial_\sigma^\beta \partial_\tau^\alpha a_{p,l}(*;\tau;\sigma)|\lesssim <\tau>^{p-\alpha} <\sigma>^{l-|\beta|}$$
on $(\R\times\sone\times\Omega_0\times\Omega_0)\times \R_\tau\times\R^2_\sigma$,
(the $*$  denoting all of the spatial variables)
 of bi-orders $(p,l)=(2,-1)$ and $(1,0)$, resp. As can be seen from (\ref{C1}),(\ref{C3}),  
 $$C_1,\, C_3\subset\big\{\zeta_2=-\zeta_1,\, |\zeta_1|=2|\tau|\big\}\subset\big\{|\zeta_1|= |\zeta_2|=2|\tau|\big\},$$
 so one can microlocalize the amplitudes in (\ref{Tonemumu}) to 
$\{|\theta_1|\sim |\theta_2|\sim|\tau|\}$ and thus replace the $a_{p,l}\cdot b_m\cdot b_m$ 
by  amplitudes  
$$a_{p+2m,l}\left(*;\left(\tau,\theta_1,\theta_2\right);\sigma\right)\in S^{p+2m,l}(\R\times\sone\times\Omega_0\times\Omega_0\times (\R_{\tau,\theta_1,\theta_2}^3\setminus 0)\times\R_{\sigma}^2)$$
with bi-orders $(2m+2,-1)$ and $(2m+1,0)$, resp.
\ms

Now homogenize  the  variables $z_1,\, z_2$,  by defining  phase variables $\eta_j:=\tau z_j\, j=1,2$. In terms of the estimates for derivatives, the new phase variables are grouped with the elliptic variables $(\tau,\theta_1,\theta_2)$; furthermore, the change of variables involves a Jacobian factor of $\tau^{-4}$, so that, mod $C^\infty$, (\ref{Tonemumu}) becomes
\begin{equation}\label{Toneall}
\int e^{i\tilde\Phi} \,a_{\tilde{p},\tilde{l}}\big(*;(\tau,\theta_1,\theta_2,\eta_1,\eta_2);\sigma\big)\, d\tau\, d\theta_1\, d\theta_2\, d\eta_1\, d\eta_2\, d\sigma ,
\end{equation}
with
\begin{eqnarray}\label{tPhi}
\tilde\Phi&=&\tilde\Phi(t,e^{i\vp};\, \tau,\theta_1,\theta_2,\eta_1,\eta_2;\, \sigma)\nonumber \\
& &:=\tau t +2 Re\big(e^{i\vp}(\eta_1-\eta_2)\big)+\theta_1 g\big(\frac{\eta_1}{\tau}\big)+\theta_2 g\big(\frac{\eta_2}{\tau}\big)+\sigma\cdot\big(\frac{\eta_1-\eta_2}{\tau}\big)
\end{eqnarray}
on $(\R\times\sone)\times(\R_{\tau,\theta_1,\theta_2,\eta_1,\eta_2}^7\setminus 0)\times\R_{\sigma}^2$ and with amplitude bi-orders $(\tilde{p},\tilde{l})=(2m-2,-1)$ and $(2m-3,0)$, resp. 
We interpret $\tilde\Phi$ 
as (a slight variation of) a multi-phase function in the sense of Mendoza \cite{Men}: 
one can check that $\tP_0:=\tP|_{\sigma=0}$ 
is a nondegenerate phase function (i.e., clean with excess $e_0=0$) which parametrizes $Sc^{(2)}(\mu)$ 
(which is thus a smooth Lagrangian). 
One does this by verifying, using (\ref{curv}), that $d^2_{(t,\phi,\tau,\theta_1,\theta_2,\eta_1,\eta_2),(\tau,\theta_1,\theta_2,\eta_1,\eta_2)}\tP_0$ has maximal rank  at  $\{d_{(\tau,\theta_1,\theta_2,\eta_1,\eta_2)}\tP_0=0\}$, namely $=7$.
On the other hand, the full phase function $\tP$ parametrizes $N^*\{t=0\}$, but rather than being nondegenerate, is clean with excess $e_1=1$, i.e., 
$d^2_{(t,\phi,\tau,\theta_1,\theta_2,\eta_1,\eta_2,\sigma),(\tau,\theta_1,\theta_2,\eta_1,\eta_2,\sigma)}\tP$ has constant rank $9-1=8$ at 
$\{d_{(\tau,\theta_1,\theta_2,\eta_1,\eta_2,\sigma)}\tP=0\}$. (See \cite{HorIV} for a discussion of clean phase functions.)
A slight modification of the results in \cite{Men} yields the following.

\begin{proposition}\label{ipl multiphase}
Suppose two smooth conic Lagrangians $\Lambda_0,\, \Lambda_1\subset T^*\R^n\setminus 0$ intersect cleanly in codimension $k$. Let 
$\phi(x,\theta,\sigma)$ be a phase function on  $\R^n\times (\R^{N+M}\setminus 0)$ be such that parametrizes $\Lambda_1$ cleanly with excess $e_1\ge 0$ and $\phi(x,\theta):=\phi|_{\sigma=0}$
parametrizes $\Lambda_0$ cleanly with excess $e_0\ge 0$. 
Suppose further that 
$a\in S^{\tilde{p},\tilde{l}}\left(\R^n\times\left(\R^N\setminus 0\right)\times\R^M\right)$. Then,
$$u(x):=\int_{\R^{N+M}} e^{i\phi_1(x,\theta,\sigma)} a(x,\theta,\sigma)\, d\theta\, d\sigma\in I^{p',l'}(\Lambda_0,\Lambda_1),$$
with 
$$p'=\tilde{p}+\tilde{l}+\frac{N+M+e_0+e_1}{2}-\frac{n}4,\quad l'=-\tilde{l}-\frac{M+e_1}2.$$
\end{proposition}
\bs

Applying the Prop. to 
each of the two bi-orders
$(\tilde{p},\tilde{l})=(2m-2,-1)$ and $(2m-3,0)$ from above, we see that $T_2^{z_0}(\mu\otimes\mu)$,
as given  by the expression (\ref{Toneall})
 is a sum of two terms,
 \be\label{twone class}\nonumber
 \tw_2^{z_0}=T_2^{z_0}(\mu\otimes\mu)
\in\big (I^{2m+\frac32,-\frac12}+ I^{2m+\frac32,-\frac32}\big)\big(Sc^{(2)}(\mu),N^*\{t=0\}\big).
\ee
Recalling that  that $N^*\{t=0\}=:Sc^{(0)}(\mu)$  and also that $I^{p',l''}\subset I^{p',l'}$ for $l''\le l'$, this yields (\ref{twone claim}),
finishing the proof of  Thm. \ref{thm omegaone}. \qed

\section{Higher order terms}\label{Analysis of the higher order terms}

\subsection{Multilinear wave front set analysis}\label{WF set analysis}

For $n\ge 3$,  and for any conductivity $\sigma$, one can analyze $WF\big(\tw_n^{z_0})$ and $WF\big(\tw_n^{a})$  by 
 $n$-linear versions of  the case $n=2$ treated in \S\S\ref{Bilinear wave front set analysis}, starting with the kernels. 
 For  $\tw_n^{z_0}$, we denote these by $K_n(t,e^{i\vp},z_1,\dots, z_{n})$, i.e., $\tw_n^{z_0}$ is given by
\begin{eqnarray}\label{twn}
\nonumber
\tw_n^{z_0}(t,e^{i\vp})\!\!&=&\!\! T_n^{z_0}(\mu\otimes\cdots\otimes\mu)\\
&:=&\!\!\int_{\C^{n}} K_n^{z_0}(t,e^{i\vp};z_1,\dots,z_{n})\, \mu(z_1)\cdots \mu(z_{n+1})\, d^2z_1\cdots d^2z_{n}.
\end{eqnarray}
The kernel for $\tw_n^{a}$ has the same geometry and orders, but amplitudes $a(\cdot)$-averaged in $z_0$, which does not affect the following analysis.
\ms

$K_n^{z_0}$ is a sum of $2^{n-1}$ terms of the form, for ${\vec\epsilon}\in\{0,1\}^{n-1}$,
\be\label{Kn} c_{\vec\epsilon}\cdot
\frac{\delta^{(n+1-|\vec\epsilon|)}\left(t+\left(-1\right)^{n+1}2\re\left(e^{i\vp}\sum_{j=1}^{n}(-1)^j z_j\right)\right)}
{(z_0-z_1)(\oz_1-\oz_2)^{1+\epsilon_1}(\oz_2-\oz_3)^{1+\epsilon_2}\cdots (\oz_{n-1}-\oz_{n})^{1+\epsilon_{n-1}}},
\ee
\smallskip

\noindent each with total homogeneity $-(2n+1)$ in $(t,z_0,\dots,z_{n})$.
 These  have singularities all  in the same locations, namely on a  lattice of submanifolds of $\R\times\mathbb S^1\times \C^{n}$.
{\mltext For each $J\in\mathcal J=\big\{ J;\ J\subset \{1,\dots,n-1\} \big\}$, as in (\ref{lagrangians LnJ}) let
\be\label{Ln}
L_n^J:=\Big\{t+\left(-1\right)^{n+1}2\re\left(e^{i\vp}\sum_{j=1}^{n}(-1)^jz_j\right)=0;\, z_j-z_{j+1}=0, \forall j\in J\Big\}.
\ee
One has $codim(L_n^J)=1+2|J|$ and $L_n^J\supset L_n^{J'}$ iff $J\subset J'$.
Rather than using set notation,  we sometimes simply list the elements of $J$.
The unique maximal element of the lattice is the hypersurface,
$$L_n^{\emptyset}:=\Big\{t+\left(-1\right)^{n+1}2\re\left(e^{i\vp}\sum_{j=1}^{n}(-1)^jz_j\right)=0\Big\},$$
while the unique minimal one  is
$$L_n^{12\cdots (n-1)}=\Big\{t=0,\, z_1=z_2=\cdots=z_{n}\Big\}.$$
(This notation replaces that used earlier for $n=1,2$: what was previously denoted $L_0$ is now $L_1^\phi$, and $L_1=L_2^\phi$, $L_3=L_2^1$.)

\ms

As stated above,
$$sing\, supp(K_n^{z_0})=\bigcup_{J\in\mathcal J} L_n^J$$
and in fact,
\be\label{WFn equal}
WF(K_n^{z_0})= \bigcup_{J\in\mathcal J} N^*L_n^J,
\ee
with the  fact that equality holds (rather than just the $\subset$ containment)  following from the nonvanishing   
in all directions at infinity of the Fourier transforms of $\delta^{(m)},\, \overline{z}^{-1}$ and $\overline{z}^{-2}$.
(However, we only need the containment, not equality,  in what follows.)

Define canonical relations 
$$C_n^J:=N^*(L_n^{J})'\subset \big(T^*(\R\times\sone)\times T^*\C^{n}\big)\setminus 0,$$ 
sometimes also denoting $C_n^\emptyset$ simply by $C_n$.
The \emph{linear} operators $T_n^{z_0}:\mathcal E'(\C^{n})\to\mathcal D'(\R\times\sone)$ with kernels $K_n^{z_0}$ 
{\revised are  (as $n$ varies) interesting prototype of generalized Fourier integral operators associated with the  
lattices $\{C_n^J: J\in\mathcal J\}$ of  canonical relations intersecting cleanly pairwise. 
There is to our knowledge no general theory of such operators, but}
in any case, we can describe the wave front relation as follows.
Let $\tSz^m$ denote the alternating sum $$\tSz^m:= z_1-z_2+\dots +(-1)^{m+1}z_m.$$

\begin{definition}\label{def resonance} 
In $T^*(\R\times\sone)\setminus 0$, define
$$Sc^{(0)}(\mu)=\big\{(0,e^{i\vp},\tau,0):\, \exists z\in\Omega\hbox{ s.t. } (z,\tau e^{-i\vp})\in WF(\mu)\big\}\subset N^*\{t=0\},$$
and, for $m\ge 1$, let
\begin{eqnarray}\label{RESj}
Sc^{(m)}(\mu)&=&\Big\{\big( (-1)^{m+1}2\re(e^{i\vp}\tSz^m), e^{i\vp},\tau,(-1)^{m}2\tau\im(e^{i\vp}\tSz^m)\big):\nonumber\\
& &\qquad \exists\,  z_1,\dots, z_m \hbox{ s.t. }\\
& & \qquad  (z_j,(-1)^{j+1}\tau e^{-i\vp})\in WF(\mu),\, 1\le j\le m\Big\}.\nonumber
\end{eqnarray}
\end{definition}

Def. \ref{def resonance} extends the definitions (\ref{Res01}) for $m=0,1$ and (\ref{C_1 on WF mu}) for $m=2$. 
The next theorem extends the WF containments (\ref{WFtw01}) for $\tw_1,\, \tw_2$, to higher $n$, 
locating microlocally the singularities of $\tw_n$. We have
\ms

\begin{theorem}\label{thm WF}
 For  any conductivity $\sigma\in L^\infty(\Omega)$ and all $n\ge 1$,

\be\label{WFtwn}
WF(\tw_n)\subset \bigcup\Big\{ Sc^{(m)}(\mu): 0\le m\le n,\,  m\equiv n\hbox{ mod }2\Big\}.
\ee
\end{theorem}

\begin{proof} This will follow from (\ref{twn}) and the H\"ormander-Sato lemma \cite[Thm. 2.5.14]{Hor1971}. 
First, to formulate the $n$-fold version of (\ref{WF mu times mu}), we introduce the following notation. For sets $A,\, B\subset T^*\C$ and 
$$I\in\mathcal I:=
{\mltext 
\big\{I;\ I\subset \{1,\dots,n\}\big\}}
$$
let
\begin{eqnarray*}
\prod_{i\in I}A_i&\times&\prod_{i'\in I^c} B_{i'} \\
&:=& \big\{(z,\zeta)\in T^*\C^{n+1}:\, (z_i,\zeta_i)\in A,\, \forall i\in I,\\
& &\qquad\qquad\qquad\qquad   (z_{i'},\zeta_{i'})\in B,\, \forall i'\in I^c\big\}.
\end{eqnarray*}
For $I\in\mathcal I$, if we set 
\be\label{WFI}
WF^I(\mu):= \prod_{i\in I}WF(\mu)_i\times\prod_{i'\in I^c} 0_{T^*\C,i'},
\ee
then the analogue of (\ref{WF mu times mu}), which follow from it  by induction, is:
\be\label{WFnplusone}
WF(\otimes^{n}\mu)\subset \bigcup_{I\in\mathcal I,\, I\ne\phi} WF^I(\mu).
\ee

Next, for $J\in\mathcal J$, define
$$\oJ:=\{i\in\{1,\dots,n\}:\, i\in J\hbox{ or }i-1\in J\}\in\mathcal I.$$
Then, $|\oJ|$ is even, and thus 
$$|\oJ^c|=|\{1,\dots,n\}\setminus \oJ|\equiv n\mod 2.$$
We can partition $\oJ=\oJ_+\cup\oJ_-\cup\oJ_\pm$, where 
\begin{eqnarray}\label{Js}
\oJ_+:=\{i\in\oJ:\, i\in J,\, i-1\notin J\}\nonumber\\
\oJ_-:=\{i\in\oJ:\, i-1\in J,\, i\notin J\}\\
\oJ_\pm:=\{i\in\oJ:\, i-1\in J,\, i\in J\}.\nonumber
\end{eqnarray}
The submanifold $L_n^J\subset\R\times\sone\times\C^{n}$ is given by defining functions $f_0,\, \{f_j\}_{j\in J}$, where 
$$f_0(t,\vp,z)=t+(-1)^{n+1}2\re\left(e^{i\vp}\sum_{i=1}^{n}\left(-1\right)^iz_i\right),$$
and
$$f_j(t,\vp,z)= z_j-z_{j+1},\, j\in J.$$
The twisted conormal bundles are  parametrized by
\begin{eqnarray*}
C_n^J&=&\Big\{\left(t,\vp,\tau d_{t,\vp}f_0; z,-\left(\tau d_zf_0+\sum_{j\in J}\sigma_j\cdot d_zf_j\right)\right):\\
& & \qquad (t,e^{i\vp},z)\in L_n^J,\, (\tau,\sigma)\in (\R\times\C^{|J|})\setminus 0\Big\}.
\end{eqnarray*}

The twisted gradients $df':=(d_{t,\vp}f,-d_z f)$ of the defining functions are
$$df_0'=\left(1,\left(-1\right)^{n+1}2\im \left(e^{i\vp}\sum_{i=1}^{n}\left(-1\right)^iz_i\right),\left(-1\right)^{n}2E\left(\vp\right)\right),$$
with $E(\vp)=(e^{-i\vp},-e^{-i\vp},e^{-i\vp},\dots,(-1)^{n}e^{-i\vp})$, 
where we identify $\pm e^{-i\vp}\in\C$ with  a real covector $(\xi_i,\eta_i)\in T^*\C$, and
$$df_j'=-\sigma_j\cdot dz_j+\sigma_j\cdot dz_{j+1},\, j\in J,$$
similarly identifying $\sigma_j\in\C$ with $(\re\, \sigma_j,\im \sigma_j)\in T^*\C$. Thus,
\begin{eqnarray}\label{Cnj}
C_n^J&=&\Big\{\Big((-1)^{n}2\re(e^{i\vp}\sum_{i=1}^{n}(-1)^iz_i),e^{i\vp},\tau,
(-1)^{n+1}2\tau\im (e^{i\vp}\sum_{i=1}^{n}(-1)^iz_i);\nonumber \\
& &\qquad z,(-1)^{n}2\tau E(\vp)+\sum_{i\in\oJ_+} \sigma_i\cdot dz_i -\sum_{i\in\oJ_-} \sigma_i\cdot dz_i \nonumber \\
& & \qquad\qquad\qquad\qquad\qquad  +\sum_{i\in\oJ_{\pm}}(-\sigma_{i-1}+\sigma_i)\cdot dz_i \Big): \\
& & \qquad\qquad e^{i\vp}\in\sone,\, z_j-z_{j+1}=0,\, j\in J,\,  (\tau,\sigma)\in (\R\times\C^{|J|})\setminus 0\Big\}.\nonumber
\end{eqnarray}
\ms

Since $WF(K_n^{z_0})'=\bigcup_{J\in\mathcal J} C_n^J$, to prove  (\ref{WFtwn}), it suffices to show that each of the $2^{n-1}(2^{n}-1)$ 
compositions $C_n^J\circ WF^I,\, J\in\mathcal J,\, I\in\mathcal I\setminus\{\emptyset\}$, is contained in one of the $Sc^{(m)}(\mu)$ for 
some $0\le m\le n$ with $m\equiv n\hbox{ mod }2$. In fact, from (\ref{WFI}) and the representation of $C_n^J$ above, 
one sees that  each $C_n^J\circ WF^I$ is either empty (e.g.,  if  $\oJ^c\cap I^c\ne\phi$), or a (potentially) nonempty subset 
of $Sc^{(m)}(\mu)$, when $m=|\oJ^c|\equiv n\mod 2$, yielding  (\ref{WFtwn}) and finishing the proof of Thm. \ref{thm WF}. \end{proof}

\section{Parity symmetry}\label{sec evenodd}

{\agt We now come to an important symmetry property  which significantly improves the imaging obtained via {\mltext our reconstruction method}.}
Recall that what we have been denoting $\tw$ is in fact  $\widehat{\omega^+}$, the partial Fourier transform  
of the correction term $\omega^+$ in the CGO solution  (\ref{intro:eq13'})
of the Beltrami equation (\ref{Beltrami}) with multiplier $\mu$. Similarly, the solution $\omega^-$ in (\ref{intro:eq13'}) 
corresponding  to $-\mu$ has partial Fourier transform $\widehat{\omega^-}$.  Astala and P\"aiv\"arinta \cite{Astala2006} 
showed that both $\omega^{+}$ and $\omega^{-}$ can be reconstructed from the Dirichlet-to-Neumann map 
$\Lambda_{\sigma}$. We  show that by taking their difference we can suppress the $\tw_n$ for \emph{even} $n$, and thus 
suppress some of the singularities described in the preceding sections, most importantly the strong singularity at 
$Sc^{(0)}(\mu)\subset N^*\{t=0\}$ coming from $\tw_2$.
\medskip

Start by writing the two Neumann series,
$$\widehat{\omega^+} \sim \sum_{n=1}^{+\infty}\widehat{\omega^+_n} = \widehat{\omega^+_{odd}}+\widehat{\omega^+_{even}},\qquad
\widehat{\omega^-} \sim \sum_{n=1}^{+\infty}\widehat{\omega^-_n} = \widehat{\omega^-_{odd}}+\widehat{\omega^-_{even}},$$
where 
$\widehat{\omega^{\pm}_{odd}}$ (resp. $\widehat{\omega^{\pm}_{even}}$) consists of the $n$ odd (resp. even) terms 
in the expansion corresponding to $\widehat{\omega^{\pm}}$. Recall that, as a function of $\mu$, $\widehat{\omega^{\pm}_n}$  is a multilinear  form of degree $n$.  
\medskip

\begin{proposition}\label{prop:sym}
Each of $\widehat{\omega^+_{odd}}$ and $\widehat{\omega^+_{even}}$ has the same parity in $t$ as the multilinear degrees of its terms, i.e.,
\begin{equation}
\widehat{\omega^+_{odd}} = -\widehat{\omega^-_{odd}}\, \hbox{ and }\, \widehat{\omega^+_{even}} = \widehat{\omega^-_{even}}.
\end{equation}
Equivalently, 
\begin{equation}
\widehat{\omega^+_{odd}} = \frac{\widehat{\omega^+}-\widehat{\omega^-}}{2}\, \hbox{ and }\,  \widehat{\omega^+_{even}} = \frac{\widehat{\omega^+}+\widehat{\omega^-}}{2}.
\end{equation}
\end{proposition}
\begin{proof}
Let $\overline u^{\pm} = - \overline \partial \omega^{\pm}$. As in Sec. \ref{Construction of CGO solutions},  $u^{\pm}$ is the solution of the integral equation \eqref{belteq3_u},
\begin{equation}
(I+A^{\pm}\rho)u^{\pm} = \mp \overline\alpha,
\end{equation}
where $A^{\pm} = \mp( \overline \alpha P + \overline \nu S)$, and $\alpha$ and $\nu$ were defined in \eqref{def:APalpha}. 
Since $A^+ = - A^-$ we have $u^+_1 = -\overline{\alpha} = -u^-_1$, $u^+_2 =-A^+ \overline{u_1^+} = - (-A^-(- \overline{u_1^-}))=u^-_2$ and by induction, for $n\geq 1$,
\begin{equation*}
u^+_{n+2} = A^+ \overline{A^+ \overline u^+_{n}} = (-1)^{n}A^- \overline{ A^- \overline{u^-_{n}}} = (-1)^{n} u^-_{n+2}.
\end{equation*}

(Another way of seeing this is that $\mu\to\tw_n$ is a form of degree $n$, with the same multilinear kernel applied to both $\pm\mu$.) \qedhere
\end{proof}

Prop. \ref{prop:sym}  provides a method to isolate the even and the odd terms in the expansion of $\tw$. 
In particular, by imaging using $\widehat{\omega_{odd}^+}$, we can eliminate the strong singularities of $\tw_2$ at $Sc^{(0)}(\mu)=N^*\{t=0\}$, described in (\ref{twone claim}),
and in fact the singularities  there of all the even terms since, by (\ref{WFtwn}), these only arise from $\tw_n$ for  even $n$.

\section{Multilinear operator theory}\label{sec higher}

Following the analysis of $\tw_2$, one can also describe the singularities of $\tw_3$, but now having to restrict  away from $t=0$.
{\revised The singularities of $\tw_3$ are of interest, since, after the symmetrization considerations from
the previous section are applied, $\tw_3$  is the first higher order term encountered after $\tw_1$.}
Recall  from above  that, if $\mu$ is {\mltext a piecewise smooth function} with jumps ($m=-1$), $\tw_2$ has  a singularity at $Sc^{(0)}(\mu)=N^*\{t=0\}$ 
as strong as  that  of $\tw_1$  at $Sc^{(1)}(\mu)$, and that its presence is due to the singularity of $K_2^{z_0}$ at the 
submanifold $L_2^1=\{t=0\}\subset L_2^\emptyset\subset \R\times\sone\times\C^2$. 
Similarly, in order to analyze $\tw_3$, we will need to 
localize $K_3^{z_0}$ away from $L_3^{12}=\{t=0\}\subset \R\times\sone\times\C^3$, 
which results in  a kernel that can then be decomposed into a sum of two kernels, each having singularities on one of two nested pairs, 
$ L_3^1\subset L_3^\emptyset$ or $L_3^2\subset L_3^\emptyset$, but not at $L_3^1\cap L_3^2=L_3^{12}=\{t=0\}$. 
We will show that applying these to $\mu\otimes\mu\otimes\mu$, as in (\ref{twn}), 
does not just result in terms with WF contained in $Sc^{(3)}(\mu)\cup Sc^{(1)}(\mu)$, as was shown in Thm. \ref{thm WF}, 
but  a more precise statement can be made:

\begin{theorem}\label{thm omegatwo}
If  $\mu\in I^m(\gamma)$ with $\gamma$ satisfying the curvature condition (\ref{curv}),
then  $Sc^{(3)}(\mu)$,  defined as in (\ref{C_1 on WF mu}), is a smooth Lagrangian manifold  in $T^*(\R\times\sone)\setminus 0$; and 

\be\label{twtwo conclusion}
\quad \tw_3|_{t\ne 0}\in I^{3m+2,-\frac12}\left(Sc^{(3)}\left(\mu\right), Sc^{(1)}\left(\mu\right)\right).
\ee 
\end{theorem}

\noindent{\bf Remark.} For $m=-1$, this is in $I^{-1}\left(Sc^{(1)}\left(\mu\right)\setminus Sc^{(3)}\left(\mu\right)\right)$, 
and thus is 1/2  derivative smoother than $\tw_1$ on $Sc^{(1)}(\mu)$. 
On the other hand, it is also in $I^{-\frac32}\left(Sc^{(3)}\left(\mu\right)\setminus Sc^{(1)}\left(\mu\right)\right)$, which is a full derivative smoother than $\tw_1$.
\ms

  To put this in perspective we first discuss what should be the leading terms 
contributing to $\tw_n$ for general $n\ge 3$. 
  The  analysis for $\tw_3|_{t\ne 0}$ given below applies 
more generally to $\tw_n$ if we localize $K_n^{z_0}$ even more strongly: not just away from  $t=0$, but away from {\em all} of the submanifolds $L_n^J\subset \R\times
\sone\times\C^{n}$ with $|J|\ge 2$. Now, for $j\ne j'$, $L_n^j\cap L_n^{j'}= L_n^{jj'}$; by localizing away from all of 
the $L_n^J$ with $|J|=2$, by a partition of unity the kernel $K_n^{z_0}$  can be decomposed into a sum of $n-1$ terms, each a nested 
conormal distribution associated with the pair $L_n^\phi\supset L_n^j,\, j=1,\dots,n-1$, resp. When these pieces of 
$K_n^{z_0}$ are applied  to $\otimes^{n}\mu$,  as in (\ref{twn}),  the results have WF  in $Sc^{(n)}(\mu)\cup Sc^{(n-2)}(\mu)$, 
and again can be shown to belong to $I^{p,l}\left(Sc^{(n)}\left(\mu\right), Sc^{(n-2)}\left(\mu\right)\right)$. However, 
as this requires localizing away from $\bigcup_{|J|\ge 2} L_n^J$, which is strictly larger than $L_n^{12\cdots (n-1)}$ if $n\ge 4$;  
thus, the analysis here is inconclusive concerning the singularities of $\tw_n|_{t\ne 0}$, and thus we only present the details for $\tw_3$.
 \ms

We now start the proof of Thm. \ref{thm omegatwo} by noting that,
for $n=3$, the lattice of submanifolds (\ref{Ln}) to which the trilinear operator $T_3^{z_0}$ is associated is  a simple diamond, $L_3^\emptyset\supset L_3^1,L_3^2\supset L_3^{12}$.
In the region $\{t\ne 0\}$, the two submanifolds $L_3^1$ and $L_3^2$ are disjoint. Hence, by a partition of unity in the spatial variables, we can write
\be\label{three mus}
\tw_3|_{t\ne 0}=\langle K_3^1+K_3^2,\mu\otimes\mu\otimes\mu\rangle,
\ee
where each $K_3^j$ is associated with the nested pair $L_3^\emptyset\supset L_3^j,\, j=1,2$. Since these two terms are so similar, we just treat the $K_3^2$ term. 
\ms

The submanifolds $L_2^2\subset L_3^\phi\subset \R\times\sone\times\C^3$ are given by
\begin{eqnarray}\label{Ltwos}
L_3^\emptyset&=&\{t-2\re\left(e^{i\vp}\left(z_1-z_2+z_3\right)\right)=0\} \hbox{ and} \nonumber \\
&  & \\
L_3^2&=& \{ t-2\re\left(e^{i\vp}\left(z_1-z_2+z_3\right)\right)=0,\, z_2-z_3=0\}.\nonumber
\end{eqnarray}
For $K_3^2$ we are localizing away from $L_3^1$, so that $z_1-z_2\ne 0$ on the support of the kernels  below. 
Thus, the factors $(\overline{z_1}-\overline{z_2})^{-1+\epsilon_1}$ in (\ref{Kn}) are smooth, 
and their dependence on $\epsilon_1$ irrelevant for this analysis. Thus, $K_3^2$ is  a sum of two terms, 
each of which we will still denote $K_3^2$,  given by
\be\label{Ktwo}
K_3^2=\int_{\R^3} e^{i[\tau\left(t-2{\mltext \text{Re}\,}\left(e^{i\vp}\left(z_1-z_2+z_3\right)\right)\right) + \left(z_2-z_3\right)\cdot\sigma]} a_{p,l}(*;\tau;\sigma)\, d\tau\, d\sigma,
\ee
where $*$ denotes the spatial variables and $a_{p,l}$ is a symbol-valued symbol of bi-order $(3,-1)$ and $(2,0)$, resp.

If, for any $c>0$, we introduce a smooth cutoff into the amplitude which is a function of $|\sigma|/|\tau|$ 
and supported in the region $\{|\sigma|\ge c|\tau|\}$,  the amplitude becomes a standard symbol of order $p+l=2$
in the phase variables $(\tau,\sigma)\in\R^3\setminus 0$. The phase function is nondegenerate and parametrizes the canonical relation
(with $C_0$ as in (\ref{C zero sharp})),
\begin{eqnarray*}
C_{0\times N}&:=& C_0\times N^*\{z_2=z_3\}\\
&\, =& \big\{\big(2\re\left(e^{i\vp}z_1\right),e^{i\vp},\tau,2\tau\im\left(e^{i\vp}z_1\right);\\
& & \qquad z_1,z_2,z_2,2\tau e^{-i\vp},\zeta_2,-\zeta_2\big): \\
& &\qquad\quad e^{i\vp}\in\sone, (z_1,z_2)\in\C^2,(\tau,\zeta_2)\in\R^3\setminus 0\big\}.
\end{eqnarray*}
This is a nondegenerate canonical relation: the projection $\pi_R:C_{0\times N}\to T^*\C^3\setminus 0$ is an immersion 
and the projection $\pi_L:C_{0\times N}\to T^*(\R\times\sone)\setminus 0$ is a submersion. Thus, this contribution to 
$K_3^2$ belongs to $I^{2+\frac32-\frac84}(C_{0\times N})=I^{\frac32}(C_{0\times N})$.
Due to the support of the 
amplitude of this term, $\pi_R(C_{0\times N})\subset\{|\zeta_1|\sim|\zeta_2|=|\zeta_3|\}$, and by reasoning similar to that used  
in the analysis of $\tw_1$,  one concludes that $\mu\otimes\mu\otimes\mu\in I^{3m}
\left(N^*\left(\gamma\times\gamma\times\gamma\right)\right)$ microlocally on this region. 
Hence, the composition
$C_{0\times N}\circ N^*\left(\gamma\times\gamma\times\gamma\right)
\subset C_0\circ N^*\gamma=:Sc^{(1)}(\mu)$
is covered by the transverse intersection  calculus,
and this contribution to $\tw_3$ belongs to 
\be\label{first twtwo term}
I^{3m+\frac32}\left(Sc^{(1)}(\mu)\right).
\ee

Now consider the contribution to (\ref{Ktwo})  from the region $\{|\sigma|\le\frac12|\tau|\}$. Writing out the 
representations of each of the three $\mu$ factors in (\ref{three mus}) as conormal distributions, we 
first note that,  using  the parametrization in (\ref{Cnj}) for $C_3^2$ and the constraint $|\sigma|\le\frac12|\tau|$, we can read off that, on $\pi_R$ of the wave front relation, 
$$|\zeta_1|=2|\tau|\hbox{ and }|\zeta_j|=|\pm(\sigma-2\tau e^{-i\vp})|\ge \frac32|\tau|,\, j=2,3.$$
Hence, again we are acting on a part of $\mu\otimes\mu\otimes\mu$ which is microlocalized where $|\zeta_1|\sim|\zeta_2|\sim|\zeta_3|$. As a result, in (\ref{Ttwoall}) below, the $\theta_j$ are grouped with $\tau$ as ``elliptic'' variables for the symbol-valued symbol estimates.
Mimicking the analysis in and following (\ref{Toneall}), homogenize $z_1,z_2,z_3$ by setting $\eta_j=\tau z_j,\, j=1,2,3$. This  leads to an expression,

\begin{equation}\label{Ttwoall}
\int e^{i\tilde\Psi} \,a_{\tilde{p},\tilde{l}}\big(*;(\tau,\theta_1,\theta_2,\theta_3,\eta_1,\eta_2,\eta_3);\sigma\big)\, d\tau\, d\theta_1\, d\theta_2\, d\theta_3\, d\eta_1\, d\eta_2\, d\eta_3\, d\sigma ,
\end{equation}
with phase
\begin{eqnarray}\label{tPsi}
\tilde\Psi&=&\tilde\Psi(t,e^{i\vp};\, \tau,\theta_1,\theta_2,\theta_3,\eta_1,\eta_2,\eta_3;\, \sigma)\nonumber \\
&:=&\tau t -2 Re\big(e^{i\vp}(\eta_1-\eta_2+\eta_3)\big)+\theta_1 g\big(\frac{\eta_1}{\tau}\big)\nonumber \\
& &\quad 
+\theta_2 g\big(\frac{\eta_2}{\tau}\big)+\theta_3 g\big(\frac{\eta_3}{\tau}\big)
+\sigma\cdot\big(\frac{\eta_2-\eta_3}{\tau}\big)\nonumber
\end{eqnarray}
on $(\R\times\sone)\times(\R_{\tau,\theta_1,\theta_2,\theta_3,\eta_1,\eta_2,\eta_3}^{10}\setminus 0)\times
\R_{\sigma}^2$ and symbol-valued symbols   with bi-orders $(\tilde{p},\tilde{l})=(3m-3,-1)$ and 
$(3m-4,0)$, resp. As with the phase $\tilde\Phi$ that arose in the analysis of $\tw_1$, $\tilde\Psi$ is a multiphase function: $\tilde\Psi_0=\tilde
\Psi|_{\sigma=0}$ is nondegenerate (excess $e_0=0$) and parametrizes $Sc^{(3)}(\mu)$,
while the full $\tilde\Psi$ is clean (excess $e_1=1$) and parametrizes $Sc^{(1)}(\mu)$. 
Applying  Prop. \ref{ipl multiphase}, with $N=10, M=2$, the terms in (\ref{Ttwoall}) with  amplitudes of bi-order $(3m-3,-1)$, 
$(3m-4,0)$, resp.,   yield  elements of $I^{3m+2,-\frac12}(Sc^{(3)}(\mu),Sc^{(1)}(\mu))$ and 
$I^{3m+2,-\frac32}(Sc^{(3)}(\mu),Sc^{(1)}(\mu))$, resp.; since the former space  contains  the latter, 
and furthermore contains the space in (\ref{first twtwo term}), we conclude that $\tw_3|_{t\ne0}\in I^{3m+2,-\frac12}(Sc^{(3)}(\mu),Sc^{(1)}(\mu))$.  
This finishes the proof of Thm. \ref{thm omegatwo}. \qed
\ms

\section{Computational studies}\label{Computational studies}

\noindent
In the idealized infinite bandwidth model discussed above, knowledge of  $\omega_1(z_0,k)$ for all complex 
frequencies $k$, and thus $T_1^{z_0}\mu=\tw(z_0,t,e^{i\vp})$ for all $(t,\,e^{i\vp})$, determines $\mu$ mod $C^\infty
$. A more physically realistic model, band limiting to $|k|\le k_{max}$, requires a windowed Fourier transform. 
This corresponds to convolving in the $t$ variable with a smooth cutoff at length-scale $\sim k_{max}^{-1}$,  
rendering the reconstruction less accurate. This section examines numerical simulations and how they are affected by this bandwidth issue. 

We first introduce a new reconstruction algorithm from the Dirichlet-to-Neumann map $\Lambda_\sigma$, 
as well as the algorithm used in the simulations. Then we will present our numerical results. In this section we take $\Omega$ to be the unit disk, $\Omega=D(0,1)$.

\subsection{Reconstruction algorithm}\label{algorithm}

The results presented in the preceding sections give rise to a linear reconstruction scheme to approximately 
recover a conductivity $\sigma$ from its Dirichlet-to-Neumann map $\Lambda_{\sigma}$. This can be summarized in the following steps:
\begin{enumerate}[(i)]

\item Find $f_{\pm \mu}(z,k)$, and so $\omega^{\pm}(z,k)$, for $z \in \partial \Omega$ and $k \in \C$, by solving the boundary integral equation
\begin{equation}\label{BIE}
f_{\pm \mu}(z,k) +e^{ikz} = (\Pc_{\pm \mu} + \Pc_0^k)f_{\pm \mu}(z,k), \qquad z \in \partial \Omega,
\end{equation}
where $\Pc_{\pm \mu}$ and $\Pc_0^k$ are projection operators constructed from $\Lambda_{\sigma}$. See \cite{Astala2006} and \cite[Section 16.3.3]{Mueller2012} for full details.

\item Write $k = \tau e^{i\varphi}$. Apply the one-dimensional Fourier transform $\F_{\tau \mapsto t}$ and 
the complex average (\ref{omega a}) in order to obtain $\widehat \omega^{a, \pm}(t,e^{i \varphi})$, with $a \equiv 1/ \sqrt{2}$.

\item Taking into account the parity result Prop. \ref{prop:sym}, define $\widehat \omega^a_{\mbox{\tiny diff}} := \frac 1 2 (\widehat \omega^{a,+}- \widehat \omega^{a, -})$. 
Apply either the exact inversion formula \eqref{t0a inversion} or the $\Lambda$-tomography analogue \eqref{t0a inversion2}  
with $\widehat \omega^a_{\mbox{\tiny diff}}$ instead of $\widehat \omega^a_1$, in order to obtain an approximation $\mu_{\mbox{\tiny appr}}$ to $\mu$.

\item The approximate conductivity is found with the identity $\sigma_{\mbox{\tiny appr}} = (1-\mu_{\mbox{\tiny appr}})/(1+\mu_{\mbox{\tiny appr}})$.
\end{enumerate}

\subsection{High-precision data assumption}

In the numerical reconstructions presented below,  the spectral parameter $k$ ranges in the disk  $\{|k|<R\}$ 
with cutoff frequency $R=60$. Such a large radius $R$ is needed for demonstrating the crucial properties of the new method; 
with a smaller radius the windowing of the Fourier transform would smooth out important features in the CGO solutions. 

Using such a large $R$ in practice would require very high precision EIT measurements, 
which cannot be achieved by current technology. However, it is possible to evaluate the needed CGO solutions 
computationally when $\sigma$  is known. (Remark: it is possible to compute useful reconstructions from real EIT 
measurements using the new method  combined with sparsity-promoting inversion algorithms, but we do not discuss such approaches further in this paper.)
This is done as in \cite{Astala2014} by solving the Beltrami equation
\begin{equation}
  \dbar_z f_\mu(z,k) = \mu(z)\,\overline{\partial_z f_\mu(z,k)},
\end{equation}
which yield very accurate solutions even for large $|k|$. From the point of view of the classical $\bar \partial$ reconstruction method \cite{Knudsen2009,Mueller2003,Mueller2012,Siltanen2000} for $C^2$ conductivities, this is the analogue of solving the Lippmann-Schwinger equation to construct the CGO solutions. 

In this section the CGO remainders $\omega^{\pm}(z,k)$, with $z \in \partial \Omega$ and $|k| < 60$, 
are constructed by solving the Beltrami equation following the Huhtanen and Per\"am\"aki  
approach \cite{Astala2014,Huhtanen2012} (see also Section \ref{Construction of CGO solutions} for more details). 
We then follow steps (ii)-(iv) of the algorithm in Section \ref{algorithm} to obtain 2D reconstructions.

\subsection{Rotationally symmetric cases}

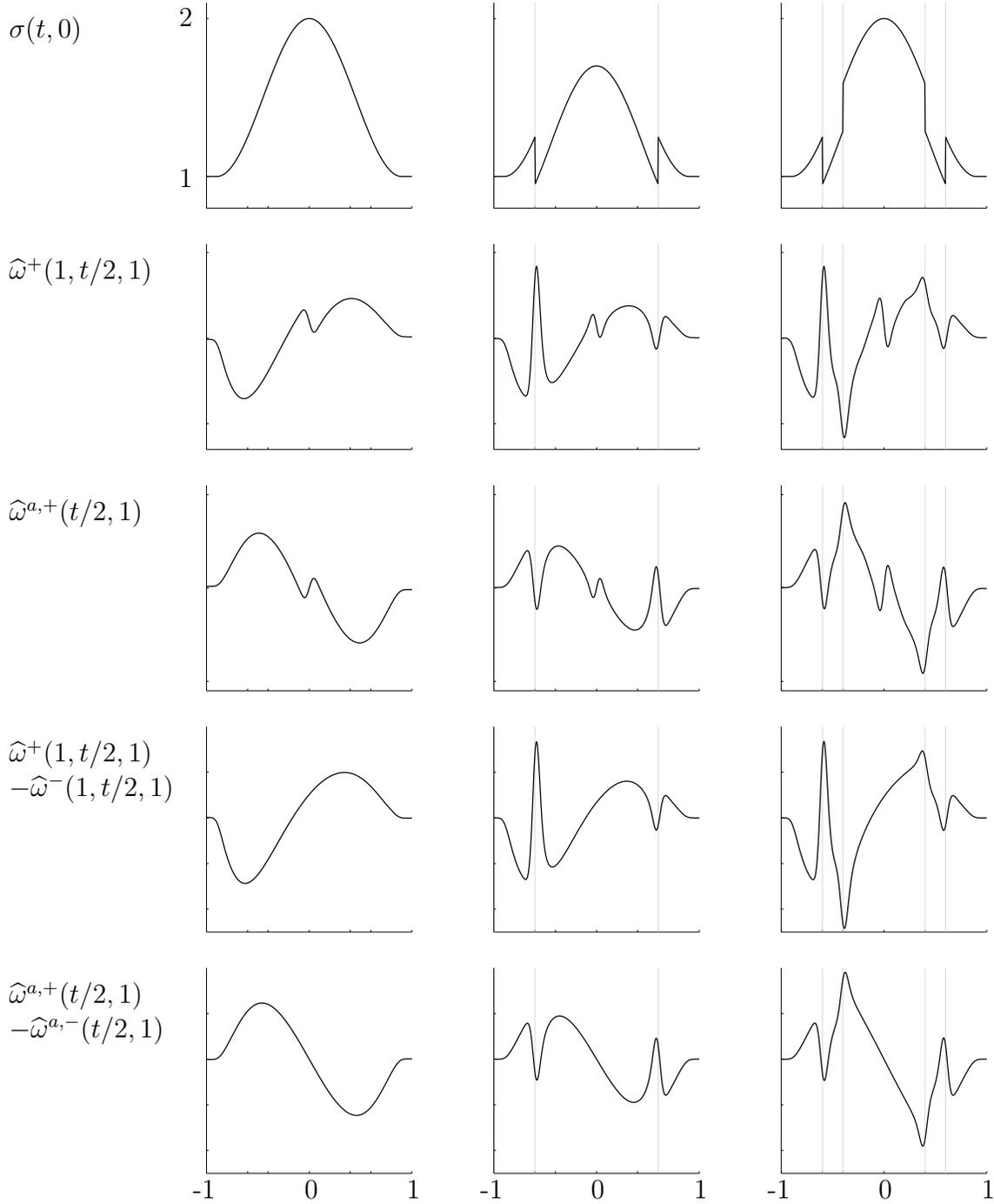
\begin{figure}
\begin{picture}(300,530)(-30,-10)
\put(-35,430){\includegraphics[width=3.2cm]{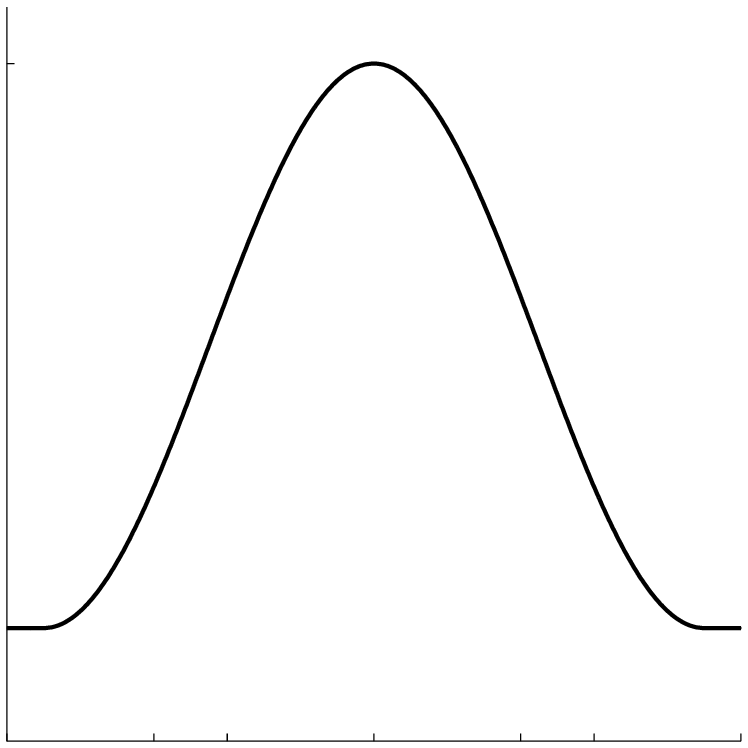}}
\put(90,430){\includegraphics[width=3.2cm]{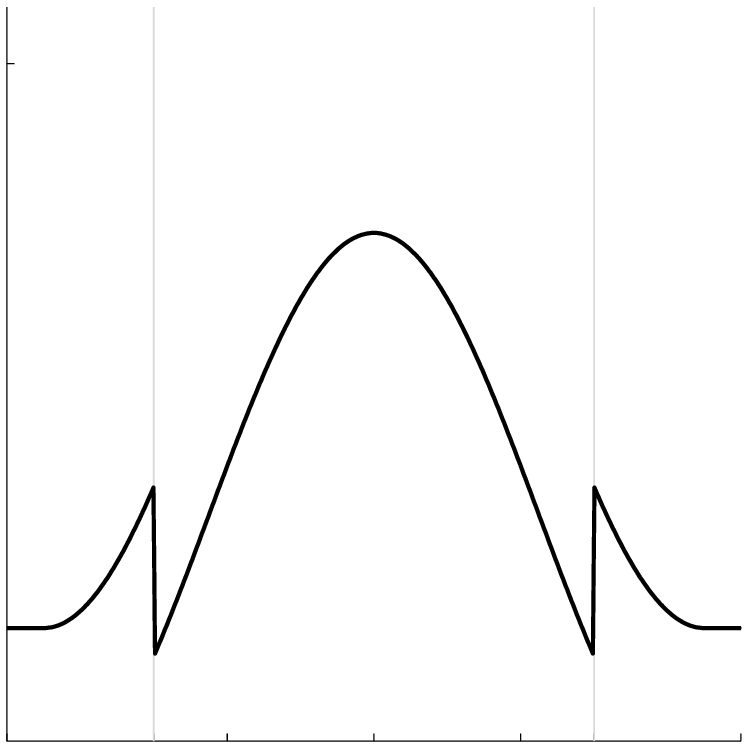}}
\put(215,430){\includegraphics[width=3.2cm]{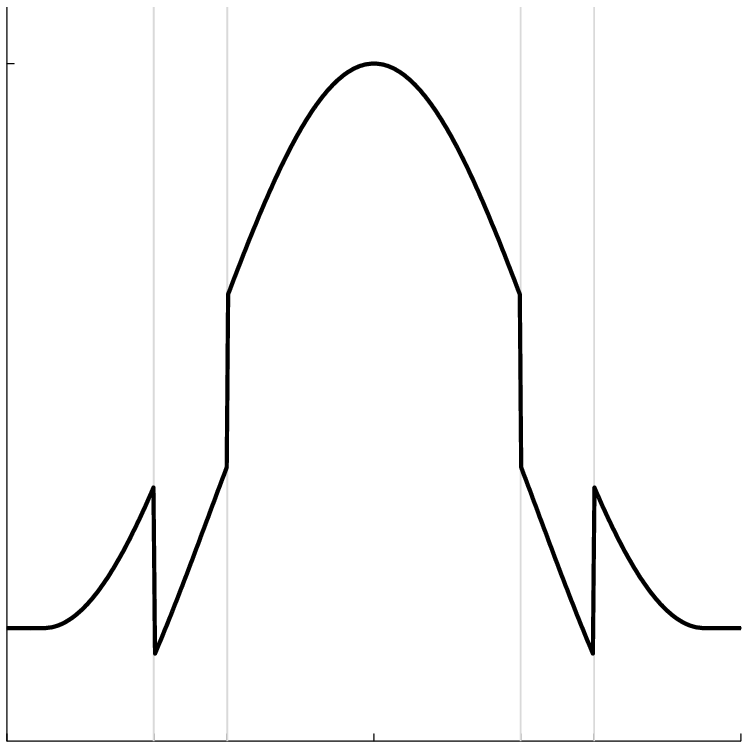}}
\put(-35,325){\includegraphics[width=3.2cm]{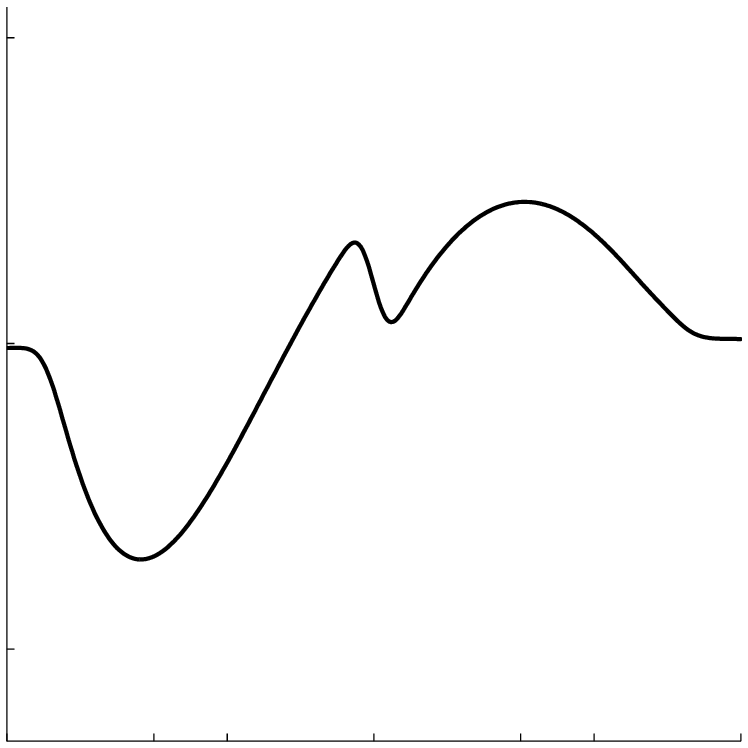}}
\put(90,325){\includegraphics[width=3.2cm]{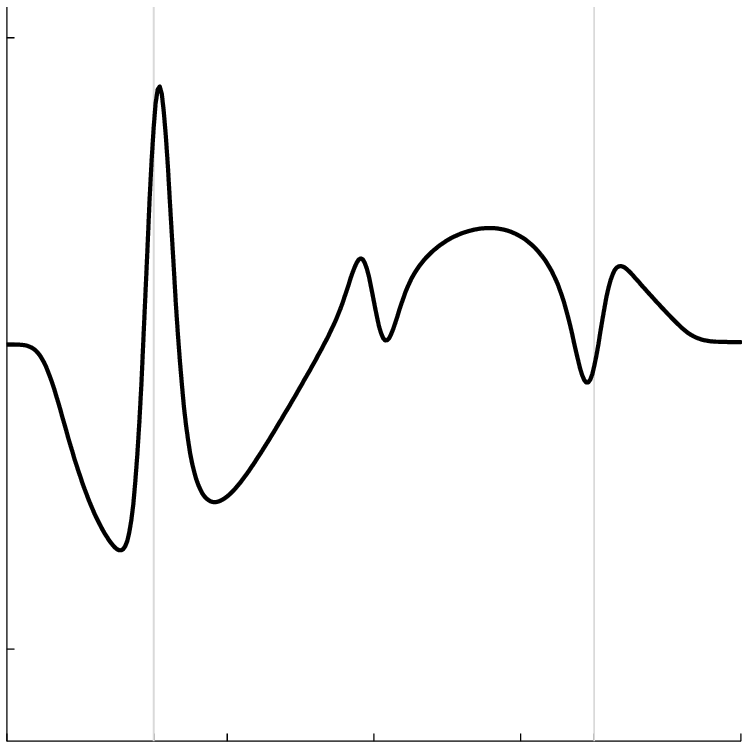}}
\put(215,325){\includegraphics[width=3.2cm]{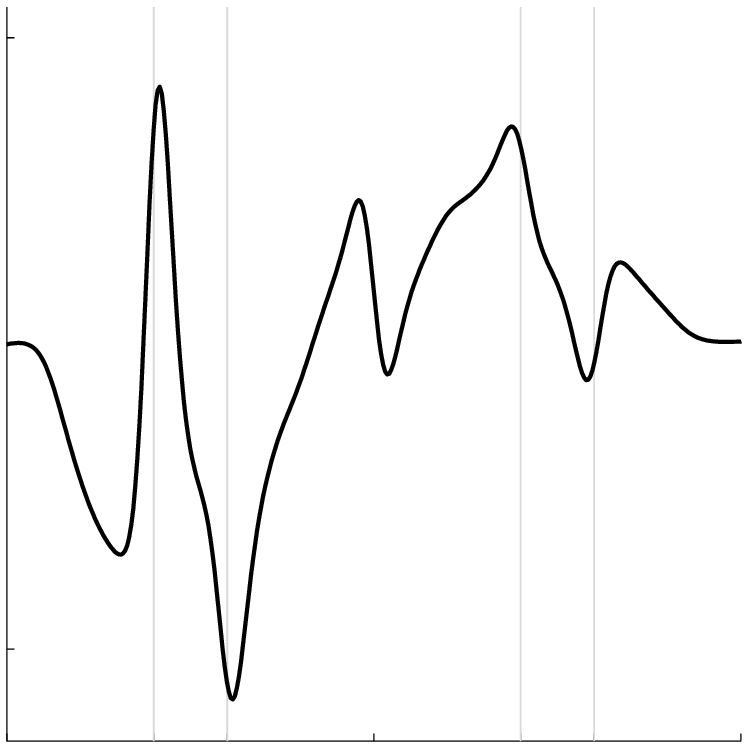}}
\put(-35,220){\includegraphics[width=3.2cm]{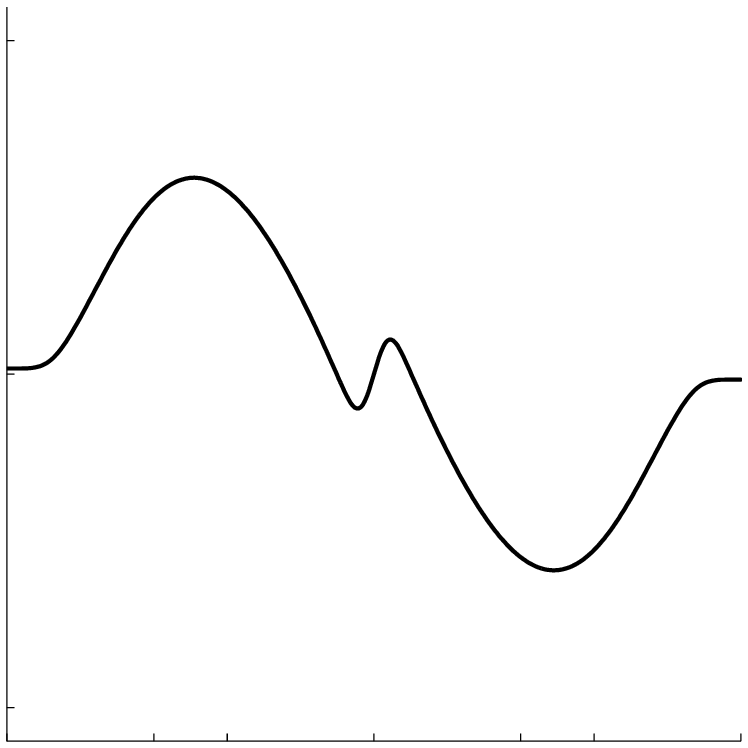}}
\put(90,220){\includegraphics[width=3.2cm]{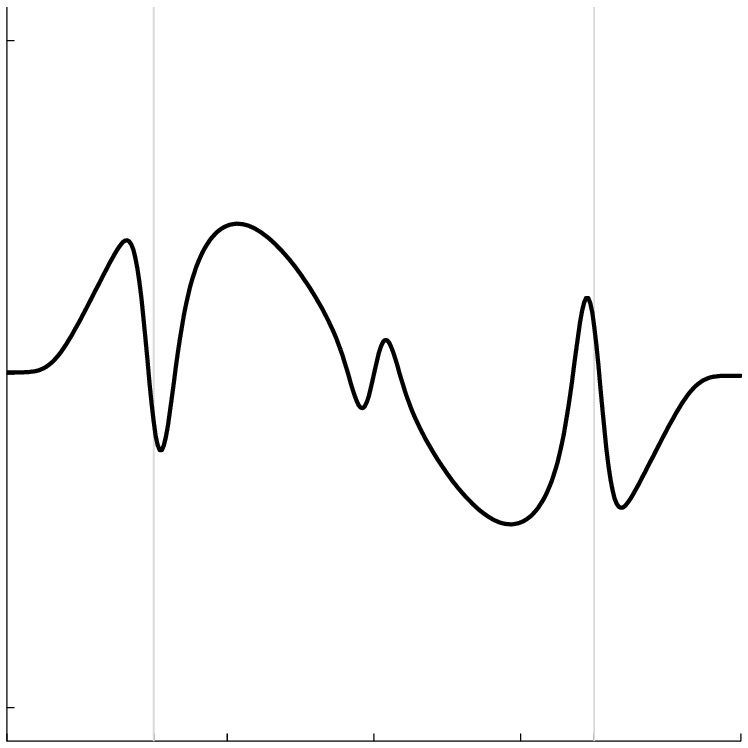}}
\put(215,220){\includegraphics[width=3.2cm]{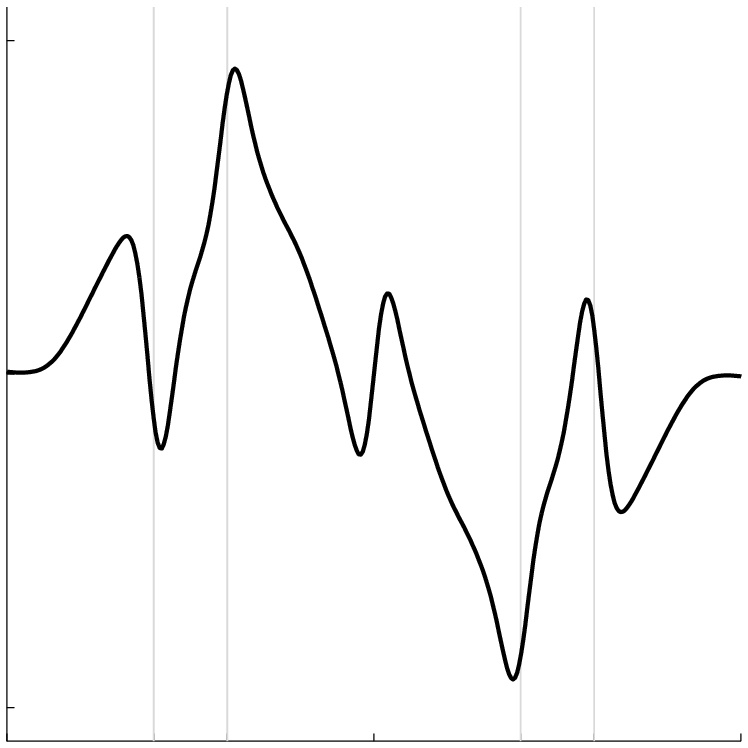}}
\put(-35,115){\includegraphics[width=3.2cm]{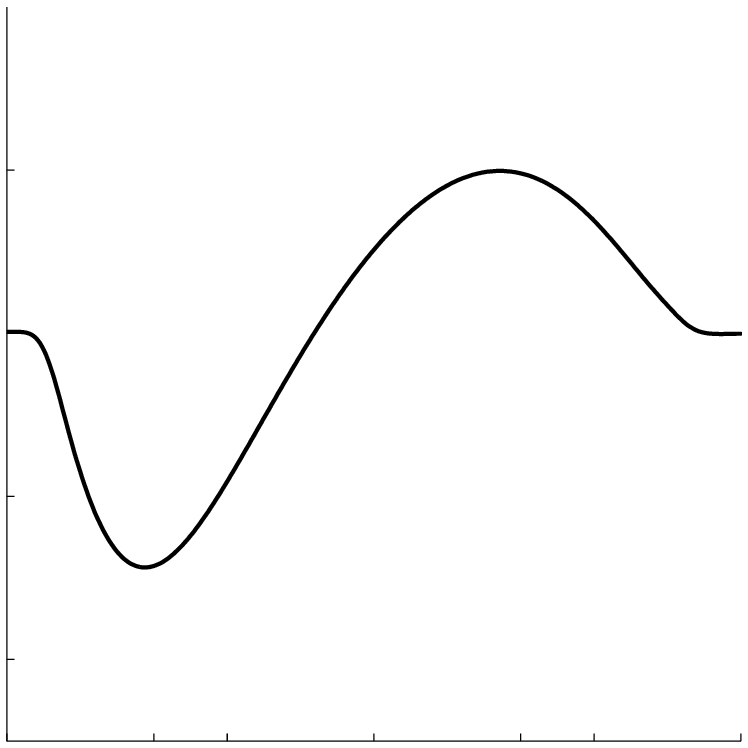}}
\put(90,115){\includegraphics[width=3.2cm]{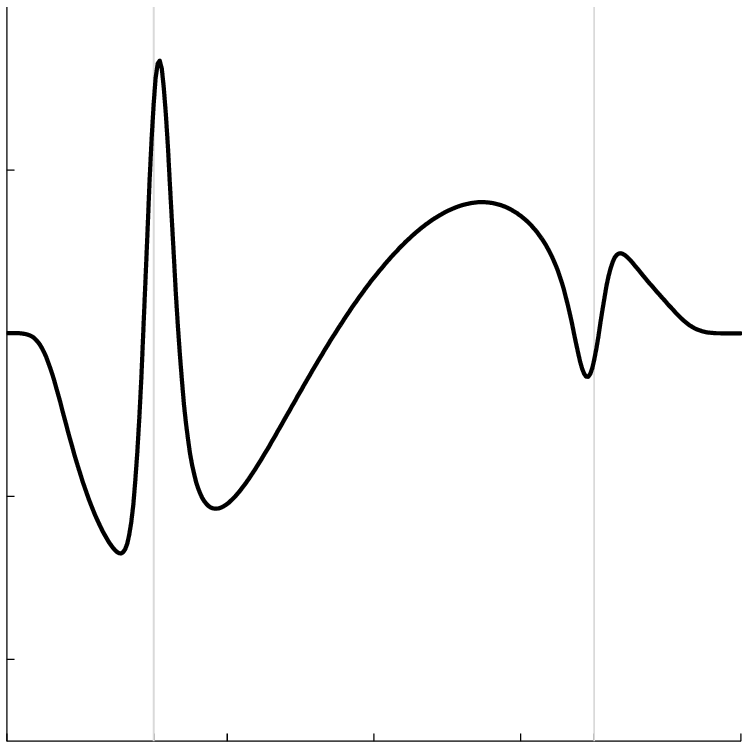}}
\put(215,115){\includegraphics[width=3.2cm]{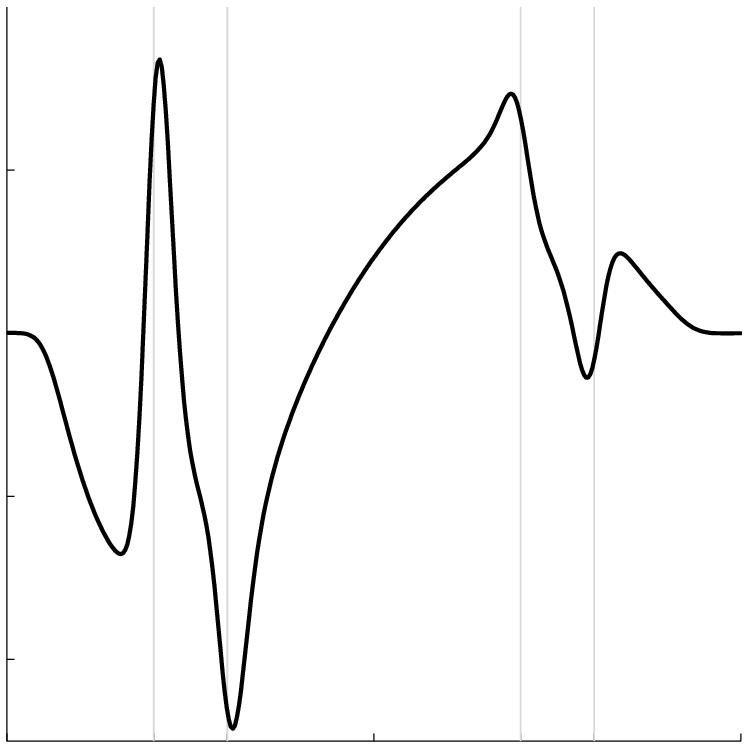}}
\put(-35,10){\includegraphics[width=3.2cm]{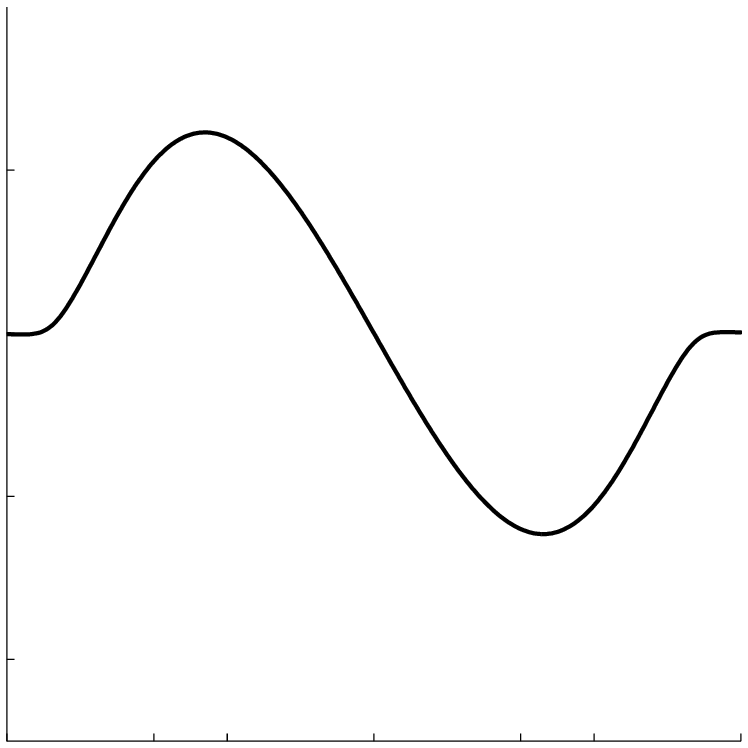}}
\put(90,10){\includegraphics[width=3.2cm]{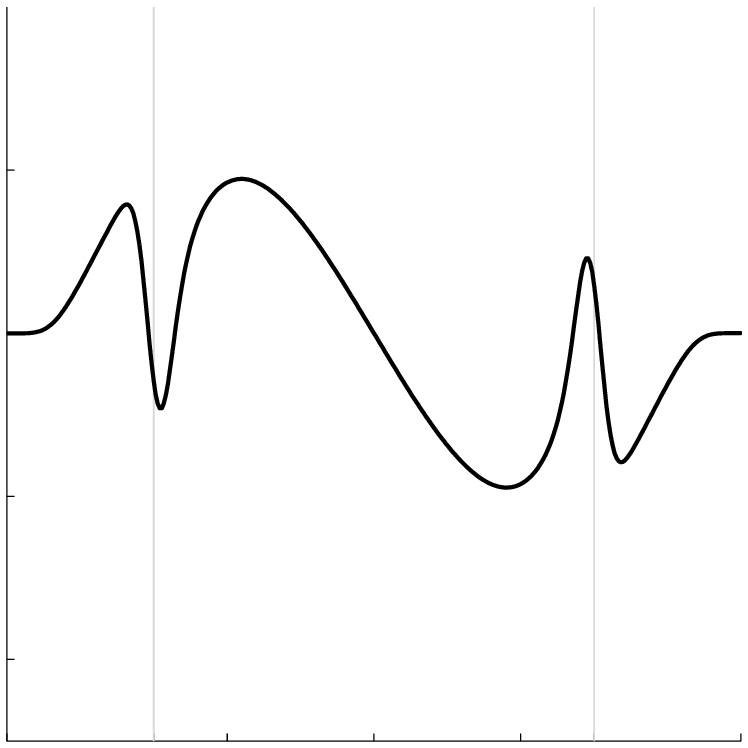}}
\put(215,10){\includegraphics[width=3.2cm]{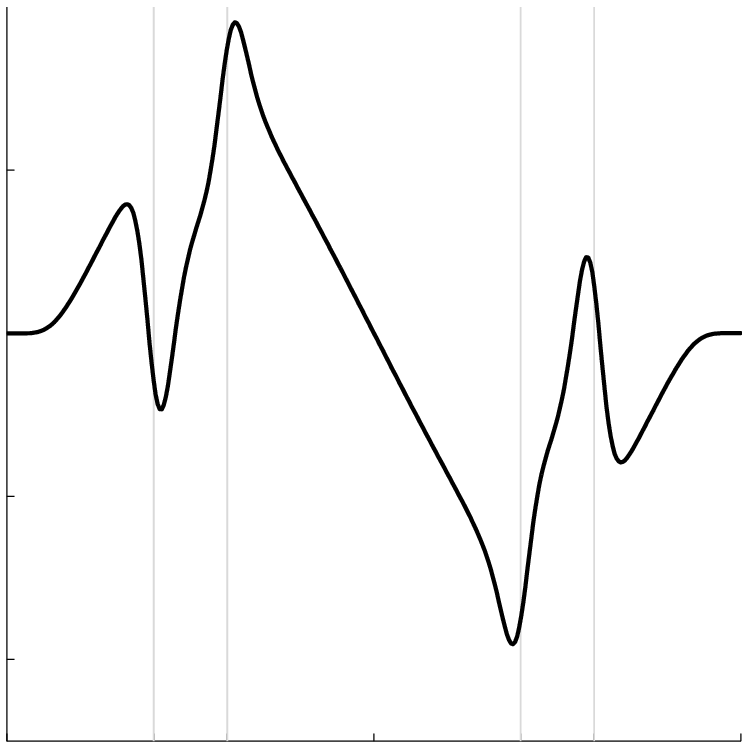}}
\put(-40,0){-1}
\put(8.5,0){0}
\put(53,0){1}
\put(85,0){-1}
\put(133.2,0){0}
\put(178,0){1}
\put(210,0){-1}
\put(258.2,0){0}
\put(303,0){1}
\put(-120,505){$\sigma(t,0)$} 
\put(-46,510){2} 
\put(-46,440){1} 
\put(-120,400){$\widehat{\omega}^+(1,t/2,1)$} 
\put(-120,295){$\widehat{\omega}^{a,+}(t/2,1)$} 
\put(-120,190){$\widehat{\omega}^+(1,t/2,1)$} 
\put(-120,175){$-\widehat{\omega}^-(1,t/2,1)$} 
\put(-120,85){$\widehat{\omega}^{a,+}(t/2,1)$} 
\put(-120,70){$-\widehat{\omega}^{a,-}(t/2,1)$} 
\end{picture}
\caption{\label{fig:rotsym}Top: profiles  of three radial conductivities along the real axis. 
The middle conductivity has a jump along the circle $|z|=0.6$; the one on the right has jumps on both  $|z|=0.4$ 
and  $|z|=0.6$. Rows 2 and 3: the functions $\widehat{\omega}^+(1,t/2,1)$ 
and $\widehat{\omega}^{a,+}(t/2,1)$, resp.; note the artifacts at $t=0$. 
Rows 3 and 4: as described in Sec. \ref{sec evenodd}, the artifacts are eliminated by subtracting  $\widehat\omega^-,\, \widehat\omega^{a,-}$, resp.
}
\end{figure}

We study three rotationally symmetric conductivities defined in the unit disc. The first conductivity $\sigma_1$ is smooth. The second conductivity is defined as
$$
  \sigma_2 = \sigma_1-0.3\chi_{D(0,0.6)}
$$
and therefore has a jump of magnitude $0.3$ along the circle centered at the origin and radius $0.6$. The third rotationally symmetric conductivity is defined as 
$$
  \sigma_3 = \sigma_2+0.3\chi_{D(0,0.4)}
$$
and has jumps of magnitude $0.3$ along the circles centered at the origin and radii $0.4$ and $0.6$. 

In Fig. \ref{fig:rotsym} we show the profiles of $\widehat \omega(1,t,1)$ for three rotationally symmetric conductivity phantoms. 
The first phantom is smooth, while the second and the third have jumps. 
The position and the sign of each jump is clearly visible from the CGO-Fourier data. 
Note that the artifact singularity appearing around $0$ in the second and  third rows 
vanishes when considering the difference of the two CGO functions, in the fourth and fifth rows. This confirms the parity symmetry analyzed in Sec. \ref{sec evenodd}.

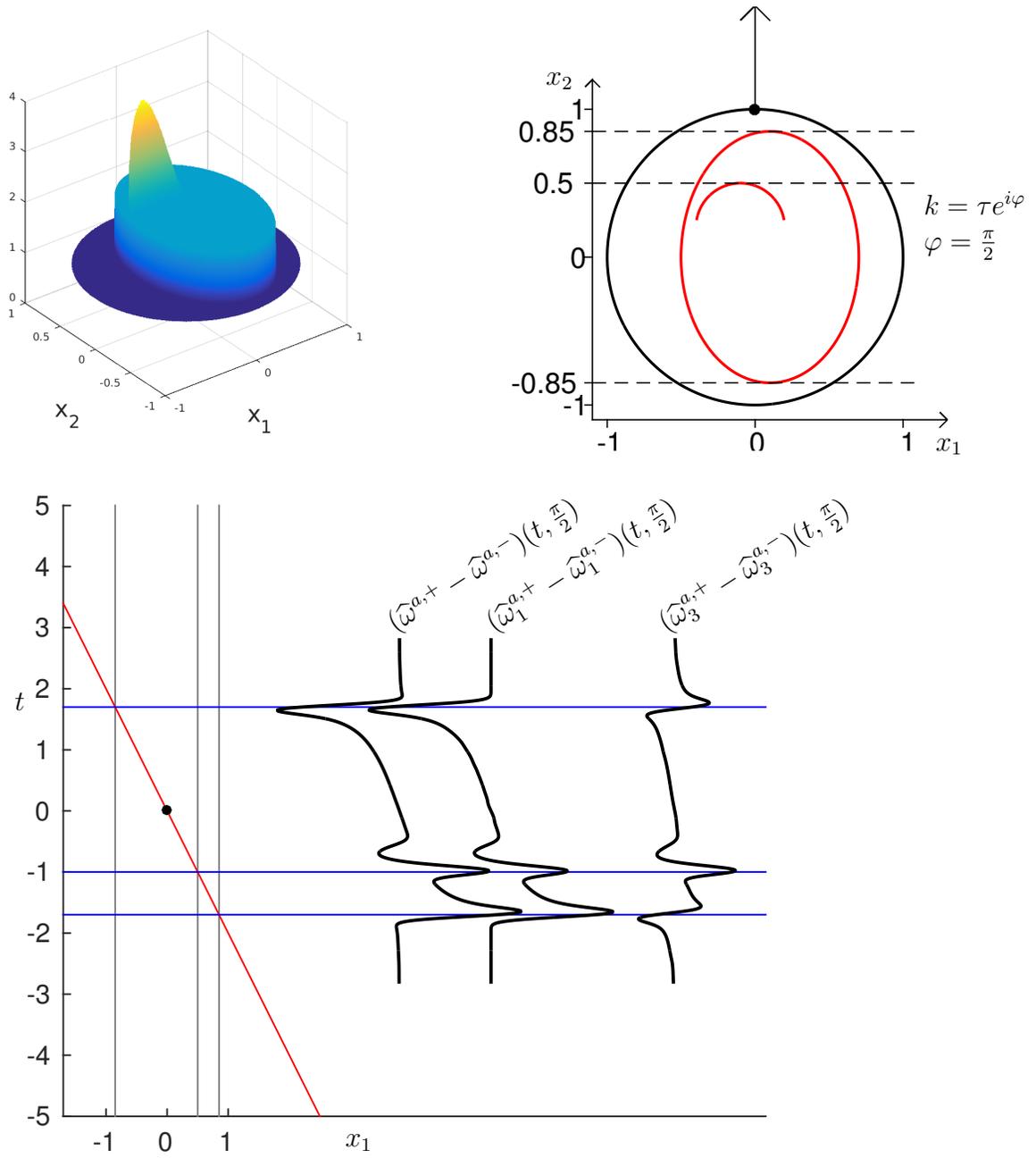
\begin{figure}
\begin{picture}(300,480)
\put(-35,0){\includegraphics[width=11cm]{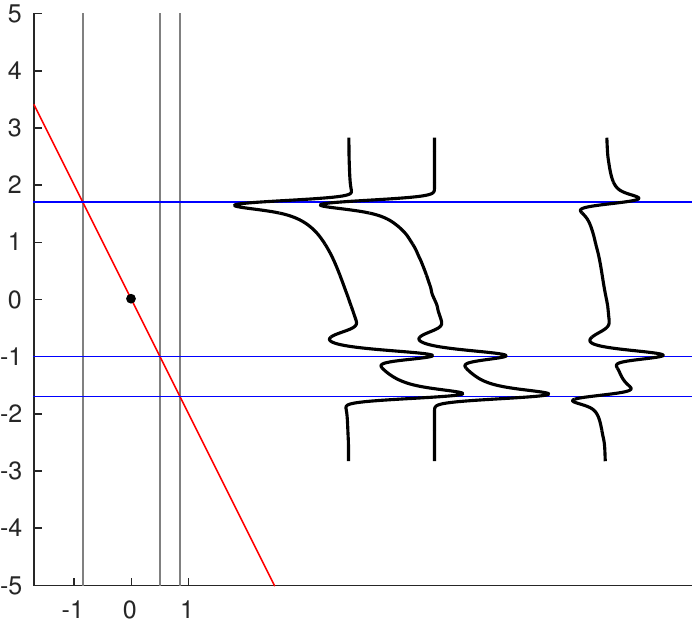}}
\put(116,223){\rotatebox{33}{$(\widehat{\omega}^{a,+}-\widehat{\omega}^{a,-})(t,\frac{\pi}{2})$}}
\put(157,223){\rotatebox{33}{$(\widehat{\omega}^{a,+}_1-\widehat{\omega}^{a,-}_1)(t,\frac{\pi}{2})$}}
\put(230,223){\rotatebox{33}{$(\widehat{\omega}^{a,+}_3-\widehat{\omega}^{a,-}_3)(t,\frac{\pi}{2})$}}
\put(100,5){$x_1$}
\put(-40,190){$t$}

\put(-100,300){\includegraphics[width=9cm]{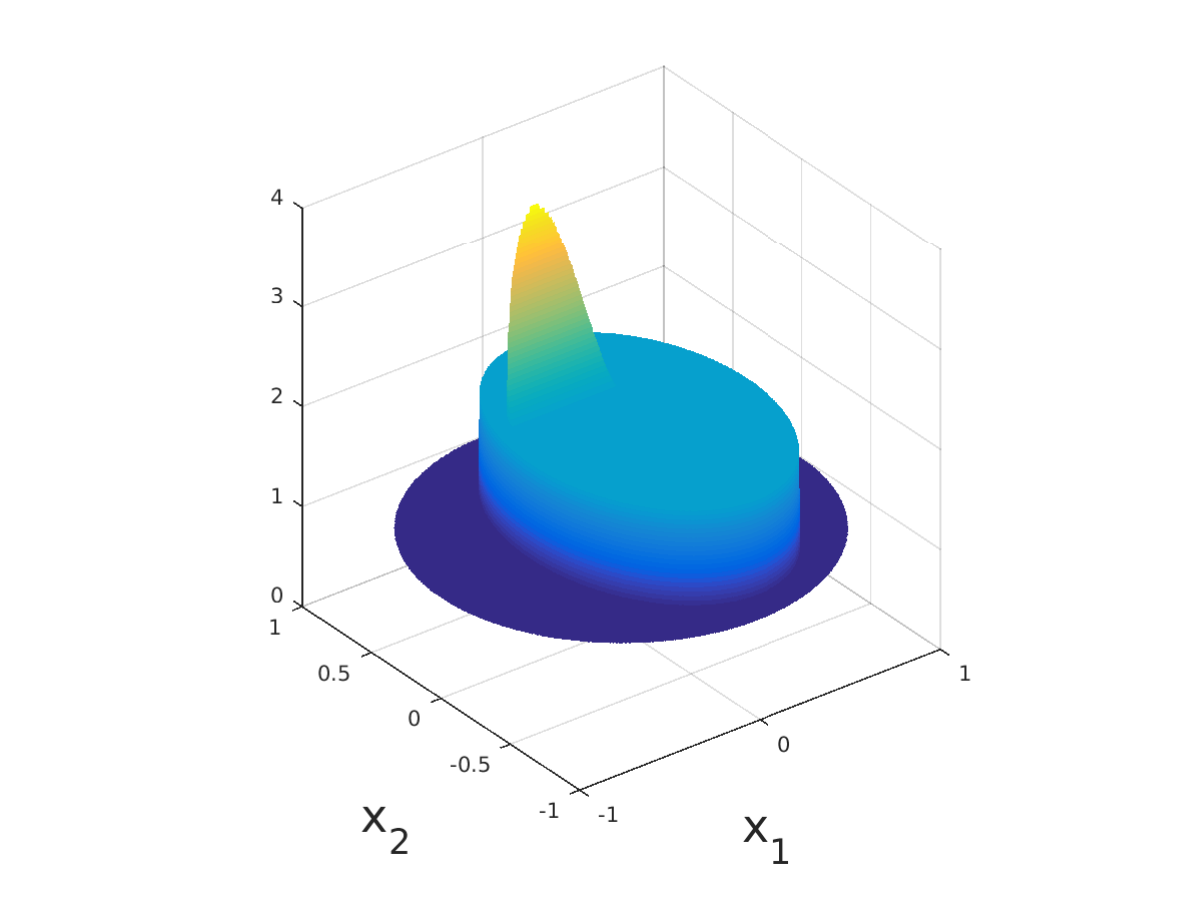}}

\put(170,300){\includegraphics[width=6.5cm]{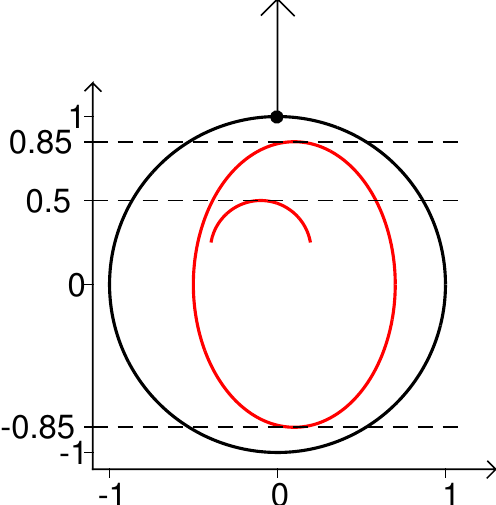}}
\put(185,455){$x_2$}
\put(350,300){$x_1$}
\put(345,400){\textcolor{black}{$k=\tau e^{i\varphi}$}}
\put(345,385){\textcolor{black}{$\varphi=\frac \pi 2$}}

\end{picture}
\caption{\label{fig:HMEi}Diagram showing the propagation of singularities for the HME phantom with zero background.  The virtual direction is $k = i$.}
\end{figure}

\begin{figure}
\begin{picture}(300,480)
\put(-35,0){\includegraphics[width=11cm]{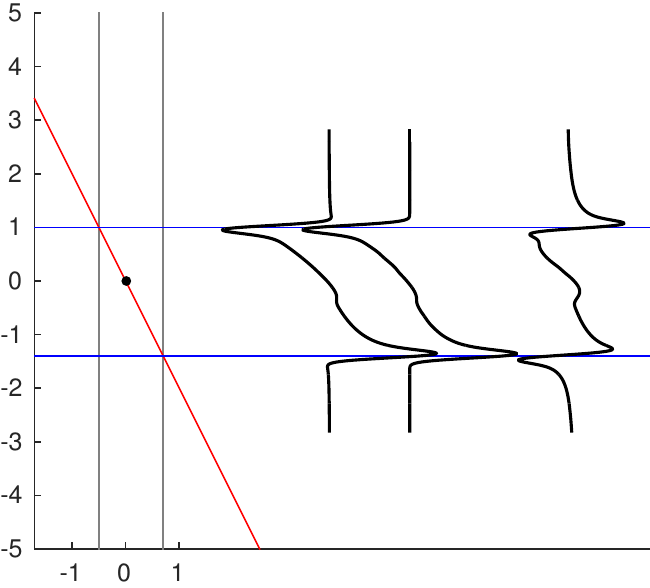}}
\put(116,223){\rotatebox{33}{$(\widehat{\omega}^{a,+}-\widehat{\omega}^{a,-})(t,0)$}}
\put(157,223){\rotatebox{33}{$(\widehat{\omega}^{a,+}_1-\widehat{\omega}^{a,-}_1)(t,0)$}}
\put(230,223){\rotatebox{33}{$(\widehat{\omega}^{a,+}_3-\widehat{\omega}^{a,-}_3)(t,0)$}}
\put(100,5){$x_1$}
\put(-40,190){$t$}

\put(-100,300){\includegraphics[width=9cm]{sigmaHEM_mesh.eps}}

\put(170,300){\includegraphics[width=6.5cm]{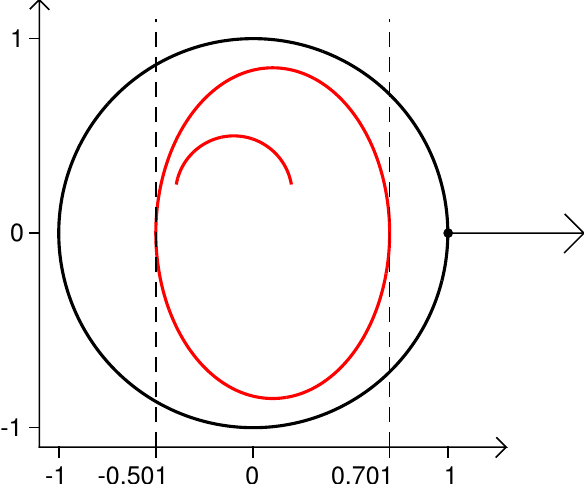}}
\put(165,455){$x_2$}
\put(330,300){$x_1$}
\put(345,360){\textcolor{black}{$k=\tau e^{i\varphi}$}}
\put(345,345){\textcolor{black}{$\varphi=0$}}

\end{picture}

\caption{\label{fig:HME1}Diagram showing the propagation of singularities for the HME phantom with zero background.  The virtual direction is $k = 1$.}

\end{figure}

\begin{figure}
\begin{picture}(300,200)
\put(-30,0){\includegraphics[width=6cm]{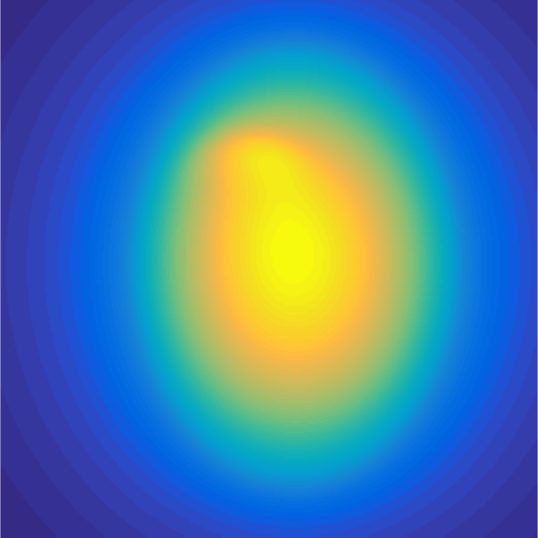}}
\put(170,0){\includegraphics[width=6cm]{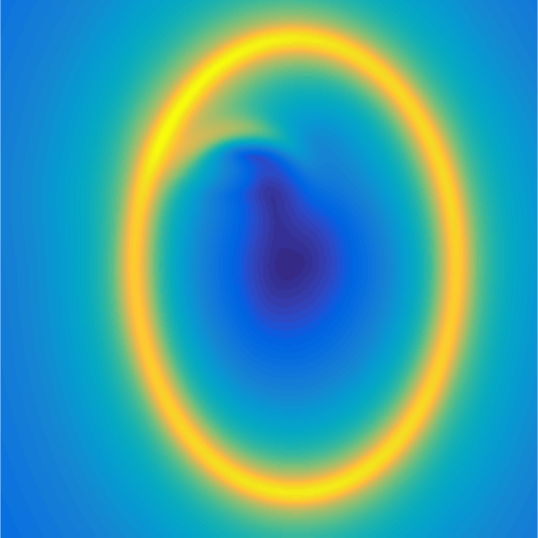}}
\end{picture}
\caption{\label{fig:HME_2D1}Reconstructions from the averaged full series $\widetilde \omega^a_{\mbox{\tiny diff}}$. 
Left: exact inversion formula. Right: $\Lambda$-tomography like reconstruction.}
\end{figure}

\begin{figure}
\begin{picture}(300,250)
\put(-60,10){\includegraphics[width=7cm]{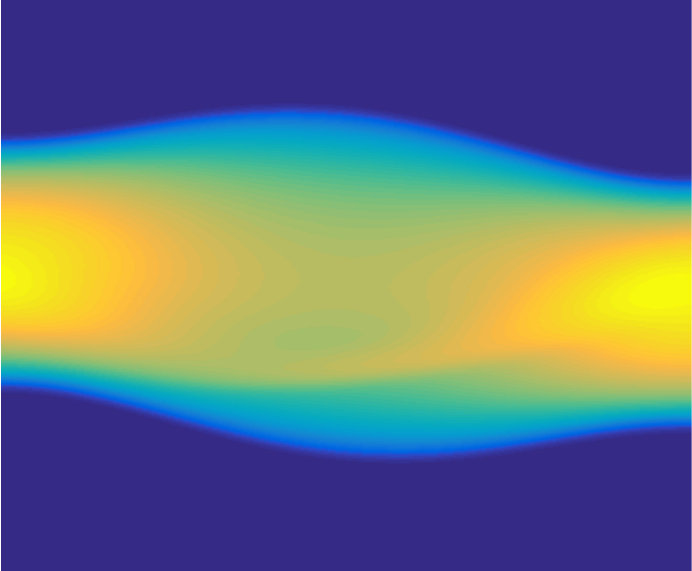}}
\put(170,10){\includegraphics[width=7cm]{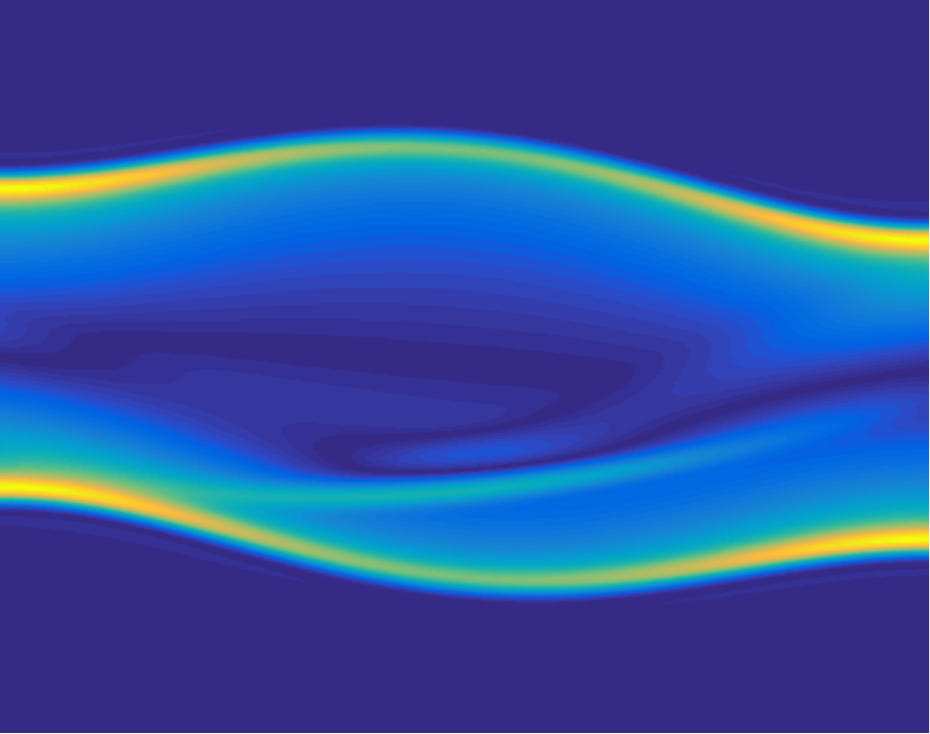}}
\put(0,0){$\varphi$}
\put(-50,0){$0$}
\put(120,0){$\pi$}
\put(-70,87){$0$}
\put(-70,62){$t$}
\put(-70,140){$2$}
\put(-80,35){$-2$}

\put(180,0){$0$}
\put(230,0){$\varphi$}
\put(350,0){$\pi$}
\put(160,87){$0$}
\put(160,62){$t$}
\put(160,140){$2$}
\put(150,35){$-2$}

\end{picture}
\caption{\label{fig:HME_2D2}Sinograms of the averaged full series $\widetilde \omega^a_{\mbox{\tiny diff}}$. Left: exact reconstruction sinogram. Right: $\Lambda$-tomography like sinogram.}
\end{figure}

\subsection{Half-moon and ellipse (HME)}

This  conductivity phantom has a large elliptical inclusion and another smaller inclusion inside the ellipse. 
The smaller inclusion has a jump along an almost complete half-circle. 
This example was chosen because it has two nontrivial features in the wave front set for the horizontal direction and three for the vertical. 
Figs. \ref{fig:HME1} and \ref{fig:HMEi} show, in particular, \textit{ladder} diagrams of the propagation of singularities 
in the directions $k = i$ and $k=1$, resp.: the zeroth and second order terms of the Neumann series for $\widehat \omega^a$ are displayed, 
as well as the full series of the difference of the CGOs: 
$\widehat\omega_{\mbox{\tiny diff}}=(\widehat\omega^+-\widehat\omega^-)/2$.

Fig. \ref{fig:HME_2D1} shows 2D reconstructions obtained using the new algorithm, with the two different inversion formulas. 
In Fig. \ref{fig:HME_2D2} we show the values of $\widehat \omega^a_{\mbox{\tiny diff}}(t,e^{i\varphi})$ for $t \in [-3,3]$ and $\varphi \in [0,\pi]$. 
We borrow the term \textit{sinogram} to describe these plots, because of the clear similarity with the sinograms of X-ray tomography.

\section{Conclusion}\label{sec conclusions}

We introduce a novel and robust method for recovering  singularities of conductivities from electric boundary measurements.  
It is unique in its capability of recovering inclusions within inclusions in an unknown inhomogeneous background conductivity. 
This method provides a new connection between diffuse tomography (EIT) and classical parallel-beam X-ray tomography and filtered back-projection algorithms.

Full analysis of  the  higher order terms $\tw_n$ remains an open problem. We point out that there is a strong  formal similarity 
between the multilinear forms $\mu\to\tw_n$ and multilinear operators considered by Brown \cite{Brown2001}, 
Nie and Brown \cite{Nie2011} and Perry and Christ \cite{Perry2011}. 
Indeed, any Born-type expansion naturally leads to expressions of this general form, 
with the places of the Cauchy and Beurling kernels for $\omega_n$ or $\tw_n$  here being taken by the appropriate Green's functions. 
However, an important feature here is that the singular coefficient in a Beltrami equation occurs in the top order term, 
rather than as a potential as in the works cited above. For the application needed in this setting, useful  function space estimates  
do not seem to follow from existing results, which would require higher regularity of $\mu$, and this is an interesting topic for future investigation.

\end{document}